\documentclass[11pt,english]{article}
\usepackage[margin=1in]{geometry}

\usepackage[utf8]{inputenc}
\usepackage[T1]{fontenc} 
\usepackage{amsmath,amssymb,amsthm}
\usepackage{mathtools}
\usepackage{hyperref}
\hypersetup{pdfusetitle,colorlinks,linkcolor={red!50!black},citecolor={blue!50!black},urlcolor={blue!80!black}}
\usepackage{cleveref}
\usepackage[inline]{enumitem}
\usepackage{fourier}
\usepackage[textsize=small]{todonotes}
\usepackage{xcolor}
\usepackage{thm-restate}
\setuptodonotes{inline}

\usepackage{tikz}
\usetikzlibrary{graphs}
\usetikzlibrary{decorations.markings,calc, shapes, arrows.meta, positioning,shapes.multipart}
\tikzstyle{every picture}=[line width=.65pt,minimum size=3pt,every label/.append style={font=\small},label distance=-2pt]
\tikzstyle{every node}=[font=\small]
\tikzset{>=stealth}
\tikzstyle{vtx}=[circle,draw,thick,inner sep = 4pt]

\setlist[enumerate,1]{label=(\roman*)}
\setlist[enumerate,2]{label=(\alph*)}
\setlist{nosep}

\graphicspath{{../}}

\newtheorem{theorem}{Theorem}
\newtheorem{proposition}[theorem]{Proposition}
\newtheorem{lemma}[theorem]{Lemma}
\newtheorem{corollary}[theorem]{Corollary}
\newtheorem{definition}[theorem]{Definition}
\newtheorem{claim}[theorem]{Claim}
\newtheorem{remark}[theorem]{Remark}

\newtheorem{problem}{Problem}
\newtheorem{observation}[theorem]{Observation}

\newcommand{\circuits}{\mathcal{C}}
\newcommand{\eps}{\varepsilon}
\newcommand{\dimension}{\operatorname{dim}}
\newcommand{\identity}{\mathbf{I}}
\newcommand{\R}{\mathbb{R}}
\newcommand{\rank}{\operatorname{rank}}
\newcommand{\smallmat}[1]{\left[ \begin{smallmatrix} #1 \end{smallmatrix} \right]}
\newcommand{\supp}{\operatorname{supp}}
\newcommand{\Z}{\mathbb{Z}}
\newcommand{\N}{\mathbb{N}}
\newcommand{\Q}{\mathbb{Q}}
\newcommand{\zero}{\mathbf{0}}
\newcommand{\one}{\mathbf{1}}
\renewcommand{\t}{^\intercal}
\newcommand{\poly}{\mathrm{poly}}

\newcommand{\contract}{/}

\DeclareMathOperator{\adj}{adj}
\DeclareMathOperator{\colsp}{colsp}
\DeclareMathOperator{\mathspan}{span}
\DeclareMathOperator{\OPT}{OPT}

\newcommand{\surf}{\Sigma}
\newcommand{\eg}{\operatorname{eg}}
\newcommand{\bd}{\operatorname{bd}}
\newcommand{\inter}{\operatorname{int}}
\newcommand{\adh}[1]{X_{#1}}

\title{Integer programs with nearly totally unimodular matrices:\\ the cographic case\thanks{An extended abstract of this work, without proofs, was published in the proceedings of the {\em 2025 ACM-SIAM Symposium on Discrete Algorithms (SODA)} \cite{aprile2025integer}.}}

\author{Manuel Aprile\thanks{University of Padova} \and Samuel Fiorini\thanks{Université libre de Bruxelles} \and Gwena\"el Joret\footnotemark[3]  \and Stefan Kober\footnotemark[3] \and Michał T.\ Seweryn\footnotemark[3]  \and Stefan Weltge\thanks{Technical University of Munich} \and Yelena Yuditsky\footnotemark[3] \thanks{University of Leeds}}

\date{\today}

\begin{document}

\maketitle

\begin{abstract}
It is a notorious open question whether integer programs (IPs) with an integer coefficient matrix $M$ whose subdeterminants are all bounded by a constant $\Delta$ in absolute value can be solved in polynomial time. 
We answer this question in the affirmative if we further require that, by removing a constant number of rows and columns from $M$, one obtains a submatrix $A$ that is the transpose of a network matrix.

Our approach focuses on the case where $A$ arises from $M$ after removing $k$ rows only, where $k$ is a constant.
We achieve our result in two main steps, the first related to the theory of IPs and the second related to graph minor theory.

First, we derive a strong proximity result for the case where $A$ is a general totally unimodular matrix: Given an optimal solution of the linear programming relaxation, an optimal solution to the IP can be obtained by finding a constant number of augmentations by circuits of $\begin{bmatrix} A& \identity \end{bmatrix}$.

Second, for the case where $A$ is transpose of a network matrix, we reformulate the problem as a maximum constrained integer potential problem on a graph $G$. We observe that if $G$ is $2$-connected, then it has no rooted $K_{2,t}$-minor for $t = \Omega(k \Delta)$. We leverage this to obtain a tree-decomposition of $G$ into highly structured graphs for which we can solve the problem locally. This allows us to solve the global problem via dynamic programming.
\end{abstract}

\section{Introduction}\label{sec:intro}

As for most computational problems that are NP-hard, the mere input size of an integer program (IP) does not seem to capture its difficulty.
Instead, several works have identified additional parameters that significantly influence the complexity of solving IPs.
These include the number of integer variables (Lenstra~\cite{Lenstra}, see also~\cite{Kannan,Dadush,reis2023subspace}), the number of inequalities for IPs in inequality form (Lenstra~\cite{Lenstra}), the number of equations for IPs in equality form (Papadimitriou~\cite{Papadimitriou}, see also~\cite{EisenbrandWeismantel}), and features capturing the block structure of coefficient matrices (see for instance \cite{cslovjecsek_2021a,cslovjecsek_2021b,eisenbrand_2022,brianski_2024,cslovjecsek_2024}).

Another parameter that has received particular interest is the \emph{largest subdeterminant} of the coefficient matrix, which already appears in several works concerning the complexity of linear programs (LPs) and the geometry of their underlying polyhedra~\cite{Tardos86,DF94,BDEHN14,EV17} as well as proximity results relating optimal solutions of IPs and their LP relaxations~\cite{cook_1986,paat2020distances,celaya_2022}.
Consider an IP of the form
\begin{equation}
    \label{eqIPgeneral}
    \max \, \left\{ p\t x : Mx \le b, \, x \in \Z^n \right\}, \tag{IP}
\end{equation}
where $M$ is an integer matrix that is \emph{totally $\Delta$-modular}, i.e., the determinants of square submatrices of $M$ are all in $\{-\Delta,\dots,\Delta\}$.
It is a basic fact that if $M$ is totally unimodular ($\Delta = 1$), then the optimum value of~\eqref{eqIPgeneral} is equal to the optimum value of its LP relaxation, implying that~\eqref{eqIPgeneral} can be solved in polynomial time.
In a seminal paper by Artmann, Weismantel \& Zenklusen~\cite{artmann_2017}, it is shown that~\eqref{eqIPgeneral} is still polynomial-time solvable if $\Delta = 2$, leading to the conjecture that this may hold for every constant $\Delta$. Below, we refer to this conjecture as the \emph{totally $\Delta$-modular IP conjecture}.
Recently, Fiorini, Joret, Yuditsky \& Weltge~\cite{FJWY25} answered this in the affirmative under the further restriction that $M$ has only two nonzeros per row or column.
In particular, they showed that in this setting, \eqref{eqIPgeneral} can be reduced to the stable set problem in graphs with bounded odd cycle packing number~\cite{BFMR14,ocp1,ocpgenus}.

We remark that the algorithm of \cite{artmann_2017} even applies to full column rank matrices $M \in \Z^{m \times n}$ for which only the $(n \times n)$-subdeterminants are required to be in $\{-\Delta,\dotsc,\Delta\}$ for $\Delta = 2$.
Further results supporting the conjecture have been recently obtained by N\"agele, Santiago \& Zenklusen~\cite{naegele_2022} and N\"agele, N\"obel, Santiago \& Zenklusen~\cite{naegele_2023} who considered the special case where all size-$(n \times n)$ subdeterminants are in $\{-\Delta,0,\Delta\}$.
Interestingly, the results of \cite{artmann_2017,naegele_2022,naegele_2023} are crucially centered around a reformulation of~\eqref{eqIPgeneral} where $M$ becomes \emph{totally unimodular up to removing a constant number of rows}, where the additional constraints capture a constant number of congruency constraints.

\subsection{Main contribution}

In an effort to provide more evidence for the totally $\Delta$-modular IP conjecture, we initiate the study of IPs in which $M$ is totally $\Delta$-modular and \emph{nearly totally unimodular}, i.e., $M$ becomes totally unimodular after removing a constant number of rows and columns.
Note that without requirements on the subdeterminants, IPs with nearly totally unimodular coefficient matrices are still NP-hard.
A famous example is the densest $k$-subgraph problem~\cite{BCCFV10,Khot06,Manurangsi17}, which can be seen as an IP defined by a totally unimodular matrix with two extra rows (modeling a single equality constraint). A closely related example is the partially ordered knapsack problem~\cite{KS02}, which is also strongly NP-hard. Another famous example is the exact matching (or red-blue matching) problem~\cite{maalouly2022exact, mulmuley_1987}, for which no deterministic polynomial-time algorithm is known (yet). We remark that our notion of nearly totally unimodular matrices differs from the one introduced in \cite{gijswijt2005integer}.

As we will see, it still does not seem to be an easy task to prove the totally $\Delta$-modular IP conjecture for the full class of nearly totally unimodular matrices.
However, in our main result, we resolve the conjecture for a main building block of such matrices:

\begin{theorem} \label{thm:main_IP}
    There is a strongly polynomial-time algorithm for solving the integer program~\eqref{eqIPgeneral} for the case where $M$ is totally $\Delta$-modular for some constant $\Delta$ and becomes the transpose of a network matrix after removing a constant number of rows and columns.
\end{theorem}

Recall that, to any given (weakly) connected directed graph $G$ and spanning tree $T$ of $G$, one associates the \emph{network matrix} $A\in\{0,\pm1\}^{E(T)\times E(G) - E(T)}$ such that $A_{e,(v,w)}$ is equal to $1$ if $e$ is traversed in the forward direction on the unique $v$-$w$-path in $T$, is equal to $-1$ if it is traversed in the backward direction, and is equal to $0$ otherwise.

By Seymour's celebrated decomposition theorem for regular matroids \cite{seymour_1980}, network matrices and their transposes are the main building blocks of totally unimodular matrices.
While we have not been able to prove an analogous result for network matrices (instead of their transposes) yet, we believe that studying the totally $\Delta$-modular IP conjecture for the above case yields several new insights that are relevant for resolving the general conjecture.
It is known that $\Delta$-modular matrices give rise to a proper minor-closed class of $\mathbb{F}_p$-representable matroids for $p > \Delta$, see for instance \cite{geelen_2021}.
A key result in the matroid minors project \cite{geelen_2015} states that, in each such family of matroids, every matroid whose vertical connectivity is sufficiently high is a low-rank perturbation\footnote{If $M$ and $N$ are matroids generated by matrices $A$ and $B$ respectively, where $\rank(A - B) \le t$, then $M$ is said to be a rank-$(\le t)$ perturbation of $N$.} of a frame\footnote{In this context, a matroid is said to be \emph{frame} if it can be represented over $\mathbb{F}_p$ by a matrix with at most two nonzeros per column, and \emph{coframe} if it is the dual of a frame matroid.} or coframe matroid.
The structure of such matroids is highly governed by a graph, and so it is to be expected that questions about graphs become a central aspect of resolving the general totally $\Delta$-modular IP conjecture.
In fact, our approach for proving \Cref{thm:main_IP} strongly builds upon results from the theory of graph minors that are complementary to those that have been used in the work in~\cite{FJWY25}.
Thus, we believe that our work significantly enhances our understanding of which questions about graphs are relevant to resolve the general totally $\Delta$-modular IP conjecture.

Moreover, note that nearly totally unimodular matrices \emph{are} low-rank perturbations of well-behaved matrices, an aspect that has not been studied yet in the context of the conjecture.
Finally, our insights also lead to a new proximity result (see \Cref{thmProximity}) in integer programming that generalizes a result by Eisenbrand \& Weismantel~\cite{EisenbrandWeismantel}, yielding another motivation for studying the class of nearly totally unimodular matrices.

\subsection{Approach}

As indicated above, we achieve our result in two main steps, one related to the theory of integer programming and one related to graph minor theory.
More specifically, the first step is concerned with bounds on distances between optimal solutions of IPs and their LP relaxations.
A classic result of this type was established by Cook, Gerards, Schrijver, \& Tardos~\cite{cook_1986} who showed that if $M$ is totally $\Delta$-modular, \eqref{eqIPgeneral} is feasible, and $x^*$ is an optimal solution of the LP relaxation, then there exists an optimal solution $z^*$ of~\eqref{eqIPgeneral} with $\|x^* - z^*\|_\infty \le n \Delta$.
It is still open whether this bound can be replaced with a function in $\Delta$ only, see~\cite{celaya_2022}.

A convenient consequence of this result is that, given $x^*$, one can efficiently enumerate the possible values of $z^*$ for a constant number of variables.
In particular, if we are given the integer program~\eqref{eqIPgeneral} with a totally $\Delta$-modular coefficient matrix $M$ that becomes totally unimodular after removing $k$ rows and $\ell$ columns, we may simply guess the values of the variables corresponding to the $\ell$ columns and solve a smaller IP for each guess.

Thus, we may assume that $M$ (is totally $\Delta$-modular and) is of the form $M = \smallmat{A \\ W}$, where $A$ is totally unimodular and $W$ is an integer matrix with only $k$ rows.
For this class of IPs, we derive a considerably strengthened proximity result:
Given an optimal solution $x^*$ of the corresponding LP relaxation, there is an optimal solution $z^*$ of~\eqref{eqIPgeneral} where $\|x^* - z^*\|_\infty \le f(k, \Delta)$ for some function $f$ that depends only on $k$ and $\Delta$ (see \Cref{thmProximity}), again provided that~\eqref{eqIPgeneral} is feasible.
In fact, by bringing~\eqref{eqIPgeneral} into equality form, we show that $x^*$ can be rounded to a closeby integer point from which $z^*$ can be reached by adding a number of conformal \emph{circuits} of $\begin{bmatrix} A & \identity \end{bmatrix}$, where $\identity$ is the identity matrix, that can be bounded in terms of $k$ and $\Delta$ only.
Moreover, we observe that the fact that $M$ is totally $\Delta$-modular implies that every circuit $c$ satisfies $\|Wc\|_\infty \le \Delta$.
(See the next section for definitions and a more precise statement.)

While these findings are valid for all totally unimodular matrices $A$, we will see that they can be crucially exploited for the case where $A$ is the transpose of a network matrix, which we refer to as the \emph{cographic case}.
For these instances, it is convenient to reformulate the original problem~\eqref{eqIPgeneral} as a particular instance of a \emph{maximum constrained integer potential problem}
\begin{equation}
    \label{eqMCIPP}
    \max \left\{ p\t y : \ell(v,w) \le y(v) - y(w) \le u(v,w) \text{ for all } (v,w) \in E(G), \, Wy = d, \, y \in \Z^{V(G)} \right\},
    \tag{MCIPP}
\end{equation}
where $G$ is a directed graph, $p \in \Z^{V(G)}$, $\ell, u \in \Z^{E(G)}$, $W \in \Z^{[k] \times V(G)}$ and $d \in \Z^{k}$, and moreover each row of $W$ sums up to zero and so does $p^\intercal$ (we refer to Section \ref{sec:MCIPP} for details on this reformulation). 
Notice that the first constraints are still given by a totally unimodular matrix, and hence we may regard $Wy = d$ as extra (or complicating) constraints.
With this reformulation, and assuming that $G$ is (weakly) connected (see Section \ref{sec:outline_conn} for further details) the circuits of $\begin{bmatrix} A & \identity \end{bmatrix}$ turn into vertex subsets $S \subseteq V(G)$ with the property that both induced subgraphs $G[S]$ and $G[\overline{S}]$ are (weakly) connected, where $\overline{S} := V(G) \setminus S$.
We call such sets \emph{doubly connected sets} or \emph{docsets}.
Using this notion, our previous findings translate to two strong properties of the instances of~\eqref{eqMCIPP} we have to solve:
First, every feasible instance has an optimal solution that is the sum of at most $f(k,\Delta)$ incidence vectors $\chi^S \in \{0,1\}^{V(G)}$, where $S$ is a docset.
Second, every docset $S$ satisfies $\|W \chi^S\|_{\infty} \le \Delta$.

Referring to the vertices whose variables appear with a nonzero coefficient in at least one of the extra constraints as \emph{roots}, the second property above implies that roots cannot be arbitrarily distributed in the input graph.
Roughly speaking, by carefully exploiting the structure of the instance, we will be able to guess $y(v)$ for each root $v$. Note that once all of these variables are fixed, the resulting IP becomes easy since its constraint matrix is totally unimodular. In fact, the guessing cannot be done for the whole graph at once and we will have to do it locally, and then combine the local optimal solutions via dynamic programming.

\begin{figure}[ht!]
\centering
\includegraphics[width=.4\textwidth]{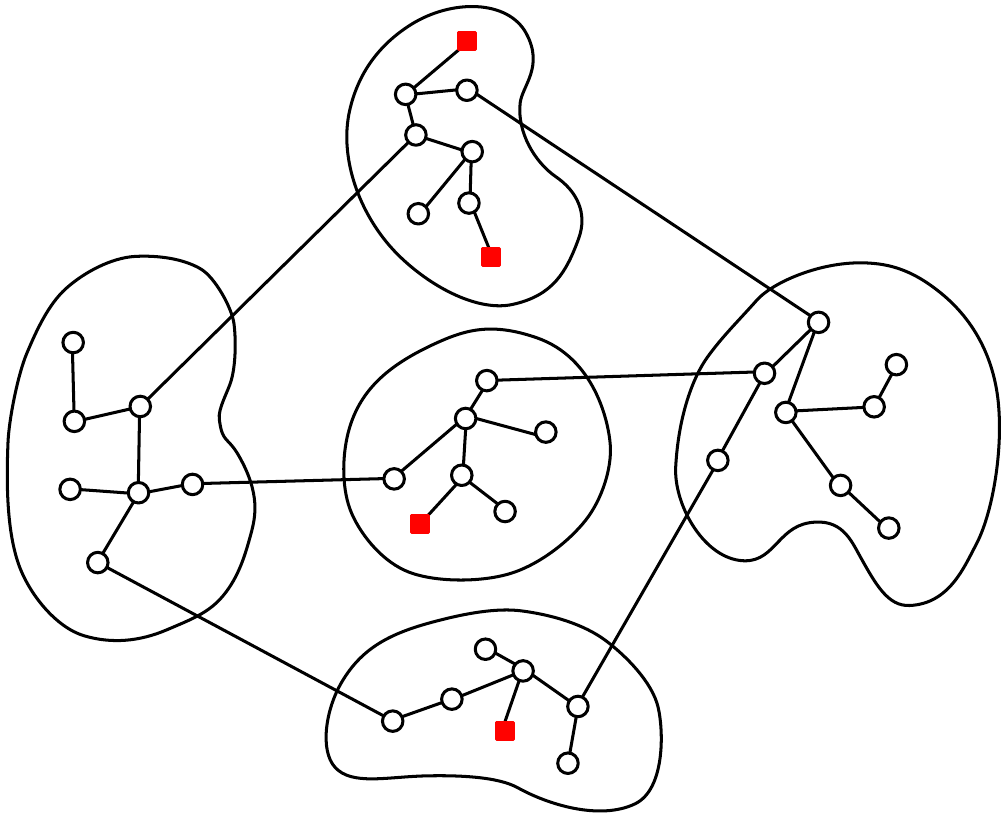}
\caption{Subgraph containing a rooted $K_{2,3}$-minor. Roots are indicated with the red squares. Contracting all the edges in each of the five \emph{branch sets} produces a properly rooted $K_{2,3}$.} \label{fig:K23}
\end{figure}

Our structural insights are based on the observation that our input graphs do not contain a \emph{rooted $K_{2,t}$-minor}, where $t = 4k\Delta+1$, provided that they are $2$-connected. Here, the minors of a rooted graph (graph with a distinguished set of vertices called roots) are defined similarly as for usual graphs, with two differences: whenever some edge $e$ is contracted we declare the resulting vertex as a root if and only if at least one of its ends is a root, and we have the possibility to remove a vertex from the set of roots. A rooted $K_{2,t}$ is said to be \emph{properly rooted} if each one of the $t$ vertices in the ``large'' side is a root. For the sake of simplicity, we say that a rooted graph contains a \emph{rooted $K_{2,t}$-minor} if it has a rooted minor that is a properly rooted $K_{2,t}$, see \Cref{fig:K23}.

Our main structural result is a decomposition theorem for rooted graphs without a rooted $K_{2,t}$-minor, see \Cref{thm:P1-P6_informal} below. It relies partly on several works about the structure of graphs excluding a minor, extending the original results of Robertson \& Seymour within the graph minors project, more specifically on works by Diestel, Kawarabayashi, M\"uller \& Wollan~\cite{DiestelKMW12}, Kawarabayashi, Thomas \& Wollan~\cite{KTW20}, and Thilikos \& Wiederrecht \cite{TW_FOCS_2022,TW_arxiv_2022}. Furthermore, we use results of B\"ohme \& Mohar~\cite{boehme_2002} and B\"ohme, Kawarabayashi, Maharry \& Mohar~\cite{boehme_2008} (see also~\cite{FKSSY}) to control the distribution of the roots in surface-embedded rooted graphs without rooted $K_{2,t}$-minors. 

Our decomposition theorem is formulated in terms of a tree-decomposition of graph $G$. Recall that a \emph{tree-decomposition} is a pair $(T, \mathcal{B})$ where $T$ is a rooted tree (tree with a unique \emph{root node}) and $\mathcal{B} = \{B_u : u \in V(T)\}$ is a collection of vertex subsets of $G$, called \emph{bags}, such that for every vertex $v$ of $G$ the set of bags containing $v$ induces a non-empty subtree of $T$, and for every edge $e$ of $G$ there is a bag that contains both ends of $e$. We define the \emph{weak torso} of a bag $B_u$ as the graph obtained from the induced subgraph $G[B_u]$ by adding a clique on $B_u \cap B_{u'}$ for each node $u' \in V(T)$ that is a child of $u$. Having stated these definitions, we are ready to state the decomposition theorem. See \Cref{fig:P1-P6} for an illustration.


\begin{theorem}[simplified version of \Cref{P1-P6}] \label{thm:P1-P6_informal}
    For every $t \in \Z_{\ge 1}$ there exists a constant $\ell = \ell(t)$ such that every $3$-connected rooted graph $G$ without a rooted $K_{2,t}$-minor admits a tree-decomposition $(T,\mathcal{B})$, where $\mathcal{B} = \{B_u : u \in V(T)\}$, with the following properties:
    \begin{enumerate}
        \item the bags $B_u$ and $B_{u'}$ of two adjacent nodes $u, u' \in V(T)$ have at most $\ell$ vertices in common, and
        \label{item:i_thm:P1-P6_informal}
        \item for every node $u \in V(T)$, all but at most \(\ell\) children $u' \in V(T)$ of \(u\) are leaves with the property that the roots in the bag $B_{u'}$ are contained in $B_u$, and
        \label{item:ii_thm:P1-P6_informal}
        \item every node $u \in V(T)$ satisfies one of the following:
        \label{item:iii_thm:P1-P6_informal}
        \begin{enumerate}
            \item bag $B_u$ has at most $\ell$ vertices, or
            \label{item:iii_a_thm:P1-P6_informal}
            \item $u$ is a leaf and $B_{u}$ has at most $\ell$ roots, all contained in the bag of the parent of $u$, or
            \label{item:iii_b_thm:P1-P6_informal}
            \item after removing at most $\ell$ vertices of $B_u$, the weak torso of $B_u$ becomes a $3$-connected rooted graph that does not contain a rooted $K_{2,t}$-minor and has an embedding in a surface of Euler genus at most $\ell$ such that every face is bounded by a cycle, and all its roots can be covered by at most $\ell$ facial cycles.
            \label{item:iii_c_thm:P1-P6_informal}
        \end{enumerate}
    \end{enumerate}
\end{theorem}

\begin{figure}[ht!]
\centering
\includegraphics[width=0.6\textwidth]{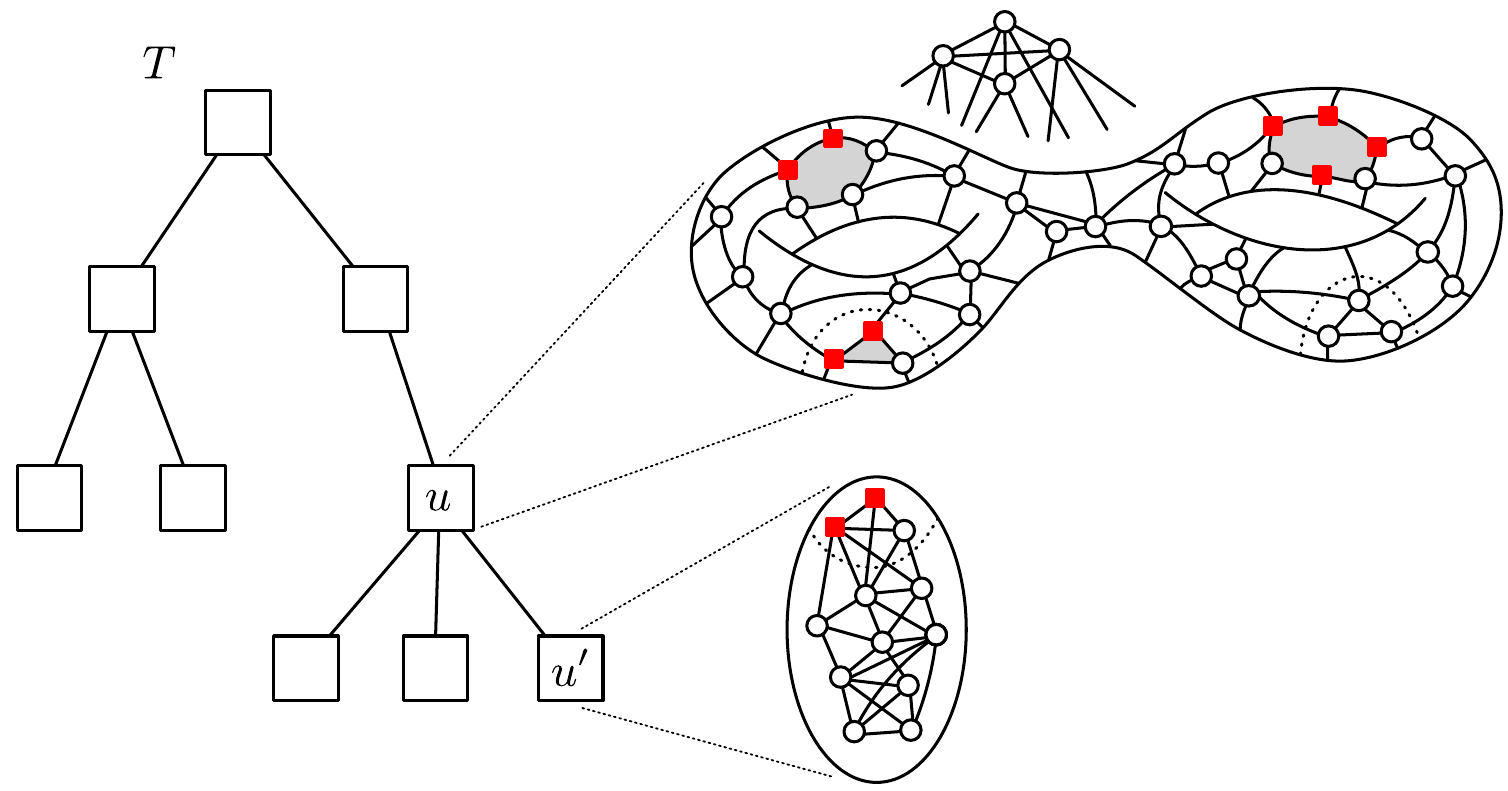}
\caption{Illustrating the decomposition of \Cref{thm:P1-P6_informal}. The decomposition tree $T$ is shown on the left. The graphs corresponding to the weak torsos of the two bags $B_u$ and $B_{u'}$ are shown on the right. The top one satisfies \ref{item:iii_thm:P1-P6_informal}.\ref{item:iii_c_thm:P1-P6_informal}. The vertices above the embedded graph indicate a set of at most $\ell$ vertices, removal of which leaves us with a graph embedded on a surface of bounded genus. The roots are indicated with red squares and are covered by $3$ faces which are drawn in grey. The vertices within the dotted regions in the top and the bottom graphs are the vertices in the intersection of those two graphs. The bottom graph satisfies \ref{item:iii_thm:P1-P6_informal}.\ref{item:iii_b_thm:P1-P6_informal}.  } \label{fig:P1-P6}
\end{figure}

As we show, there is a polynomial-time algorithm that finds the tree-decomposition of \Cref{thm:P1-P6_informal} together with a polynomial-size collection $\mathcal{X}_u$ for each node $u \in V(T)$, containing all the possible intersections of a docset of $G$ with the roots contained in bag $B_u$. This yields an efficient dynamic programming algorithm to solve the instances of~\eqref{eqMCIPP} we are interested in, which proves \Cref{thm:main_IP}.

In the next section, we give a more detailed overview of our approach together with important intermediate results, as well as an outline of the paper. The proofs of these results, as well as the missing definitions, can be found in the subsequent sections.
\section{Overview}\label{sec:outline}

Suppose we are given an instance of~\eqref{eqIPgeneral}, where $M$ is totally $\Delta$-modular and becomes totally unimodular (TU) after removing a bounded number of rows and columns. 
(Here and in the following discussion, `bounded' means `bounded by a constant'.) 
Using Tardos' algorithm~\cite{Tardos86}, we can solve its LP relaxation in strongly polynomial time.
We may assume that the latter has an optimal solution.\footnote{If the LP relaxation is infeasible, so is~\eqref{eqIPgeneral}. If the LP is unbounded, we may first replace the objective by the all-zero vector and run our algorithm.
If we find a feasible integer solution, we can conclude that the original instance is also unbounded. Otherwise, the original problem is also infeasible.}
Thus, by using the proximity result of~\cite{cook_1986}, we may eliminate a bounded number of variables such that remaining problem is of the form $\max \{ p^\intercal x : Ax \le b, \, Wx \le d, \, x \in \Z^n \}$, where $A \in \Z^{m \times n}$ is TU, $\smallmat{A \\ W}$ is totally $\Delta$-modular, and $W \in \Z^{k \times n} $ has a bounded number of rows.
Now that our initial reduction is established, we proceed by explaining how we exploit the fact that $\smallmat{A \\ W}$ is totally $\Delta$-modular.
To this end, we consider the circuits of matrices, which are defined as follows.

\subsection{Circuits}

For a subset $U \subseteq \R^n$, we say that a vector $x \in U \setminus \{\zero\}$ is a \emph{support-minimal} vector of $U$ if there is no other vector in $U \setminus \{\zero\}$ whose support is strictly contained in the support of $x$.
A \emph{circuit} of a rational matrix $B \in \Q^{m \times n}$ is a support-minimal (nonzero) vector of the kernel of $B$ that is integer and reduced (that is, the gcd of its components is $1$).
The collection of all the circuits of $B$ is denoted by $\circuits(B)$.
If $B \in \{0,\pm 1\}^{m \times n}$ is a TU matrix, then it is a basic fact that every circuit of $B$ is in $\{0,\pm 1\}^n$, see for instance Onn~\cite[\S 2.3.4]{onn_2010}.
We observe the following.\footnote{Recall that all proofs, as well as some definitions, are deferred to later sections.}

\begin{lemma}
    \label{lemCharTotalDelta}
    Let $A \in \{0,\pm 1\}^{m \times n}$ be totally unimodular, $w \in \Z^n$, and $\Delta \ge 1$.
    Then $\smallmat{A \\ w^\intercal}$ is totally $\Delta$-modular if and only if every circuit $\smallmat{x \\ y}$ of $\begin{bmatrix} A & \identity \end{bmatrix}$ satisfies $|w^\intercal x| \le \Delta$.
\end{lemma}

This implies that, for our instance, every circuit $\smallmat{x \\ y}$ of the matrix $\begin{bmatrix}A & \identity\end{bmatrix}$ satisfies $\|Wx\|_\infty \le \Delta$.

\subsection{Equality Form}\label{overviewEquality}

For the next considerations, it is convenient to bring our integer program into equality form.
By adding slack variables $y$ and $z$ (note that this transformation preserves total unimodularity of $A$ and total $\Delta$-modularity of $\smallmat{A\\W}$), we obtain the equivalent problem
\[
    \max \left\{ p^\intercal x : Ax + y = b, \, Wx + z = d, \, x \in \Z^n, \, y \in \Z^m_{\ge 0}, \, z \in \Z^k_{\ge 0} \right\}.
\]
By again using the proximity result of~\cite{cook_1986}, we obtain finite lower and upper variable bounds that are satisfied by at least one optimal solution of this integer program.
Thus, by setting $\bar{p} = \begin{bmatrix} p^\intercal & \zero & \zero \end{bmatrix}^\intercal$, $\bar{A} = \begin{bmatrix} A & \identity & \zero \end{bmatrix}$, $\bar{W} = \begin{bmatrix} W & \zero & \identity \end{bmatrix}$, and $\bar{n} = n + m + k$, our integer program turns into
\[
    \max \left\{ \bar{p}^\intercal \bar{x} : \bar{A} \bar{x} = b, \, \bar{W} \bar{x} = d, \, \bar{x} \in \Z^{\bar{n}}, \, \ell \le \bar{x} \le u \right\}
\]
for some $\ell, u \in \Z^{\bar{n}}$.
Now, consider any circuit $\bar{x}$ of $\bar{A}$.
Note that it is of the form $\bar{x} = \begin{bmatrix} x^\intercal & y^\intercal & \zero \end{bmatrix}^\intercal$, where $\smallmat{x \\ y}$ is a circuit of $\begin{bmatrix} A & \identity \end{bmatrix}$.
By \Cref{lemCharTotalDelta}, we obtain $\|\bar{W} \bar{x}\|_\infty = \|Wx\|_\infty \le \Delta$.
Thus, we want to solve, i.e., find the optimal value and a solution attaining it, a problem of the following form.

\begin{problem}
    \label{problemEquality}
    Let $k,\Delta \in \Z_{\ge 1}$ be fixed.
    Given $p,\ell,u \in \Z^n$, $A \in \Z^{m \times n}$, $b \in \Z^m$, $W \in \Z^{k \times n}$, and $d \in \Z^k$ such that $A$ is totally unimodular and $\max \{ \|Wx\|_\infty : x \in \circuits(A) \} \le \Delta$, solve
    \begin{equation}
        \label{eqIPequality}
        \max \left\{ p^\intercal x : Ax = b, \, Wx = d, \, \ell \le x \le u, \, x \in \Z^n \right\}. \tag{$\mathrm{IP}_1$}
    \end{equation}
\end{problem}

Moreover, in the above transformation, note that if $A$ is the transpose of a network matrix, so is $\bar{A}$. 

Thus, in order to prove \Cref{thm:main_IP}, we need to give a strongly polynomial-time algorithm for instances of \Cref{problemEquality} where $A$ is the transpose of a network matrix.
Before we focus on such matrices, we derive a strengthened proximity result that holds for all totally unimodular matrices $A$.

\subsection{Proximity}

For \Cref{problemEquality}, we will prove a strengthened proximity result on distances between optimal solutions of \eqref{eqIPequality} and its LP relaxation
\begin{equation}
    \label{eqLPequality}
    \max \left\{ p^\intercal x : Ax = b, \, Wx = d, \, \ell \le x \le u \right\}\,. \tag{$\mathrm{LP}_1$}
\end{equation}
Our first result is the following.

\begin{theorem}
    \label{thmProximity}
    If $x^*$ is an optimal solution of \eqref{eqLPequality} and \eqref{eqIPequality} is feasible, then there is an optimal solution $z^*$ of \eqref{eqIPequality} with $\|x^* - z^*\|_\infty \le f_{\ref{thmProximity}}(k,\Delta)$.
    Here, $f_{\ref{thmProximity}}(k,\Delta)$ is a function\footnote{All functions of $k$ and $\Delta$ defined in this paper are in fact computable.} of $k$ and $\Delta$ only.
\end{theorem}

The proof of \Cref{thmProximity} is based on the following result, which we prove following the strategy by Eisenbrand and Weismantel~\cite{EisenbrandWeismantel}.
We remark that our proof of \Cref{propReductionIP} extends the proximity result in~\cite{EisenbrandWeismantel} to general totally unimodular matrices with a constant number of additional equalities, and recovers the proximity bound for their setting, that is, an equality system with a constant number of constraints, and entries bounded by a constant in absolute value.

Here, we say that vectors $c_1,\dots,c_t \in \R^n$ are \emph{conformal} if, in every coordinate, they have the same sign, that is, $c_1,\dots,c_t$ are contained in the same orthant (see~\cite[\S 3.1]{onn_2010}).

\begin{proposition}
    \label{propReductionIP}
    Given an instance of \Cref{problemEquality}, in strongly polynomial time we can compute an integer point $z \in \Z^n$ satisfying $Az = b$, $\ell \le z \le u$ such that there is an optimal solution to \eqref{eqIPequality} of the form $z + \sum_{j=1}^t c_j$ where $c_1,\dots,c_t \in \circuits(A)$ are conformal and $t \le f_{\ref{propReductionIP}}(k, \Delta)$, or conclude that \eqref{eqIPequality} is infeasible.
\end{proposition}

Notice that after applying \Cref{propReductionIP}, we can perform a translation of the solution space mapping the integer point $z$ to the origin. This allows us to reduce to the special case of \Cref{problemEquality} where $b = \zero$, and look for an optimal solution $x \in \Z^n$ that is a conformal sum of at most $f_{\ref{propReductionIP}}(k, \Delta)$ circuits of $A$.

\subsection{Increasing the connectivity of the TU matrix}\label{sec:outline_conn}

In order to deal with \Cref{problemEquality}, it is convenient to ensure some connectivity properties of $A$.
Suppose that, possibly after permuting the columns of $A$, we can write $A = \begin{bmatrix} A_1 & A_2 \end{bmatrix}$ in such a way that the column space of $A_1 \in \{0,\pm 1\}^{m \times n_1}$ and the column space of $A_2 \in \{0,\pm 1\}^{m \times n_2}$ intersect in a linear subspace of dimension at most $q-1$, for some $q \in \Z_{\ge 1}$. Provided that $\min\{n_1,n_2\} \ge q$, the corresponding partition of the columns of $A$ is a \emph{separation of order $q$}, or shortly a \emph{$q$-separation}. We define the \emph{connectivity} of $A$ as the minimum $q$ such that $A$ has a $q$-separation, and call $A$ \emph{$q$-connected} if its connectivity is at least $q$. 

In case the connectivity of $A$ is low, it is natural to split \eqref{eqIPequality} into two IPs with bounded interaction. We show how to do this for separations of order $q \le 2$, and give a (strongly) polynomial time, black-box reduction of instances of \Cref{problemEquality} to instances such that $A$ is \emph{almost $3$-connected}, in the sense that $A$ becomes $3$-connected when we delete from $A$ every column that is a nonzero multiple of another column. We point out that special care should be taken because of the extra constraints $Wx = d$.

Connectivity is a fundamental concept for studying TU matrices and the corresponding matroids, which are known as \emph{regular} matroids. By Seymour's decomposition theorem~\cite{seymour_1980}, if $A$ is $4$-connected\footnote{ If $A$ is 3-connected but not 4-connected, then it has a $3$-separation: we do not know if our reduction can be extended to 3-separations, however we do not need this in order to prove \Cref{thm:main_IP}. See Section \ref{sec:discussion} for further discussions.} then there exists a network matrix $D \in \{0,\pm 1\}^{r \times (n-r)}$ such that system $Ax = \zero$ is equivalent either to $\begin{bmatrix} D &\identity_r \end{bmatrix} x = \zero$ or to $\begin{bmatrix} \identity_{n-r} &-D^\intercal \end{bmatrix} x = \zero$. The two cases are dual of each other. The first case corresponds to the graphic case (i.e. the linear matroid of $A$ is a graphic matroid) and we do not treat it in this paper, see Section \ref{sec:discussion}.  
In order to prove \Cref{thm:main_IP}, we now focus on the latter case, the cographic case.

\subsection{Reducing to the maximum constrained integer potential problem}\label{sec:MCIPP}

Assume that $Ax = \zero$ is equivalent to $\begin{bmatrix} \identity_{n-r} &-D^\intercal \end{bmatrix} x = \zero$ where $D \in \{0,\pm 1\}^{(n-r) \times r}$ is a network matrix, and $b = \zero$. We perform a change of variables that transforms any such instance of \Cref{problemEquality} into an instance of the maximum constrained integer potential problem. 

Let $G$ denote a (weakly) connected directed graph such that $D = B^{-1} N$ where $M = \begin{bmatrix} N &B \end{bmatrix}$ is the vertex-edge incidence matrix of $G$ with one row deleted, and $B$ is a basis of $M$. Writing $x \in \R^n$ as $x = \smallmat{x_N\\ x_B}$, we have
$$
A x = \zero \iff x_N - D^\intercal x_B = \zero \iff x_N = D^\intercal x_B \iff x_N = N^\intercal B^{-\intercal} x_B\,.
$$

Now let $y := B^{-\intercal} x_B$. Then $x = \smallmat{N^\intercal y\\ B^\intercal y} = M^\intercal y$ and \eqref{eqIPequality} can be rewritten as
\begin{equation}
\label{eq:IP1-cographic} 
\max \left\{ p^\intercal M^\intercal y : \ell \leq M^\intercal y \leq u,\ W M^\intercal y = d,\ y \in \Z^{r} \right\}\,.
\tag{$\mathrm{IP}_1$-cographic}
\end{equation}
Note that $x = \smallmat{x_N\\ x_B} \in \Z^n$ holds if and only if $y \in \Z^r$, since $B$ is unimodular and $x_N = N^\intercal B^{-\intercal} x_B$. We claim that \eqref{eq:IP1-cographic} is equivalent to \eqref{eqMCIPP}. In order to see this, let $v_0 \in V(G)$ denote the vertex whose row is missing from $M$. Notice that $r = |V(G-v_0)| = |V(G)| - 1$. Now append to $M$ the missing row for $v_0$, and similarly append a row to $y$ for $v_0$. After performing this modification and renaming the profit vector and weight matrix, \eqref{eq:IP1-cographic} transforms precisely into \eqref{eqMCIPP}. (Now that $M$ is the full incidence matrix of $G$, notice that we have $M^\intercal \one = \zero$, hence each row of $p^\intercal M^\intercal$ or $W M^\intercal$ in \eqref{eq:IP1-cographic} sums up to zero.)

Consider the resulting instance of the maximum constrained integer potential problem, see \eqref{eqMCIPP}, where each vertex $v\in V(G)$ has weights $W(i,v)$ for $i\in [k]$. Let $R := \{v \in V(G) : \exists i \in [k] : W(i,v) \neq 0\}$ denote the set of roots of $G$. Below, we denote the resulting rooted graph as $(G,R)$. (Also, we regard $G$ as an undirected graph most of the time, since the edge directions are relevant only when we go back to solving \eqref{eqMCIPP}.)

By applying \Cref{lemCharTotalDelta} and \Cref{propReductionIP}, we infer the two properties of the resulting MCIPP instance mentioned in the previous section: First, provided that the instance is feasible, it has an optimal solution that is the sum of at most $f_{\ref{propReductionIP}}(k, \Delta)$ incidence vectors of docsets. Second, the weight vector $W \chi^S$ of every docset $S$ has all its components in $[-\Delta,\Delta] \cap \Z$. We in fact obtain two further properties. Third, since $A$ is (without loss of generality) almost $3$-connected, $G$ is a $3$-connected graph some of whose edges are subdivided. Fourth, $(G,R)$ does not contain a rooted $K_{2,t}$-minor for $t = \Omega(k \Delta)$. This fourth property, which follows from the second, is the starting point of our structural analysis. The precise result we show is as follows.

\begin{lemma} \label{lem:4kDelta+1}
Let $k, \Delta \in \Z_{\ge 1}$ be fixed, and let $t := 4k\Delta + 1$. Consider an instance of the maximum constrained integer potential problem, see \eqref{eqMCIPP}. If $G$ is $2$-connected and $W \in \Z^{[k] \times V(G)}$ satisfies $||W \chi^S||_\infty \le \Delta$ for all docsets $S$, then $(G,R)$ does not contain a rooted $K_{2,t}$-minor.
\end{lemma}

\subsection{Decomposing graphs forbidding a rooted $K_{2,t}$-minor}

\Cref{lem:4kDelta+1} motivates the structural investigation of rooted graphs $(G,R)$ with no rooted $K_{2,t}$-minor. (We mention in passing that
rooted $K_{2,t}$-minors were studied before in the literature, for instance in connection to the problem of computing the genus of apex graphs, see Mohar~\cite{Mohar01}.) On this front, the main result we prove is \Cref{thm:P1-P6_informal}. We start here with a general discussion of the theorem and its context, follow this with an examination of its algorithmic consequences in \Cref{sec:algo_consequences}, and conclude by outlining its proof in \Cref{sec:proving_decomp}.

We point out that \Cref{thm:P1-P6_informal} assumes $G$ to be $3$-connected. 
This is justified by our results on $1$- and $2$-separations. In case $G$ is a $3$-connected graph some of whose edges are subdivided, we may suppress the degree-$2$ vertices, and apply the theorem to the resulting $3$-connected graph. For the rest of the discussion, we assume that $G$ is $3$-connected. 

There are a number of basic cases for which \Cref{thm:P1-P6_informal} holds with a trivial or almost trivial tree-decomposition $(T,\mathcal{B})$. 

First, if the number of roots of $(G,R)$ is bounded by some constant $\ell$, then \ref{item:i_thm:P1-P6_informal}, \ref{item:ii_thm:P1-P6_informal} and \ref{item:iii_thm:P1-P6_informal}.\ref{item:iii_a_thm:P1-P6_informal} always hold in \Cref{thm:P1-P6_informal} for $T := (\{r,u\},\{ru\})$, $B_r := R$ and $B_u := V(G)$. Hence, we may assume without loss of generality that $(G,R)$ has many roots.


Second, it is easy to see that a $3$-connected graph $G$ with at least two roots always has a rooted $K_{2,2}$-minor, hence the case $t = 2$ of \Cref{thm:P1-P6_informal} is trivial. Robertson and Seymour~\cite{RS90} characterized the $3$-connected rooted graphs $(G,R)$ with no rooted $K_{2,3}$-minor: either $G$ has at most $2$ roots, or $G$ is planar and has all its roots incident to a common face. This settles the case $t = 3$ of \Cref{thm:P1-P6_informal}: we let $\ell := 2$ and $(T,\mathcal{B})$ be the \emph{trivial} tree-decomposition with $T := (\{u\},\varnothing)$ and $B_u := V(G)$. Then \ref{item:i_thm:P1-P6_informal} and \ref{item:ii_thm:P1-P6_informal} trivially hold, and either \ref{item:iii_thm:P1-P6_informal}.\ref{item:iii_a_thm:P1-P6_informal} or \ref{item:iii_thm:P1-P6_informal}.\ref{item:iii_c_thm:P1-P6_informal} holds. Hence we may assume without loss of generality that $t \ge 4$.

Third, assume that $G$ is a planar $3$-connected graph without a rooted $K_{2,t}$-minor, where $t$ is any fixed integer, and consider again \Cref{thm:P1-P6_informal}. We let once more $(T,\mathcal{B})$ be the trivial tree-decomposition, which in particular implies \ref{item:i_thm:P1-P6_informal} and \ref{item:ii_thm:P1-P6_informal}. It turns out that \ref{item:iii_thm:P1-P6_informal}.\ref{item:iii_c_thm:P1-P6_informal} always holds, in virtue of the following result of B\"ohme \& Mohar. 

\begin{theorem}[Theorem 1.2 of \cite{boehme_2002}]\label{thmBoehme}
There is a function $f_{\ref{thmBoehme}} : \Z_{\ge 1} \rightarrow \Z_{\ge 1}$ such that the following holds. Let $G$ be a 3-connected, planar graph and let $R \subseteq V(G)$ be a set of roots. Assume that $(G, R)$ does not have a rooted $K_{2,t}$-minor. Then $G$ has a collection of at most $f_{\ref{thmBoehme}}(t)$ facial cycles such that each root is contained in at least one of these cycles.
\end{theorem}

Fourth, assume that $G$ is a $3$-connected graph embedded in a surface of Euler genus $g \ge 1$, where $g$ is bounded.\footnote{See \Cref{sec:structure} for the missing definitions.} 
It turns out that \Cref{thmBoehme} extends to this case, assuming that the facewidth of $G$ is large enough:

\begin{theorem}\label{BKMM}
There is a function $f_{\ref{BKMM}}: \Z_{\ge 0} \times \Z_{\ge 0} \to \Z_{\ge 2}$,  such that the following holds. Let $(G, R)$ be a $3$-connected rooted graph without a rooted \(K_{2,t}\)-minor, embedded in a surface of Euler genus $g$ with facewidth at least $f_{\ref{BKMM}}(g, t)$. Then $G$ has a collection of at most $f_{\ref{BKMM}}(g, t)$ facial cycles covering all the roots.
\end{theorem}

This result was stated in B\"ohme, Kawarabayashi, Maharry \& Mohar~\cite{boehme_2008} as a remark, without a proof (a new proof can be found in~\cite{FKSSY}). 
\Cref{BKMM} implies that \Cref{thm:P1-P6_informal} holds with a trivial tree-decomposition whenever $G$ can be embedded in a surface of bounded genus, with large facewidth. In case the facewidth of the embedding is not large enough, then several things might happen. For instance, it might be that there exists a small vertex subset $Z$ such that $G - Z$ is $3$-connected and has large facewidth. The vertices in $Z$ are commonly known as \emph{apices}. In this case, \ref{item:iii_thm:P1-P6_informal}.\ref{item:iii_c_thm:P1-P6_informal} of \Cref{thm:P1-P6_informal} still holds. Beyond this case, we have to decompose the graph in a non-trivial way.

Roughly speaking, \Cref{thm:P1-P6_informal} states that every $3$-connected rooted graph forbidding a rooted $K_{2,t}$-minor can be obtained by gluing in a tree-like fashion the graphs that we examined so far: graphs with a bounded number of roots, and graphs with bounded genus whose roots can be covered with a bounded number of face boundaries, plus a bounded number of apices. 

\subsection{Algorithmic consequences of the decomposition theorem} \label{sec:algo_consequences}

We discuss here how \Cref{thm:P1-P6_informal} leads to an efficient algorithm for solving the instances of the maximum constrained integer potential problem resulting from cographic instances of \Cref{problemEquality}.

Consider a node $u \in V(T)$ of the decomposition tree and the corresponding bag $B_u \subseteq V(G)$. We claim that every docset of $G$ intersects the roots in $B_u$ in a polynomial number of ways.
If \ref{item:iii_thm:P1-P6_informal}.\ref{item:iii_a_thm:P1-P6_informal} or \ref{item:iii_thm:P1-P6_informal}.\ref{item:iii_b_thm:P1-P6_informal} holds, there are at most $\ell$ roots in $B_u$. Hence, any docset of $G$ can intersect the roots in $B_u$ in at most $2^\ell$ ways, which is a constant. (Recall that $t = 4k\Delta + 1$ and $\ell = \ell(t)$.)

From now on, assume that \ref{item:iii_thm:P1-P6_informal}.\ref{item:iii_c_thm:P1-P6_informal} holds. Let $G^{\#}[B_u]$ denote the weak torso of $B_u$ and $Z \subseteq V(G^{\#}[B_u])$ denote the set of apices. That is, $|Z| \le \ell$, and $G^{\#}[B_u] - Z$ is a $3$-connected graph without a rooted $K_{2,t}$-minor that has an embedding in a surface of Euler genus at most $\ell$ such that every face is bounded by a cycle, and admits a collection $\mathcal{C}$ of at most $\ell$ facial cycles covering all its roots.

Consider any docset $S$ of $G$, any facial cycle $C$ in $\mathcal{C}$, and the roots covered by $C$. Observe that the roots in $C$ are cyclically ordered. Our strategy to establish the claim is to prove that the intersection of $S$ with the roots in $C$ can be partitioned into a bounded number of blocks of consecutive elements in the cyclic ordering, say at most $\kappa = \kappa(t)$ intervals. Since every root in $B_u$ is either in $Z$ or in one of the cycles in $\mathcal{C}$, this directly yields a $|V(G)|^{O(\kappa \ell)}$ bound on the number of possible intersections of $S$ with the roots in $B_u$.
We elaborate briefly on the reasons why the above strategy is sound. 

First, consider the case where there are no apices, that is $Z = \varnothing$. In $G^{\#}[B_u]$, we add a clique on $B_u \cap B_{u'}$ where $u'$ is the parent of $u$. The resulting graph is the (full) \emph{torso} of $B_u$. Since $B_u \cap B_{u'}$ has at most $\ell$ vertices, the torso of $B_u$ does not contain a rooted $K_{2,t'}$-minor for $t' := 2^{\poly(\ell)} t$. Notice that the intersection of docset $S$ with the torso of $B_u$ is in fact a docset of the latter. We show that if the intersection of $S$ with the roots in some $C \in \mathcal{C}$ cannot be partitioned into a small number of blocks of consecutive elements, we find a $K_{3,s}$-minor in $G^{\#}[B_u]$ for some $s = s(\kappa,\ell)$ that is increasing in $\kappa$ and decreasing in $\ell$. 
If we choose $\kappa$ large enough, then the Euler genus of $K_{3,s}$ (which is linear in $s$) can be made larger than $\ell$. This contradicts the fact that the weak torso $G^{\#}[B_u]$ is embedded in a surface of Euler genus at most $\ell$.


Second, in the presence of apices, the argument gets slightly more involved and in one case we obtain a rooted $K_{2,t'}$-minor in the torso of $B_u$, which is also a contradiction.

Finally, we state a result that formalizes the way in which we use our decomposition theorem algorithmically. It is also stated in terms of tree-decompositions. However, we slightly simplify the tree-decomposition of \Cref{thm:P1-P6_informal} by merging some leaf bags into their parents. This has some advantages, in particular for optimization purposes. 

Let $\ell \in \Z_{\ge 1}$ be fixed. We say that a (rooted) tree-decomposition $(T, \mathcal{B})$ of $G$ is \emph{$\ell$-special} if (i) for each $tt' \in E(T)$ we have $|B_t \cap B_{t'}| \le \ell$ and (ii) every node $t \in V(T)$ has at most $\ell$ children. 

Let $\mathcal{S} = \mathcal{S}(G)$ denote the collection of all docsets in $G$. For $R' \subseteq R$, the \emph{docset profile} of $R'$ is the collection of sets $\mathcal{P}_G(R') := \{R'\cap S : S \in \mathcal{S}\}$. A \emph{docset superprofile} of $R'$ is a collection of subsets of $R'$ that is a superset of the docset profile $\mathcal{P}_G(R')$.

We prove the following result, which is the centerpiece of our algorithm.

\begin{restatable}{theorem}{SpecialDecomposition}\label{SpecialDecomposition}
    There exists a function \(f_{\ref{SpecialDecomposition}} : \Z_{\ge 1} \to \Z_{\ge 1}\) such that 
    for every fixed \(t \in \Z_{\ge 1}\), there is a polynomial-time algorithm that given a rooted \(3\)-connected 
    graph $(G,R)$ having no rooted \(K_{2,t}\)-minor, outputs an \(f_{\ref{SpecialDecomposition}}(t)\)-special 
    tree-decomposition \((T, \mathcal{B})\) of \(G\) and a collection \(\{\mathcal{X}_u : u \in V(T)\}\) where 
    each \(\mathcal{X}_u\) is a docset superprofile of \(R \cap B_u\) in \(G\) of size polynomial in $|V(G)|$.
\end{restatable}

Given \Cref{SpecialDecomposition}, it is easy to design a strongly polynomial-time algorithm to solve the instances of the maximum constrained integer potential problem that originate from cographic instances of \Cref{problemEquality}, see \eqref{eqMCIPP}. Roughly speaking, the idea is to first compute polynomially many optimal local solutions for each bag $B_u$, namely, one for each guess on the variables $y(v)$ where $v \in B_u$ is a root or belongs to an adjacent bag $B_{u'}$. For each fixed bag and guess, an optimal local solution can be found in (strongly) polynomial time by solving a single IP on a TU constraint matrix. Next, we use a dynamic programming approach to find optimal ways to combine the optimal local solutions. In virtue of \Cref{propReductionIP} and \Cref{SpecialDecomposition}, the whole dynamic programming algorithm runs in (strongly) polynomial time.

\subsection{Proving the decomposition theorem} \label{sec:proving_decomp}

The purpose of this section is to describe our proof strategy for \Cref{thm:P1-P6_informal}. Before this, we briefly discuss how to deduce \Cref{SpecialDecomposition} from \Cref{thm:P1-P6_informal}. Let $(G,R)$ be a rooted $3$-connected graph without a rooted $K_{2,t}$-minor, and let $(T, \mathcal{B})$ denote the tree-decomposition of $G$ that we get as the output of \Cref{thm:P1-P6_informal}. Consider any node $u \in V(T)$. We say that a child $u'$ of $u$ is \emph{tame} if $u'$ is a leaf and there is no root in $B_{u'} \setminus B_u$. We turn $(T, \mathcal{B})$ into an $\ell$-special tree-decomposition $(T', \mathcal{B}')$ satisfying the statement of \Cref{SpecialDecomposition}, as follows: For each node $u \in V(T)$, we contract every edge connecting $u$ to a tame child $u'$, adding all the vertices of $B_{u'}$ to $B_u$. After observing that the original tree-decomposition $(T, \mathcal{B})$ can be computed in polynomial time, we prove that the docset profile $\mathcal{P}_G(R \cap B'_u)$ is of polynomial size for each $u \in V(T')$, and that we can furthermore compute a docset superprofile $\mathcal{X}_u \supseteq \mathcal{P}_G(R \cap B'_u)$ for each $u \in V(T')$ in polynomial time.

Now, we turn to the proof of \Cref{thm:P1-P6_informal}. We build the required tree-decomposition recursively. Assume we are sitting at some node $u \in V(T)$. The decomposition process defines for us a corresponding induced subgraph $H$ of $G$. Let $R_H=R\cap V(H)$. We consider two cases for $H$. We say that $H$ is \emph{$k$-interesting} (for an integer $k=k(t)$ chosen large enough for the rest of our arguments) if $H$ contains a set of vertices $X$ such that $|X|=k$ and $X$ is well-connected with respect to the vertices in it and is well-connected to $R_H$.
If $H$ is not $k$-interesting then we call it \emph{$k$-boring}. It is easy to deal with the $k$-boring graphs. In each such graph, we find $3$ subsets $B_0,B_1,B_2\subseteq V(H)$ such that we can set $B_0$ to be a child bag of $u$ in our decomposition and either further decompose $H[B_1]$ and $H[B_2]$ or just decompose one of them and the other becomes a leaf bag.

\begin{figure}[ht!]
\centering
\includegraphics[width=0.5\textwidth]{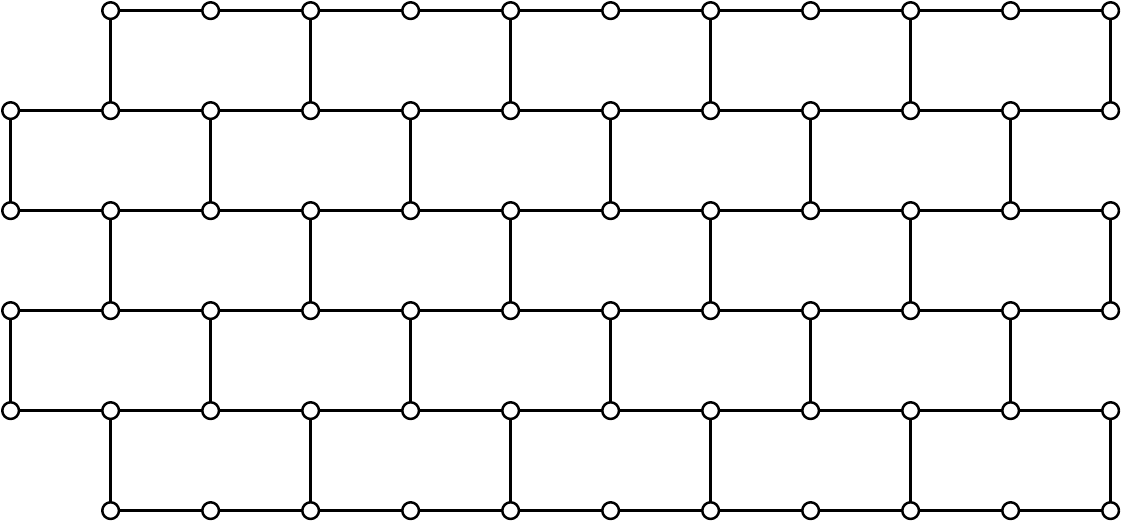}
\caption{An elementary wall of height $6$. Walls of height $h$ are defined as subdivisions of elementary walls of the same height $h$.} \label{fig:wall}
\end{figure} 

In what follows, we focus on $k$-interesting subgraphs $H$. For an integer $h$, if $k$ is large enough as a function of $h$ then it is possible to find a \emph{wall} $W$ of height of $h$ in $H$ (see \Cref{fig:wall}) such that there are $h$ vertex-disjoint paths between the set $X$ and any \emph{transversal set} of $W$. A transversal set of a wall is a maximum collection of vertices of degree $3$ in $W$ such that no two vertices are chosen from the same column or the same row of $W$. Moreover, there are $h$ vertex-disjoint paths between $R_H$ and any transversal set of $W$. We show that because $H$ does not contain a rooted $K_{2,t}$-minor, $H$ cannot contain a $K_{t'}$-minor \emph{grasped} by the wall $W$ for $t'=t'(t)$ and $h\ge t'$. 

Using the results in \cite{KTW20} and \cite{DiestelKMW12} we derive that graphs $H$ have a restricted structure: There is a bounded-size set $Z\subseteq V(H)$ such that $H - Z$ has an embedding on a bounded genus surface $\surf$, with a bounded number of \emph{large vortices} and a possibly unbounded number of \emph{small vortices}. 
(See~\cref{sec:structure of graphs excluding a minor} for the definition of vortices and related notions.) 
Moreover $H - Z$ contains a large wall $W'$ which can be essentially found as a minor of $W$. The wall $W'$ is drawn inside a disk on $\surf$ and is far enough from any large vortex. We handle the large vortices using the results in \cite{TW_arxiv_2022}. We again use the fact that $H$ does not contain a rooted $K_{2,t}$-minor to show that after removal of a bounded number of vertices (which are added to $Z$), each large vortex can be either drawn on $\surf$ or turned into a small vortex. 

The graph we obtain after the removal of the new set $Z$ is not necessarily $3$-connected, but after additional steps, including considering the SPQR tree for $H-Z$, we can ensure $3$-connectivity while keeping large facewidth. This allows us to apply \Cref{BKMM} and find a small collection of facial cycles covering all the roots in $R_H$ that are not in $Z$.

\subsection{Outline}

We begin the main part of our paper in \Cref{sec:proximity} with an in-depth treatment of circuits and proximity, including proofs for \Cref{lemCharTotalDelta}, \Cref{thmProximity} and \Cref{propReductionIP}. We follow this in \Cref{sec:1sum2sum} by designing black-box reductions in the presence of $1$- or $2$-separations in the TU matrix $A$. These first two sections consider general TU matrices. Starting from \Cref{sec:cographic}, we focus on the cographic case. We prove that none of our input graphs contains a rooted $K_{2,t}$-minor, see \Cref{lem:4kDelta+1}, and give a polynomial-time dynamic programming algorithm solving the cographic instances of \Cref{problemEquality}, assuming \Cref{SpecialDecomposition}. This proves \Cref{thm:main_IP}. \Cref{sec:structure} establishes our structural results on rooted graphs forbidding a rooted $K_{2,t}$-minor, including \Cref{thm:P1-P6_informal} and \Cref{SpecialDecomposition}. We close our paper in \Cref{sec:discussion} with remarks on possible approaches to generalize \Cref{thm:main_IP} to all nearly TU matrices $M$, and a discussion of related research questions.

\section{Circuits and Proximity}\label{sec:proximity}

In this section, we collect some basic facts about circuits of totally unimodular matrices, including a proof of \Cref{lemCharTotalDelta}, and provide proofs for our main results on proximity, \Cref{thmProximity} and \Cref{propReductionIP}.

\subsection{Decompositions into Circuits}
Let us start by proving \Cref{lemCharTotalDelta}, which states that $\smallmat{A \\ w^\intercal}$ is totally $\Delta$-modular if and only if every circuit $\smallmat{x \\ y}$ of $\begin{bmatrix} A & \identity \end{bmatrix}$ satisfies $|w^\intercal x| \le \Delta$.
Our proof is based on the following observation.
Here, for a matrix $A \in \R^{m \times n}$ and $I \subseteq [m]$, $J \subseteq [n]$, the submatrix of $A$ whose rows correspond to $I$ and whose columns correspond to $J$ is denoted by $A_{I,J}$.
Similarly, the restriction of a vector $x \in \R^n$ to the entries indexed by $J$ is denoted by $x_J$.

\begin{remark} \label{remCircuitWitness}
Let $A \in \R^{m \times n}$, $w \in \R^n$, and $x \in \ker(A)$ such that $w^\intercal x \neq 0$.
Then $x$ is support-minimal in $\ker(A)$ if and only if $\smallmat{A\\ w^\intercal}$ has an invertible (hence square) submatrix of the form $\smallmat{A_{I,J}\\ w^\intercal_J}$ with $\supp(x) \subseteq J$.
\end{remark}
\begin{proof}
    Let $x$ be support-minimal in $\ker(A)$ and set $J = \supp(x)$.
    Then $\ker(A_{[m],J})$ is the $1$-dimensional subspace spanned by $x_J$.
    Indeed, given any two linearly independent vectors $y$ and $y'$ in $\ker(A_{[m],J})$, there is a nonzero linear combination of $y$ and $y'$ whose support is strictly contained in $J$, contradicting the support-minimality of $x$.
    Thus, $\rank(A_{[m],J}) = |J| - 1$ and hence there is a subset $I \subseteq [m]$ with $\rank(A_{I,J}) = |J| - 1$.
    Consequently, $\smallmat{A_{I,J}\\ w^\intercal_J}$ is invertible.

    Let $I \subseteq [m]$, $J \subseteq [n]$ such that $\smallmat{A_{I,J}\\ w^\intercal_J}$ is invertible and $\supp(x) \subseteq J$.
    Since $\rank(A_{I,J}) = |J| - 1$, we see that $\ker(A_{I,J})$ is the $1$-dimensional subspace spanned by $x_J$.
    Note that for every $y \in \ker(A) \setminus \{\zero\}$ with $\supp(y) \subseteq \supp(x) \subseteq J$ we have $y_J \in \ker(A_{I,J})$ and hence $y$ is a scalar multiple of $x$, in particular, $\supp(y) = \supp(x)$.
\end{proof}

\begin{proof}[Proof of \Cref{lemCharTotalDelta}]
Foremost, notice that $A$ is TU if and only if $\begin{bmatrix} A & \identity \end{bmatrix}$ is TU, and
$\smallmat{A\\ w^\intercal}$ is totally $\Delta$-modular if and only if $\smallmat{A& \identity\\ w^\intercal &\zero^\intercal}$ is totally $\Delta$-modular.

Suppose first that $\smallmat{A \\ w^\intercal}$ is totally $\Delta$-modular and consider any circuit $\smallmat{x\\ y} \in \{0,\pm 1\}^{n + m}$ of $\begin{bmatrix} A & \identity \end{bmatrix}$.
If $w^\intercal x \ne 0$, then by \Cref{remCircuitWitness}, there exist $I \subseteq [m]$, $J \subseteq [n]$, $K \subseteq [m]$ with $\supp(x) \subseteq J$ and $\supp(y) \subseteq K$ such that $B := \smallmat{A_{I,J} & \identity_{I,K}\\ w^\intercal_J & \zero^\intercal}$ is invertible.
Consider the adjugate matrix $\adj(B)$ of $B$ and let $\smallmat{\bar{x}_J\\ \bar{y}_{K}} \in \{0,\pm 1\}^{J \cup K}$ denote the last column of $\adj(B)$. Since $B \adj(B) = \det(B) \identity$, we have $\smallmat{A_{I,J} \bar{x}_J + \identity_{I,K} \bar{y}_K\\ w^\intercal_J \bar{x}_J} = 
B \smallmat{\bar{x}_J\\ \bar{y}_{K}} = \smallmat{\zero\\ \det(B)}$, which implies $\smallmat{\bar{x}_{J}\\ \bar{y}_{K}} = \pm \smallmat{x_{J}\\ y_{K}}$ and $|w^\intercal x| = |w^\intercal_J x_J| = |w^\intercal_J \bar{x}_J| = |\det(B)| \le \Delta$.

Suppose now that every circuit $\smallmat{x \\ y}$ of $\begin{bmatrix} A & \identity \end{bmatrix}$ satisfies $|w^\intercal x| \le \Delta$.
Consider any invertible square submatrix $B$ of $\smallmat{A& \identity\\ w^\intercal &\zero^\intercal}$.
We have to show that $|\det(B)| \le \Delta$.
Notice that if $|\det(B)| \ge 2$, then $B = \smallmat{A_{I,J} & \identity_{I,K}\\ w^\intercal_J & \zero^\intercal}$ for some $I \subseteq [m]$, $J \subseteq [n]$, $K \subseteq [m]$. Notice that $K \subseteq I$ since otherwise $B$ has a zero column.
By adding every index $i \in \overline{I} = [m] \setminus I$ both to $I$ and $K$, we may even assume $B = \smallmat{A_{[m],J} & \identity_{[m],K}\\ w^\intercal_J & \zero^\intercal}$.
Pick $\smallmat{\bar{x}_J\\ \bar{y}_{K}} \in \{0,\pm 1\}^{J \cup K}$ in the kernel of $\smallmat{A_{[m],J} & \identity_{[m],K}}$ such that $w^\intercal_J \bar{x}_J \ne 0$.
Extend $\bar{x}$ and $\bar{y}$ to $x \in \{0,\pm 1\}^n$ and $y \in \{0,\pm 1\}^m$ by adding zeros, respectively.
By \Cref{remCircuitWitness}, $\smallmat{x \\ y}$ is a circuit of $\begin{bmatrix} A & \identity \end{bmatrix}$.
As above, we conclude $|\det(B)| = |w^\intercal x| \le \Delta$.
\end{proof}

Next, we recall that circuits generalize edge directions of polyhedra.
For instance, every edge direction of the polyhedron $\{x \in \R^n : Ax = b, \, \ell \le x \le u\}$ is a scalar multiple of a support-minimal vector of $\ker(A)$, see \cite[Lem. 2.18]{onn_2010}.
More importantly, we will crucially exploit the fact that nonzero vectors in kernels of totally unimodular matrices can be decomposed into circuits in a conformal way:

\begin{lemma}[{see~\cite[Lemma 7]{naegele_2022}}]
    \label{lemDecomposition}
    Let $P = \{x \in \R^n : Ax = b, \, \ell \le x \le u\}$ where $A \in \{0,\pm 1\}^{m \times n}$ is totally unimodular, $b \in \R^m$, and $d = \dimension(P)$, $\ell \in \Z^n$, $u \in \Z^n$, and let $x,x' \in P$.
    In strongly polynomial time, we can find conformal $c_1,\dots,c_d \in \circuits(A)$ and $\lambda_1,\dots,\lambda_d \ge 0$ with $x' = x + \sum_{j=1}^d \lambda_j c_j$.
    If $x,x'$ are integer vectors, then $\lambda_1,\dots,\lambda_d$ can be chosen to be nonnegative integers.
\end{lemma}

\begin{remark}
    \label{remarkProximity}
    If $P$, $x$, $x'$, $c_1,\dots,c_d$, and $\lambda_1,\dots,\lambda_d$ are as above, then for every $\mu_1,\dots,\mu_d$ with $0 \le \mu_i \le \lambda_i$ for all $i \in [d]$, the vector $x + \sum_{j=1}^d \mu_j c_j$ is also contained in $P$.
\end{remark}
\begin{proof}
    Let $x'' = x + \sum_{j=1}^d \mu_j c_j$.
    Since $c_1,\dots,c_d$ are circuits of $A$, we have $Ax'' = Ax = b$.
    To see that $\ell \le x \le u$ holds, consider any coordinate index $i \in [n]$.
    Suppose first that $x_i \le x_i'$.
    Since $c_1,\dots,c_d$ are conformal, this implies $(c_j)_i \ge 0$ for all $j \in [d]$, and hence
    \[
        \ell_i
        \le x_i
        \le x_i + \sum_{j=1}^d \mu_j (c_j)_i
        = x_i''
        \le x_i + \sum_{j=1}^d \lambda_j (c_j)_i
        = x_i' \le u_i.
    \]
    The case $x_i \ge x_i'$ follows analogously.
\end{proof}

\subsection{An Intermediate Result}\label{sec:intermediate_result}

In general, our proof of the proximity result follows the same framework as the proximity proof by Eisenbrand and Weismantel~\cite{EisenbrandWeismantel}, extending their results to totally unimodular inequality systems with a constant number of additional equality constraints.
A similar proximity framework has recently been applied for the exact weight basis problem in matroids~\cite{ERW24}.

Our proofs of \Cref{thmProximity} and \Cref{propReductionIP} are based on the following result.

\begin{proposition}
    \label{propProximity}
    Let $k,\Delta \in \Z_{\ge 1}$ be fixed, and let $p$, $A$, $b$, $W$, $d$, $\ell$, $u$ define an instance of \Cref{problemEquality}.
    Given an optimal vertex solution $x^* \in \R^n$ of \eqref{eqLPequality}, in strongly polynomial time we can compute an integer point $z \in \Z^n$ with
    \begin{enumerate}
        \item \label{itemPropProximity1} $Az = b$, $\ell \le z \le u$,
        \item \label{itemPropProximity2} $\|x^* - z\|_\infty < k$, and
        \item \label{itemPropProximity3} if \eqref{eqIPequality} is feasible, then it has an optimal solution $z^* = z + \sum_{j=1}^t c_j$ where $c_1,\dots,c_t \in \circuits(A)$ are conformal and $t \le k(2k \Delta + 1)^k =: f_{\ref{propProximity}}(k, \Delta)$.
    \end{enumerate}
\end{proposition}

Note that \Cref{propReductionIP} follows directly from \Cref{propProximity} by taking $f_{\ref{propReductionIP}}(k,\Delta) = f_{\ref{propProximity}}(k, \Delta)$. To see how \Cref{propProximity} implies \Cref{thmProximity}, consider any instance of \Cref{problemEquality}.
Let $x^*$ be any optimal solution of \eqref{eqLPequality}, and assume that \eqref{eqIPequality} is feasible.
Pick a vertex $x'$ of $\{x \in \R^n : Ax = b, \, Wx = d, \, \lfloor x^* \rfloor \le x \le \lceil x^* \rceil\}$ with $p^\intercal x' = p^\intercal x^*$.
Here, the vector $\lfloor x^* \rfloor$ arises from $x^*$ by applying $\lfloor \cdot \rfloor$ to each entry of $x^*$, and $\lceil x^* \rceil$ is defined analogously.
Let $I = \{i \in [n] : x'_i \in \Z\}$.
Note that, for the sake of proving \Cref{thmProximity}, we can assume that we know the optimal solution to \eqref{eqIPequality}, hence we may partition $I$ into $I_\le$ and $I_\ge$ such that the optimum value of
\begin{equation}
    \label{eqIPbla}
    \max \left \{ p^\intercal x : Ax = b, \, Wx = d, \, \ell \le x \le u,
    x_i \le x_i' \text{ for all } i \in I_\le, \, x_i \ge x_i' \text{ for all } i \in I_\ge, \, x \in \Z^n \right \}
\end{equation}
is equal to \eqref{eqIPequality}.
Since $x'$ is an optimal vertex solution of the LP relaxation of \eqref{eqIPbla}, we may apply \Cref{propProximity} to obtain a point $z \in \Z^n$ with $\|x' - z\|_\infty < k$ such that \eqref{eqIPbla} has an optimal solution $z^* = z + \sum_{j=1}^t c_j$ where $c_1,\dots,c_t \in \circuits(A)$ and $t \le f_{\ref{propProximity}}(k, \Delta)$.
Note that $z^*$ is also an optimal solution of \eqref{eqIPequality} and observe that we have
\[
    \|x^* - z^*\|_\infty
    \le \|x^* - x'\|_\infty + \|x' - z\|_\infty + \|z - z^*\|_\infty
    \le 1 + k + \|z - z^*\|_\infty
    \le 1 + k + \sum \nolimits_{j=1}^t \|c_j\|_\infty
    = 1 + k + t.
\]
The result follows by taking $f_{\ref{thmProximity}}(k,\Delta) := 1 + k + f_{\ref{propProximity}}(k, \Delta)$.

The remainder of this section is devoted to the proof of \Cref{propProximity}.
To this end, let $p$, $A$, $b$, $W$, $d$, $\ell$, $u$ define an instance of \Cref{problemEquality} and denote $P := \{x \in \R^n : Ax = b,\, \ell \le x \le u\}$.
Moreover, let $x^*$ be an optimal vertex solution of \eqref{eqLPequality}.

\subsection{Proof of \Cref{propProximity}}

Let us first describe how the integer point $z$ can be computed, and then show that it satisfies \ref{itemPropProximity1}--\ref{itemPropProximity3}.
To this end, consider the inclusionwise smallest face $F$ of $P$ that contains $x^*$.
Recall that $Wx = d$ consists of only $k$ equations, and hence $\dimension(F) \le k$.
Notice also that $F = \{x \in \R^n : Ax = b,\, \ell' \le x \le u'\}$ for some appropriately chosen $\ell',u' \in \Z^n$.
Since $A$ is totally unimodular, we can compute a vertex $z'$ of $F$ in strongly polynomial time.
As $P$ is an integer polytope, so is $F$ and hence $z'$ is integer-valued.
Moreover, we can invoke \Cref{lemDecomposition} to find circuits $c'_1,\dots,c'_k$ of $A$ and $\lambda_1,\dots,\lambda_k \in \R_{\ge 0}$ with $z' = x^* + \sum_{j=1}^k \lambda_j c'_j$.
We define
\[
    z = x^* + \sum_{j=1}^k (\lambda_j - \lfloor \lambda_j \rfloor) c'_j
\]
and note that, by \Cref{remarkProximity}, it is contained in $F$.
In particular, we see that $z$ satisfies~\ref{itemPropProximity1}.
Note also that $z$ is an integer vector since
\[
    z = x^* + \sum_{j=1}^k (\lambda_j - \lfloor \lambda_j \rfloor) c'_j
    = \underbrace{z'}_{\in \Z^n} - \sum_{j=1}^k \underbrace{\lfloor \lambda_j \rfloor}_{\in \Z} \underbrace{c'_j}_{\in \Z^n}.
\]
Moreover, it satisfies \ref{itemPropProximity2} since
\[
    \|x^* - z\|_\infty
    = \left\| \sum_{j=1}^k (\lambda_j - \lfloor \lambda_j \rfloor) c'_j \right\|_\infty
    \le \sum_{j=1}^k (\lambda_j - \lfloor \lambda_j \rfloor) \|c'_j\|_\infty
    < \sum_{j=1}^k \|c'_j\|_\infty
    \le k.
\]
Thus, it remains to show \ref{itemPropProximity3}.
To this end, we first observe that since $F$ is the minimal face of $P$ containing $x^*$, we have
\begin{equation}
    \label{eqMinFace}
    z + v \in P \implies x^* + \eps v \in P \text{ for some } \eps > 0
\end{equation}
for every vector $v \in \R^n$.

Now, take any optimal solution $z^*$ of the integer program, and decompose
\[
    z^* - z = c_1+\dotsb+c_t
\]
where $c_1,\dots,c_t \in \circuits(A)$ are conformal (such a decomposition exists due to \Cref{lemDecomposition}), and choose $z^*$ and this decomposition such that $t$ is smallest possible.
Recall that for every $I \subseteq [t]$, the vector $z + \sum_{i \in I} c_i$ is also contained in $P$.

\subsubsection{Reordering Circuits}

To obtain a bound on the number $t$ of circuit vectors, let us first reorder $c_1,\dots,c_t$ such that all partial sums $\sum_{i=1}^s Wc_i$ have small norm.
For this purpose, we make use of the Steinitz lemma~\cite{steinitz_1913}:

\begin{lemma}[{see \cite{sevast1978approximate,sevast1997steinitz}}]
    \label{lemSteinitz}
    Let $\|\cdot\|$ be any norm on $\R^k$.
    If $y_1,\dots,y_{t'} \in \R^k$ satisfy $y_1+\dots+y_{t'} = \zero$, then there exists a permutation $\pi$ on $[t']$ such that
    \[
        \left\| \sum \nolimits_{i=1}^s y_{\pi(i)} \right\| \le k \cdot \max \{\|y_1\|,\dots,\|y_{t'}\|\}
    \]
    holds for all $s \in [t']$.
\end{lemma}

To apply the above lemma, we first set $y_i = Wc_i$ for $i=1,\dots,t$.
Recall that we consider an instance of \Cref{problemEquality}, and hence
\[
    \|Wc_i\|_\infty \le \Delta
\]
holds for all $i \in [t]$.
Moreover, we have
\[
    \|Wc_1 + \dotsb + Wc_t\|_\infty
    = \|W(z^* - z)\|_\infty
    \le \|W(z^* - x^*)\|_\infty + \|W(x^* - z)\|_\infty
    = \|W(x^* - z)\|_\infty
    \le k\Delta,
\]
where the equality holds since we have $Wz^* = d$ and $Wx^* = d$.
This means that we can find integer vectors $y_{t+1},\dots,y_{t+k} \in \Z^k$ such that $y_{t+1} + \dots + y_{t+k} = -(Wc_1 + \dotsb + Wc_t)$ and $\|y_{t+j}\|_\infty \le \Delta$ holds for all $j \in [k]$.
By construction, we have $y_1 + \dots + y_{t+k} = \zero$ and $\|y_i\|_\infty \le \Delta$ for all $i \in [t+k]$.
Thus, by \Cref{lemSteinitz} there exists a permutation $\pi$ on $[t+k]$ such that 
\begin{equation}\label{eqh9f2fhufwe}
    \left\| \sum_{i=1}^s y_{\pi(i)} \right\|_\infty \le k \Delta
\end{equation}
holds for all $s \in [t+k]$.

We now claim that
\begin{equation}
    \label{eqBoundNumberOfCircuitVectors}
    t \le k(2k \Delta + 1)^k
\end{equation}
holds, which yields \ref{itemPropProximity3}.
For the sake of contradiction, suppose that the claim does \emph{not} hold.
Since there are $(2k\Delta+1)^k$ integer points of infinity norm at most $k\Delta$ in $\Z^k$, we can apply the pigeonhole principle and~\eqref{eqh9f2fhufwe} to obtain $k+1$ indices $1\le s_1<s_2,\dots<s_{k+1}\le t+k$ such that $\sum_{i=1}^{s_1} y_{\pi(i)} = \sum_{i=1}^{s_j} y_{\pi(i)}$ for all $j\in[k+1]$.
Thus, the (nonempty) sets
\begin{equation*}
    S_{i} := \{\pi(s_i+1),\dots,\pi(s_{i+1})\}\text{ for } i\in[k], 
\end{equation*}
and
\begin{equation*}
    S_{k+1} := \{\pi(1),\dots,\pi(s_1)\} \cup \{\pi(s_{k+1}+1),\dots,\pi(t+k)\}
\end{equation*}
yield $k+1$ disjoint partial sums, summing up to zero each.
By the pigenohole principle, one of the $S_i$ does not contain any $j\in\{t+1,\dots,t+k\}$, say $S'$.
Setting $\bar{c} = \sum_{i \in S'} c_i$, we have $W\bar{c} = \zero$.


\subsubsection{An Optimal Solution with a Shorter Decomposition}

Consider the vector $\tilde z := z^* - \bar{c} = z + \sum_{i \in [t] \setminus S'} c_i \in P$.
Since $W\tilde z = Wz^* - W\bar{c} = d$, we see that $\tilde z$ is feasible for the integer program.
We claim that
\begin{equation}
    \label{eqashd89ahds}
    p^\intercal \bar{c} \le 0
\end{equation}
holds.
Note that the latter claim implies $p^\intercal \tilde z = p^\intercal z^* - p^\intercal \bar{c} \ge p^\intercal z^*$.
This shows that $\tilde z$ is also an optimal solution for the integer program with a shorter decomposition, a contradiction.
Thus, our assumption (that $t$ does not satisfy the bound in~\eqref{eqBoundNumberOfCircuitVectors}) was wrong and we are done.

To show~\eqref{eqashd89ahds}, note that $z + \bar{c} = z + \sum_{i \in I} c_i \in P$.
Thus, by~\eqref{eqMinFace} there exists some $\eps > 0$ such that $x^* + \eps \bar{c} \in P$.
Moreover, since $W(x^* + \eps \bar{c}) = Wx^* = d$ we see that $x^* + \eps \bar{c}$ is a feasible solution for the LP relaxation.
By the optimality of $x^*$, this implies $p^\intercal \bar{c} \le 0$.


\section{Dealing with $1$-sums and $2$-sums}\label{sec:1sum2sum}

In \Cref{sec:configs}, we define $1$-sums and $2$-sums for vector configurations. Applied to the vector configurations arising from the columns of a TU matrix, we recover the standard notions of $1$-sum and $2$-sum of TU matrices (see for instance~\cite{schrijver_1998}), and regular matroids (see for instance~\cite{oxley_2006}). In \Cref{sec:MCICP_bis}, we state the actual version of \Cref{problemEquality} that we solve. Next, in \Cref{sec:reduction_2-connected_bis}, we explain how to reduce \Cref{problemEquality} to the $2$-connected case. Finally, in \Cref{sec:reduction_3-connected_bis}, we explain how to further reduce to the almost $3$-connected case. We point out that our reductions are all black-box. 

\subsection{Vector configurations} \label{sec:configs}

We regard matrices $A \in \R^{m \times n}$ as \emph{(vector) configurations} in $\R^m$. In other words, we see $A \in \R^{m \times n}$ as an ordered multiset of $n$ vectors $\{a_1,\ldots,a_n\}$ in $\R^m$, namely, the $n$ columns of $A$. The \emph{size} of a configuration is defined as its number of vectors, taking into account multiplicities. Slightly abusing notation, we write $v \in A$ to mean that $v$ is a column of $A$. For $v \in A$, we let $A-v$ denote the configuration obtained from $A$ by deleting the column for $v$, hence decreasing the multiplicity of $v$ by $1$. If $A_1 \in \R^{m \times n_1}$ and $A_2 \in \R^{m \times n_2}$ are configurations in $\R^m$, we let $A_1 \uplus A_2 := \begin{bmatrix} A_1 &A_2 \end{bmatrix}$ denote the configuration obtained by concatenating the two configurations. Below, we let $\colsp(A) := \mathspan(\{a_1,\ldots,a_n\})$ denote the column space of $A$.

We point out that configurations are, basically, representations of real matroids. Most of the notions we discuss here such as connectivity, $1$-sum and $2$-sum are the counterparts of standard matroid notions. The reader who has some knowledge on matroids should ``feel at home''.
We remark that our definition of configurations can be replaced by matrices with appropriately defined block operations, as for instance done in~\cite{schrijver_1998}. For our purpose here, it is more convenient to work with configurations in order to avoid dealing with block matrices and their respective dimensions.

\begin{definition}[regular configuration]
We call a configuration $A \in \R^{m \times n}$ \emph{regular} whenever every circuit of $A$ is in $\{0,\pm 1\}^n$. (This terminology is inspired by Tutte's notion of a regular subspace, see~\cite{Tutte65}.)
\end{definition}

Every TU matrix $A \in \R^{m \times n}$, seen as a vector configuration, is a regular configuration. Conversely, if $A \in \R^{m \times n}$ is a regular configuration then there exists a TU matrix $A' \in \{0,\pm 1\}^{m \times n}$ such that $\ker A = \ker A'$. Assume that $A$ is a regular configuration. We say that a circuit $c \in \circuits(A)$ \emph{uses} a vector $v \in A$ if $c(v) \neq 0$. Circuit $c$ uses $v$ \emph{positively} if $c(v) = 1$, and \emph{negatively} if $c(v) = -1$.

We recall the definition of connectivity given in the introduction.

\begin{definition}[separation, connectivity]
A configuration $A \in \R^{m \times n}$ has a \emph{$q$-separation} (for $q \in \Z_{\ge 1}$) if it can be written, possibly after permuting its columns, as $A = \begin{bmatrix} A_1 & A_2 \end{bmatrix}$, with $\dim (\colsp(A_1) \cap \colsp(A_2)) \le q-1$ where both $A_1$ and $A_2$ have at least $q$ columns. The \emph{connectivity} of $A$ is defined as the smallest order of a separation of $A$. We say that $A$ is \emph{$q$-connected} if its connectivity is at least $q$. A configuration $A$ is said to be \emph{almost $3$-connected} if it becomes $3$-connected when one deletes from $A$ every column that is a nonzero multiple of another column.
\end{definition}

We remark that in particular, under this definition a configuration consisting of a single column has no separation, and hence is $q$-connected for any $q \ge 1$.

Next, for $q = 1, 2$, we relate these definitions to the notion of a $q$-sum.

\begin{definition}[$1$-sum]
Let $A_1 \in \R^{m \times n_1}$ and $A_2 \in \R^{m \times n_2}$ be configurations such that $\colsp(A_1) \cap \colsp(A_2) = \{\zero\}$. The \emph{$1$-sum} of $A_1$ and $A_2$ is the configuration $A_1 \oplus_1 A_2 := A_1 \uplus A_2 \in \R^{m \times (n_1 + n_2)}$. 
\end{definition}

The circuits of a $1$-sum can be characterized as follows:
\begin{equation}
\nonumber
\circuits(A_1 \oplus_1 A_2) = \left\{ \begin{bmatrix} c_1\\ \zero\end{bmatrix} : c_1 \in \circuits(A_1) \right\} \cup \left\{ \begin{bmatrix} \zero\\ c_2\end{bmatrix} : c_2 \in \circuits(A_2) \right\}\,. 
\end{equation}

A configuration $A$ of size at least $2$, that contains the zero vector is never $2$-connected. We say that two vectors $v, v' \in \R^m$ are \emph{parallel} if one is a nonzero multiple of the other. Hence the zero vector is only parallel to itself.

\begin{definition}[$2$-sum]\label{def:2sum}
Let $A_1 \in \R^{m \times (n_1 + 1)}$ and $A_2 \in \R^{m \times (n_2 + 1)}$ be configurations with a common nonzero vector $v$, and such that $\colsp(A_1) \cap \colsp(A_2) = \mathspan(\{v\})$. The \emph{$2$-sum} of $A_1$ and $A_2$ is the configuration $A_1 \oplus_2 A_2 := (A_1 - v) \uplus (A_2 - v) \in \R^{m \times (n_1 + n_2)}$.
\end{definition}

We give an example of a $2$-sum of two configurations in \Cref{fig:2-sum}.

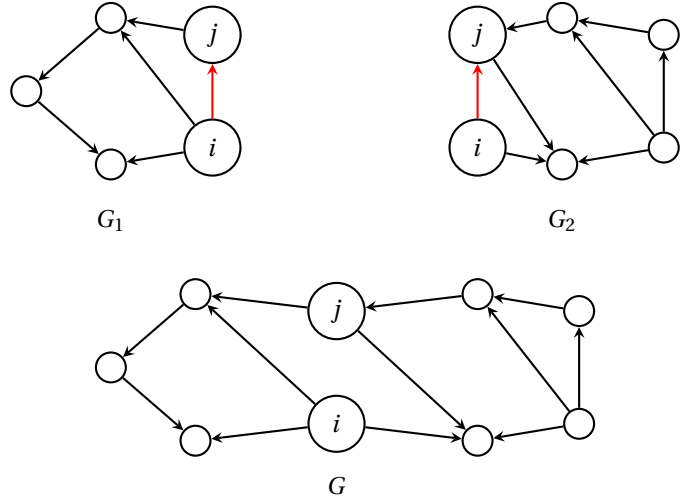
\begin{figure} 
   \centering
      \begin{tikzpicture}[inner sep=2.5pt,scale=.75]

 \pgfmathsetmacro{\l}{2} 

\node[vtx, inner sep = 4.8pt] (a1) at (1.8,0) {\small $i$};
\node[vtx] (a2) at (1.8,\l) {\small $j$};
\node[vtx] (a3) at (0,\l+0.3) {};
\node[vtx] (a4) at (0,-0.3) {};
\node[vtx] (a5) at (-1.5,0.5*\l) {};

\node at  ([yshift=-1cm] a4) {$G_1$};

\draw[ thick,->,color=red] (a1) -> (a2);
\draw[ thick,->,color=black] (a2) -> (a3);
\draw[ thick,->,color=black] (a3) -> (a5);
\draw[ thick,->,color=black] (a5) -> (a4);
\draw[ thick,->,color=black] (a1) -> (a4);
\draw[thick,->,color=black] (a1) -> (a3);

 \pgfmathsetmacro{\d}{8}

    \node[vtx] (b1) at (1.8+\d,0) {};
    \node[vtx] (b2) at (1.8+\d,\l) {};
    \node[vtx] (b3) at (0+\d,\l+0.3) {};
    \node[vtx] (b4) at (0+\d,-0.3) {};
    \node[vtx] (b5) at (-1.5+\d,\l) {\small $j$};
    \node[vtx, inner sep = 4.8pt] (b6) at (-1.5+\d,0) {\small $i$};

    \node at ([yshift=-1cm] b4) {$G_2$};

    \draw[thick,->,color=black] (b1) -- (b2);
    \draw[thick,->,color=black] (b2) -- (b3);
    \draw[thick,->,color=black] (b3) -- (b5);
    \draw[thick,->,color=black] (b5) -- (b4);
    \draw[thick,->,color=black] (b6) -- (b4);
    \draw[thick,->,color=black] (b1) -- (b4);
    \draw[thick,->,color=black] (b1) -- (b3);
      \draw[thick,->,color=red] (b6) -- (b5);
\end{tikzpicture}

\vspace{0.5cm}
\begin{tikzpicture}[inner sep=2.5pt,scale=.75]

 \pgfmathsetmacro{\l}{2} 
  \pgfmathsetmacro{\s}{0.7} 

\node[vtx, inner sep = 4.8pt] (a1) at (1.8+\s,0) {\small $i$};
\node[vtx] (a2) at (1.8+\s,\l) {\small $j$};
\node[vtx] (a3) at (0,\l+0.3) {};
\node[vtx] (a4) at (0,-0.3) {};
\node[vtx] (a5) at (-1.5,0.5*\l) {};

\draw[ thick,->,color=black] (a2) -> (a3);
\draw[ thick,->,color=black] (a3) -> (a5);
\draw[ thick,->,color=black] (a5) -> (a4);
\draw[ thick,->,color=black] (a1) -> (a4);
\draw[thick,->,color=black] (a1) -> (a3);

 \pgfmathsetmacro{\d}{5} 

\node at ([yshift=-1.1cm,xshift=0 cm]  a1) {$G$};

    \node[vtx] (b1) at (1.8+\d,0) {};
    \node[vtx] (b2) at (1.8+\d,\l) {};
    \node[vtx] (b3) at (0+\d,\l+0.3) {};
    \node[vtx] (b4) at (0+\d,-0.3) {};

    \draw[thick,->,color=black] (b1) -- (b2);
    \draw[thick,->,color=black] (b2) -- (b3);
    \draw[thick,->,color=black] (b3) -- (a2);
    \draw[thick,->,color=black] (a2) -- (b4);
    \draw[thick,->,color=black] (a1) -- (b4);
    \draw[thick,->,color=black] (b1) -- (b4);
    \draw[thick,->,color=black] (b1) -- (b3);
\end{tikzpicture}

\caption{An example of 2-sum of incidence matrices, represented as directed graphs. In order to fit Definition \ref{def:2sum}, $A_1$ should be the incidence matrix of $G_1$ with four extra zero rows for the vertices in $V(G_2) \setminus V(G_1)$, and $A_2$ should be defined similarly, so that $A_1$ and $A_2$ have as many rows as the number of vertices of $G$ and share a column $v$ corresponding to arc $(i,j)$. Then it is easy to check that $A_1\oplus_2 A_2$ is the incidence matrix of $G$.}
\label{fig:2-sum}
\end{figure}

Assuming that configurations $A_1$ and $A_2$ are regular, and that the common vector $v$ is the last vector of $A_1$ and the first vector of $A_2$, we have the following characterization of the circuits of a $2$-sum:
\begin{align}
\circuits(A_1 \oplus_2 A_2) &= \left\{ \begin{bmatrix} c_1\\ \zero\end{bmatrix} : \begin{bmatrix} c_1\\ 0 \end{bmatrix} \in \circuits(A_1) \right\} \cup \left\{ \begin{bmatrix} \zero\\ c_2\end{bmatrix} :  \begin{bmatrix} 0\\ c_2 \end{bmatrix} \in \circuits(A_2) \right\} \nonumber\\
&\mbox{} \quad \cup \left\{ \begin{bmatrix} c_1\\ c_2 \end{bmatrix} : \begin{bmatrix} c_1\\ \pm 1 \end{bmatrix} \in \circuits(A_1),\ \begin{bmatrix} \mp 1\\ c_2 \end{bmatrix} \in \circuits(A_2) \right\}\,.\nonumber
\end{align}
In other words, each circuit of $A_1 \oplus_2 A_2$ is either a circuit of $A_i$ that does not use the common vector $v$, for some $i \in [2]$, or obtained by combining a circuit of $A_i$ that uses $v$ positively together with a circuit of $A_{3-i}$ that uses $v$ negatively, for some $i \in [2]$. We say that the circuits of the second type \emph{cross} $v$.

We call a \emph{circulation} of $A$ any vector $x \in \R^n$ such that $Ax = \zero$, that is, any vector in $\ker(A)$. The above characterization of the circuits of a $2$-sum naturally extends to circulations: $\smallmat{x_1\\ x_2}$ is a circulation of $A_1 \oplus_2 A_2$ if and only if there exists a scalar $\varphi$ such that $\smallmat{x_1\\ -\varphi}$ is a circulation of $A_1$ and $\smallmat{\varphi\\ x_2}$ is a circulation of $A_2$. We say that $\smallmat{x_1\\ -\varphi}$ and $\smallmat{\varphi\\ x_2}$ are obtained by \emph{splitting} $\smallmat{x_1\\ x_2}$. If $\smallmat{x_1\\ x_2}$ is an integer circulation splitting into $\smallmat{x_1\\ -\varphi}$ and $\smallmat{\varphi\\ x_2}$, implying that $\varphi$ is integer, then it can be written as a conformal sum of circuits in $\circuits(A_1 \oplus_2 A_2)$.
If $A$ is a regular configuration, then we can assume that exactly $|\varphi|$ circuits cross $v$.

Notice that a configuration is $2$-connected if and only if it cannot be written as the $1$-sum of two smaller configurations, possibly changing the ordering of its vectors, and $3$-connected if and only if it cannot be written as the $2$-sum of two configurations (possibly changing the ordering of its vectors), each with at least $3$ vectors.

\begin{definition}[decomposition tree]
A tree $T$ is a \emph{decomposition tree} of a $2$-connected configuration $A \in \R^{m \times n}$ of size at least $3$ if it satisfies the following conditions:\medskip
\begin{enumerate}
\item each node $t$ of $T$ has a corresponding configuration $A(t)$ in $\R^m$ of size at least $3$,
\item each edge $tt'$ of $T$ has a corresponding nonzero vector\footnote{The vector $\overline{v} = \overline{v}(tt')$ is considered as being \emph{private} to $A(t)$ and $A(t')$. Formally, we would need to use labels for the vectors of the configurations under consideration in order to distinguish the different copies of the same vector. Then, $\overline{v}$ could possibly appear as a vector in $A(t'')$ for some node $t''$ distinct from $t$ and $t'$, but not with the same label. For the sake of simplicity, we will however not do this explicitly.} $\overline{v} = \overline{v}(tt') \in \R^m$ that belongs to both $A(t)$ and $A(t')$, and such that $\colsp(A(t)) \cap \colsp(A(t')) = \mathspan(\{\overline{v}\})$,
\item if $t$ and $t'$ are distinct nodes of $T$ then $\dim(\colsp(A(t)) \cap \colsp(A(t'))) \le 1$, and if $t''$ is a node of the $t$--$t'$ path in $T$ then $\colsp(A(t)) \cap \colsp(A(t')) \subseteq \colsp(A(t''))$, 
\item \label{cond:merge} the configuration obtained by iteratively performing (in any order) the $2$-sum corresponding to each edge of $T$, and then possibly reordering the vectors, equals $A$.
\end{enumerate}
\end{definition}

\begin{figure}[h] 
 \centering
 \begin{tikzpicture}[inner sep=2.5pt,scale=.75]

 \pgfmathsetmacro{\l}{2} 

\node[vtx] (a1) at (2.8,0) {};
\node[vtx] (a2) at (2.8,\l) {};
\node[vtx] (a3) at (0,\l+0.3) {};
\node[vtx] (a4) at (0,-0.3) {};
\node[vtx] (a5) at (-1.5,0.5*\l) {};

\draw[ dashed,->,color=red] (a1) -> (a2);
\draw[ thick,->,color=black] (a2) -> (a3);
\draw[ thick,->,color=black] (a3) -> (a5);
\draw[ thick,->,color=black] (a5) -> (a4);
\draw[ thick,->,color=black] (a1) -> (a4);
\draw[thick,->,color=black] (a1) -> (a3);

 \pgfmathsetmacro{\d}{6}

    \node[vtx] (b1) at (2.8+\d,0) {};
    \node[vtx] (b2) at (2.8+\d,\l) {};
    \node[vtx] (b3) at (0+\d,\l+0.3) {};
    \node[vtx] (b4) at (0+\d,-0.3) {};

    \node at ([yshift=-2cm] b4) {$G$}; \node at ([yshift=-3cm] b4) {};
\draw[thick,->,color=black] (a1) -- (b3);
    \draw[dashed,->,color=blue] (b1) -- (b2);
    \draw[thick,->,color=black] (b2) -- (b3);
    \draw[thick,->,color=black] (b3) -- (a2);
    \draw[thick,->,color=black] (a2) -- (b4);
    \draw[dashed,->,color=green] (a1) -- (b4);
    \draw[thick,->,color=black] (b1) -- (b4);
    \draw[thick,->,color=black] (b1) -- (b3);

        \node[vtx] (c1) at (1+2*\d,-0.3) {};
    \node[vtx] (c2) at (1+2*\d,0.5*\l) {};
    \node[vtx] (c3) at (1+2*\d,\l+0.3) {};

    \draw[thick,->,color=black] (c1) -- (c2);
    \draw[thick,->,color=black] (c2) -- (c3);
    \draw[thick,->,color=black] (c2) -- (b1);
    \draw[thick,->,color=black] (c3) -- (b2);
     \draw[thick,->,color=black] (b2) -- (c2);
    \draw[thick,->,color=black] (c1) -- (b1);
 \pgfmathsetmacro{\b}{1.5} 
 
\node[vtx] (d1) at (1+0,-\b) {};
\node[vtx] (d2) at (1+\l,-\b) {};
\node[vtx] (d3) at (1+2*\l,-\b) {};

           \draw[thick,->,color=black] (a1) -- (d1);
     \draw[thick,->,color=black] (d1) -- (d2);
    \draw[thick,->,color=black] (d3) -- (d2);
       \draw[thick,->,color=black] (d3) -- (b4);
   \draw[thick,->,color=black] (d2) -- (b4);
   \draw[thick,->,color=black] (a1) -- (d2);

\end{tikzpicture}

      \begin{tikzpicture}[inner sep=2.5pt,scale=.75]

 \pgfmathsetmacro{\l}{2} 

\node[vtx] (a1) at (1.8,0) {};
\node[vtx] (a2) at (1.8,\l) {};
\node[vtx] (a3) at (0,\l+0.3) {};
\node[vtx] (a4) at (0,-0.3) {};
\node[vtx] (a5) at (-1.5,0.5*\l) {};

\node at  ([yshift=-1cm] a4) {$G_1$};

\draw[ thick,->,color=red] (a1) -> (a2);
\draw[ thick,->,color=black] (a2) -> (a3);
\draw[ thick,->,color=black] (a3) -> (a5);
\draw[ thick,->,color=black] (a5) -> (a4);
\draw[ thick,->,color=black] (a1) -> (a4);
\draw[thick,->,color=black] (a1) -> (a3);

 \pgfmathsetmacro{\d}{6}

    \node[vtx] (b1) at (1.8+\d,0) {};
    \node[vtx] (b2) at (1.8+\d,\l) {};
    \node[vtx] (b3) at (0+\d,\l+0.3) {};
    \node[vtx] (b4) at (0+\d,-0.3) {};
    \node[vtx] (b5) at (-1.5+\d,\l) {};
    \node[vtx] (b6) at (-1.5+\d,0) {};

    \node at ([yshift=-1cm] b4) {$G_2$};
\draw[thick,->,color=black] (b6) -- (b3);
    \draw[thick,->,color=blue] (b1) -- (b2);
    \draw[thick,->,color=black] (b2) -- (b3);
    \draw[thick,->,color=black] (b3) -- (b5);
    \draw[thick,->,color=black] (b5) -- (b4);
    \draw[thick,->,color=green] (b6) -- (b4);
    \draw[thick,->,color=black] (b1) -- (b4);
    \draw[thick,->,color=black] (b1) -- (b3);
      \draw[thick,->,color=red] (b6) -- (b5);

        \node[vtx] (c1) at (1+2*\d,-0.3) {};
    \node[vtx] (c2) at (1+2*\d,0.5*\l) {};
    \node[vtx] (c3) at (1+2*\d,\l+0.3) {};
    \node[vtx] (c5) at (-1.5+2*\d,\l) {};
    \node[vtx] (c6) at (-1.5+2*\d,0) {};

    \node at ([xshift = 1.1cm,yshift=-1cm] c6) {$G_3$};

    \draw[thick,->,color=black] (c1) -- (c2);
    \draw[thick,->,color=black] (c2) -- (c3);
    \draw[thick,->,color=blue] (c6) -- (c5);
    \draw[thick,->,color=black] (c2) -- (c6);
    \draw[thick,->,color=black] (c3) -- (c5);
     \draw[thick,->,color=black] (c5) -- (c2);
    \draw[thick,->,color=black] (c1) -- (c6);
\end{tikzpicture}

\vspace{0.5cm}
\begin{tikzpicture}[inner sep=2.5pt,scale=.75]

 \pgfmathsetmacro{\d}{4} 
 \pgfmathsetmacro{\l}{2} 

\node[vtx] (a1) at (0-\d,0) {};
\node[vtx] (a2) at (\l-\d,0) {};
\node[vtx] (a3) at (2*\l-\d,0) {};

 \node[vtx] (b1) at (1.5*\l-\d,0+\l) {};
   \node[vtx] (b4) at (0.5*\l-\d,\l) {};
    
 \node at ([xshift = 0cm,yshift=-1cm] a2) {$G_4$};
    
       \draw[thick,->,color=green] (b4) -- (b1);

           \draw[thick,->,color=black] (b4) -- (a1);
     \draw[thick,->,color=black] (a1) -- (a2);
    \draw[thick,->,color=black] (a3) -- (a2);
       \draw[thick,->,color=black] (a3) -- (b1);
   \draw[thick,->,color=black] (a2) -- (b1);
   \draw[thick,->,color=black] (b4) -- (a2);

   \node[vtx] (t1) at (+\d,\l) { $G_1$};
\node[vtx] (t2) at (\l+\d,\l) {$G_2$};
\node[vtx] (t3) at (2*\l+\d,\l) {$G_3$};

 \node[vtx] (t4) at (\l+\d,0) {$G_4$};

  \node at ([xshift = 0cm,yshift=-1cm] t4) {$T$};

\draw[very thick,color=red] (t1) -- (t2);
\draw[very thick,color=blue] (t3) -- (t2);
\draw[very thick,color=green] (t4) -- (t2);
 
\end{tikzpicture}
\caption{An example of a 2-sum decomposition of an incidence matrix into four matrices, represented as directed graphs as in Figure \ref{fig:2-sum}, with corresponding decomposition tree $T$. Arcs with the same color represent the common vector in the span of two incidence matrices. Given a node $t$ of $T$, $\mathrm{colsp}(A(t))$ is generated by the vectors corresponding to a spanning tree of the corresponding graph, and contains all possible arcs on the respective vertex set.}
\label{fig:2-sumtree}
\end{figure}

We give an example for such a decomposition tree in \Cref{fig:2-sumtree}.
It follows directly from \cite{cunningham_1980,seymour_1981} that every $2$-connected configuration of at least $3$ vectors admits a decomposition tree in which $A(t)$ is $3$-connected for each node $t \in V(T)$. 
Moreover, if $A$ is regular (in particular, when $A$ is a TU matrix) then we may choose the configurations $A(t)$ in such a way that $A(t)$ is regular for every node $t \in V(T)$. We will always assume this.



\subsection{The maximum constrained integer circulation problem} \label{sec:MCICP_bis}

Consider any (feasible) instance of \Cref{problemEquality}. Recall that by \Cref{propReductionIP}, we may reduce to the case where $b = \zero$ by performing a translation. We now define, for better reference, a new problem which is essentially the restriction of Problem \ref{problemEquality} to the instances that have $b = \zero$. This is the version we will need in our reduction. 

\begin{problem}[Maximum constrained integer circulation problem (MCICP)]
    \label{problemMCC}
    Let $k,\Delta \in \Z_{\ge 1}$ be constants.    
    Given $p \in \Z^n$, a regular configuration $A \in \{0,\pm 1\}^{m \times n}$, $W \in \Z^{k \times n}$ such that $\|Wc\|_\infty \le \Delta$ for all circuits $c \in \circuits(A)$, $d \in \Z^k$ and $\ell, u \in \Z^n$ such that $\ell \le u$, solve
    \begin{equation}
        \label{eqMCC}
        \max \left\{ p^\intercal x : Ax = \zero, \, Wx = d, \, \ell \le x \le u, \, x \in \Z^n \right\}\,.\tag{$\mathrm{IP}_2$}
    \end{equation}
Letting $I := (p,A,W,d,\ell,u)$, we define $\OPT(I)$ as the optimum value of \eqref{eqMCC}. If $I$ is infeasible, then $\OPT(I) = -\infty$.
\end{problem}

Here, we call a circulation $x \in \ker(A)$ \emph{feasible} if $\ell \le x \le u$. This terminology is meant to generalize to any regular configuration the usual notion of a circulation in a directed graph, which arises when $A$ is an incidence matrix. 

Recall that \Cref{propReductionIP} implies the following: \eqref{eqMCC} either is infeasible, or has an optimal solution $x$ that is the conformal sum of at most $f_{\ref{propReductionIP}}(k,\Delta)$ circuits of $A$. Also, thanks to the proof of \Cref{propReductionIP} (in particular see \eqref{eqMinFace}), we may assume that there is no \emph{feasible} $c \in \circuits(A)$ such that $Wc = \zero$ and $p^\intercal c > 0$. We will always assume this.

Before discussing $1$-sums and $2$-sums further, we establish some useful notation and terminology relative to the MCICP. For $x \in \R^n$, it will be convenient to denote by $x(v)$ the coordinate of $x$ corresponding to vector $v \in A$. For instance, using this notation, the equation $Ax = \zero$ becomes $\sum_{v \in A} x(v) \cdot v = \zero$. We will refer to $x(v)$ as the \emph{flow} along vector $v$. 

In the next section we show how to reduce any MCICP instance to a constant number of MCICP instances that are $2$-connected in the sense that their corresponding vector configuration is $2$-connected. The next reduction to the almost $3$-connected case is more involved, and is discussed in the section after the next.

\subsection{Reduction to the $2$-connected case} \label{sec:reduction_2-connected_bis}

Consider any MCICP instance $I = (p,A,W,d,\ell,u)$ that is not $2$-connected. After permuting its $n$ columns, $A$ decomposes as a $1$-sum $A = A_1 \oplus_1 \cdots \oplus_1 A_q$ of $2$-connected configurations, where $2 \le q \le n$. We extend the decomposition of $A$ to a corresponding decomposition of vectors $p$, $\ell$, $u$ and matrix $W$. For $j \in \{1,\ldots,q\}$, we write
$$
p_{\le j} := \begin{bmatrix}p_1\\ \vdots\\ p_j\end{bmatrix}, \quad W_{\le j} := \begin{bmatrix} W_1 & \cdots & W_j \end{bmatrix},
\quad \ell_{\le j} := \begin{bmatrix}\ell_1\\ \vdots\\ \ell_j\end{bmatrix} \quad \text{and} \quad u_{\le j} := \begin{bmatrix}u_1\\ \vdots\\ u_j\end{bmatrix}\,.
$$



We use a straightforward dynamic program (DP) to solve \eqref{eqMCC}. Let $B := [-\Delta f_{\ref{propReductionIP}}(k,\Delta), \Delta f_{\ref{propReductionIP}}(k,\Delta)]^k \cap \Z^k$ be an integer box containing all target vectors that can be attained by solutions that are sums of at most $f_{\ref{propReductionIP}}(k,\Delta)$ circuits. For each $j \in [q]$ and $d_{\le j} \in B$ we let 
$$
F_j(d_{\le j}) := \OPT ( p_{\le j}, A_1 \oplus_1 \cdots \oplus_1 A_j, W_{\le j}, d_{\le j}, \ell_{\le j}, u_{\le j})\,.
$$
These optimum values can be easily computed through the following DP equation, whose proof is left to the reader.

\begin{observation} \label{obs:DP_1-sum}
For all $j > 1$ and $d_{\le j} \in B$ we have
$$
F_j(d_{\le j}) = \max \left\{ F_{j-1}(d_{\le j-1}) + \OPT (p_j,A_j,W_j,d_j,\ell_j,u_j) : 
d_{\le j-1} \in B,\ d_j \in B,\ d_{\le j-1} + d_{j} = d_{\le j} \right\}\,.
$$
\end{observation}

Since $\OPT(I) = F_q(d)$, and since it is straightforward to extend the DP approach to also recover corresponding optimal solutions, \Cref{obs:DP_1-sum} implies the following result.

\begin{theorem}\label{thm:2-connected_red}
Let $k,\Delta \in \Z_{\ge 1}$ be constants. Any instance of the MCICP on a matrix with $n$ columns can be solved after solving at most $(2\Delta f_{\ref{propReductionIP}}(k,\Delta)+1)^k n$ instances of the MCICP that are $2$-connected and on matrices with at most $n$ columns, and performing extra work in time $(\Delta f_{\ref{propReductionIP}}(k,\Delta))^{O(k)} n$.
\end{theorem}

Note that the extra work in \Cref{thm:2-connected_red} refers to finding the aggregated maxima, and composing the final combined solution of all subproblems.
For the remainder of this section, we assume $A$ to be $2$-connected\footnote{Based on our definition of connectivity for regular configurations, a single column is also $2$-connected. We remark that due to our results in \Cref{sec:proximity}, we can efficiently solve problems on a bounded number of variables.}.
This can be done without loss of generality, due to \Cref{thm:2-connected_red}.

\subsection{Reduction to the almost $3$-connected case} \label{sec:reduction_3-connected_bis}

In this section we consider a variant of \Cref{problemMCC}, which we call the \emph{rooted} MCICP. This problem is defined exactly as \Cref{problemMCC}, with two additional input data: a \emph{root vector} $\overline{v} \in A$ and a prescribed flow value $\varphi \in \Z$ for $\overline{v}$. We let $\ell(\overline{v}) = u(\overline{v}) := \varphi$, which forces $x(\overline{v}) = \varphi$ for all feasible circulations $x$. 

It is easy to convert any MCICP instance into an equivalent rooted MCICP instance. For example, we may duplicate corresponding columns of $A$ and $W$, name the new columns $\overline{v}$ and let $\varphi := 0$ (by convention, we set all weights and profit of the column to $0$). On the other hand, the rooted MCICP is a just special case of MCICP, hence the two problems are actually equivalent. For convenience, we focus on the 
rooted MCICP throughout this section.

\begin{definition}[rooted MCICP instances for subtrees]\label{rootedMCICP}
Consider a rooted MCICP instance defined by $I = (p,A,W,d,\ell,u,\overline{v},\varphi)$, and consider a decomposition tree $T$ for $A$.
Recall that we can assume that $A(t)$ is regular and $3$-connected for all nodes $t \in V(T)$. 
There exists a unique \emph{root node} $r \in V(T)$ such that $\overline{v}$ is a column of $A(r)$. 
This turns $T$ into a rooted tree.
Let $B := [-\Delta f_{\ref{propReductionIP}}(k,\Delta), \Delta f_{\ref{propReductionIP}}(k,\Delta)]^k \cap \Z^k$. Also, let $\Phi := [-f_{\ref{propReductionIP}}(k,\Delta),f_{\ref{propReductionIP}}(k,\Delta)] \cap \Z$.

For $t \in V(T)$, we let $A_t$ denote the configuration obtained by performing all the $2$-sums corresponding to the edges of the subtree of $T$ rooted at $t$. Let $\overline{v}_t \in A_t$ denote the vector associated to the edge between $t$ and its parent. If $t$ is the root of $T$, then we let $\overline{v}_t := \overline{v}$. For $d_t \in B$ and $\varphi_t \in \Phi$ we let $I_t = I_t(d_t,\varphi_t)$ denote the rooted MCICP instance $I_t := (p_t,A_t,W_t,d_t,\ell_t,u_t,\overline{v}_t,\varphi_t)$, where $\overline{v}_t\notin A$ is the auxiliary vector associated to the edge between $t$ and its parent.
Further $p_t$, $\ell_t$ and $u_t$ are found by restricting the corresponding vectors in the original instance $I$ to the columns of $A_t\setminus\{\overline{v}_t\}$ and letting $p_t(\overline{v}_t) := 0$ and $\ell_t(\overline{v}_t) = u_t(\overline{v}_t) := \varphi_t$. $W_t$ is found by restricting $W$ to the columns of $A_t\setminus\{\overline{v}_t\}$ and defining $W_t(\cdot,\overline{v}_t)$ in such a way that $||W_t c_t||_\infty \leq \Delta$ for all circuits $c_t$ of $A_t$; this is possible by \Cref{lem:splitting_weight} below.
We call $I_t$ a \emph{rooted} instance at node $t$. We let
$$
F_t(d_t,\varphi_t) := \OPT(I_t) = \OPT(p_t,A_t,W_t,d_t,\ell_t,u_t,\overline{v}_t,\varphi_t)
$$
denote the optimum value of the rooted MCICP instance corresponding to node $t \in V(T)$ and to the choices of target vector $d_t \in B$ and prescribed flow value $\varphi_t \in \Phi$ for the root vector $\overline{v_t}$.
\end{definition}

\begin{lemma} \label{lem:splitting_weight}
    Let $t\in V(T)$.
    There exist integers $W_t(i,\overline{v}_t)$ for each $i \in [k]$ such that $||W_t c_t||_\infty \leq \Delta$ for all circuits $c_t$ of $A_t$.
\end{lemma}

\begin{proof}
If $t$ is the root of $T$, we simply let $W_t(i,\overline{v}_t) = W_t(i,\overline{v}) := 0$ for all $i \in [k]$. From now on, assume that $t$ is not the root of $T$. Let $t'$ denote the parent of $t$. Consider the configuration $A^t$ obtained by performing each $2$-sum corresponding to the edges of $T - tt'$ that are \emph{not} in the subtree rooted at $t$. We get the decomposition $A = A^t \oplus_2 A_t$. Let $\overline{c} \in \circuits(A)$ be any circuit crossing $\overline{v}_t$. By negating $\overline{c}$ if necessary, we may assume that $\overline{c} = \smallmat{\overline{c}_1\\ \overline{c}_2}$ with $\smallmat{\overline{c}_1\\ -1} \in \circuits(A^t)$ and $\smallmat{1\\ \overline{c}_2} \in \circuits(A_t)$. For each $i \in [k]$, we let
$$
W_t(i,\overline{v}_t) := \sum_{v \in A^t \atop v \neq \overline{v}_t} W(i,v) \overline{c}(v)
$$ 
denote the \emph{upper weight} of $\overline{c}$ with respect to the $i$-th row of $W$.

Now consider any circuit $c_t$ of $A_t$. If $c_t(\overline{v}_t) = 0$ then $c_t$ yields a circuit of $A$, say $c \in \circuits(A)$, such that $||W_t c_t||_\infty = ||W c||_\infty \leq \Delta$, by setting $c_v=0$ for all $v\notin A_t$.
Otherwise, we may assume that $c_t(\overline{v}_t) = 1$. Then $c_t$ combines with $\smallmat{\overline{c}_1\\ -1}$ into a circuit of $A$, say $c' \in \circuits(A)$. By our choice of $W_t(\cdot,\overline{v}_t)$, we get $||W_t c_t||_\infty = ||W c'||_\infty \leq \Delta$.
\end{proof}

If $t$ is a leaf of $T$ then we can compute $F_t(d_t,\varphi_t)$ for every fixed choices of $d_t \in B$ and $\varphi_t \in \Phi$ by solving a single $3$-connected MCICP instance. Now consider a node $t$ that is not a leaf, let $t_1$, \ldots, $t_q$ denote the children of $t$. Hence, $\overline{v}_{t_1}$, \ldots, $\overline{v}_{t_q}$ are the vectors of $A(t)$ respectively associated to the edges $tt_1$, \ldots, $tt_q$.

Consider a choice of $d^0_{t}, d_{t_1}, \ldots, d_{t_q} \in B$ and $\varphi_{t_1}, \ldots, \varphi_{t_q} \in \Phi$. Let $I^0_t = I^0_t(d^0_{t},\varphi_{t},-\varphi_{t_1},\ldots,-\varphi_{t_q})$ denote the rooted MCICP instance obtained from $I_t = I_t(d_t,\varphi_{t})$ by restricting to the sole vectors of $A(t)$, redefining the target vector to $d^0_{t}$ and setting both bounds for each $\overline{v}_{t_i}$ to $-\varphi_{t_i}$ for $i \in [q]$. Moreover, we define the weight $W^0_t(i,\overline{v}_{t_j}) \in \Z$ for each $i \in [k]$ and $j \in [q]$ in such a way that $||W^0_t c^0_t||_\infty \leq \Delta$ for all $c^0_t \in \circuits(A(t))$. This is done by using the ``lower weight'' of any circuit crossing $\overline{v}_{t_j}$ with respect to $W(i,\cdot)$, similarly as in the proof of \Cref{lem:splitting_weight}. 
To be precise, let $\overline c\in\circuits(A)$ be a circuit crossing $\overline{v}_{t_j}$, such that $\overline c = \smallmat{\overline c_1\\ \overline c_2}$ with $\smallmat{\overline c_1\\ -1} \in \circuits(A^{t_j})$ and $\smallmat{1\\ \overline c_2} \in \circuits(A_{t_j})$. 
Then, for each $i \in [k]$ and $j \in [q]$, we let 
\begin{equation*}
    W_t^0(i,v_{t_j}):=\sum_{v\in A_{t_j}\atop v\neq \overline{v}_{t_j}} W(i,v)\overline c(v).
\end{equation*}
We call $I^0_t$ a \emph{local} instance at node $t$. Next, let $F^0_t(d^0_{t},\varphi_{t},-\varphi_{t_1},\ldots,-\varphi_{t_q})$ denote the optimum value of instance $I^0_t$. The following result generalizing \Cref{obs:DP_1-sum} can be used to compute $F_t(d_t,\varphi_t)$ when the number of children of $t$ is small (say, logarithmic in $n$). 

\begin{lemma} \label{lem:2-sum_basic_DP}
With the above notation, we have
\begin{align}\label{eq:DP_2-sum}
F_t(d_t,\varphi_t) = \max \Bigg\{ F^0_t(d^0_{t},\varphi_{t},-\varphi_{t_1},\ldots,-\varphi_{t_q}) + \sum_{i=1}^q F_{t_i}(d_{t_i},\varphi_{t_i}) :\, & d^0_{t}, d_{t_1}, \ldots, d_{t_q} \in B,\ \varphi_{t_1},\ldots,\varphi_{t_q} \in \Phi,
\\
&d^0_{t} + \sum_{i=1}^q \left( d_{t_i} - \left(W_t(\cdot,\overline{v}_{t_i}) - W^0_t(\cdot,\overline{v}_{t_i})\right) \varphi_{t_i} \right) = d_t\nonumber\Bigg\}
\end{align}
for every $d_t \in B$ and $\varphi_t \in \Phi$.
\end{lemma}

\begin{proof} 
Let $x_t$ denote an optimum solution of $I_t = I_t(d_t,\varphi_t)$. Consider a conformal decomposition $x_t = c_{t,1} + \cdots + c_{t,\ell}$ as a sum of $\ell \le f_{\ref{propReductionIP}}(k,\Delta)$ circuits of $A_t$. We split $x_t$ into circulations $x^0_t$ in $A(t)$ and $x_{t_1}$, \ldots, $x_{t_q}$ in $A_{t_1}$, \ldots, $A_{t_q}$ respectively, by restricting to the corresponding space $A_{t_i}$ and setting the flow on $\overline v_{t_i}$ to the unique value such that $x_{t_i}$ is a circulation for all $i\in[q]$. Further, we let $\varphi_{t_i} := x_{t_i}(\overline{v}_{t_i}) = - x^0_t(\overline{v}_{t_i})$. Hence, $|\varphi_{t_i}| \le \ell \le f_{\ref{propReductionIP}}(k,\Delta)$ which implies $\varphi_{t_i} \in \Phi$. Next, let $d^0_t$ and $d_{t_1}$, \ldots, $d_{t_q}$ denote the weight vectors of $x^0_t$ and $x_{t_1}$, \ldots, $x_{t_q}$ respectively. That is, let $d^0_t := W^0_t x^0_t$ and $d_{t_i} := W_{t_i} x_{t_i}$ for $i \in [q]$ and $j \in [k]$. By our choice of $W^0_t(\cdot,\overline{v}_{t_i})$ and $W_{t_i}(\cdot,\overline{v}_{t_i})$ for $i \in [q]$, all these weight vectors are in $B$. The sum 
$d^0_{t} + d_{t_1} + \cdots + d_{t_q}$ and the weight vector of $x_t$ are related as follows:
$$
d^0_{t} + d_{t_1} + \cdots + d_{t_q} = W_t x_t + \sum_{i=1}^q (W_t(\cdot,\overline{v}_{t_i}) - W^0_t(\cdot,\overline{v}_{t_i})) \varphi_{t_i}\,.
$$
This shows that, in \eqref{eq:DP_2-sum}, the left hand side is at most the right hand side. 

In order to show the converse inequality, fix some valid choice of $d^0_{t}, d_{t_1}, \ldots, d_{t_q} \in B$ and $\varphi_{t_1},\ldots,\varphi_{t_q} \in \Phi$. Let $x^0_t$ be an optimal solution for $I_t^0(d^0_{t},\varphi_{t},-\varphi_{t_1},\ldots,-\varphi_{t_q})$ and, for $i \in [q]$, let $x_{t_i}$ be an optimal solution for $I_{t_i}(d_{t_i},\varphi_{t_i})$. By what precedes, these optimal solutions combine into a feasible solution $x_t$ for $I_t(d_t,\varphi_t)$, proving that the left hand side of \eqref{eq:DP_2-sum} is at least the right hand side.
\end{proof}

In case $q$ is large (super-logarithmic in $n$), \Cref{lem:2-sum_basic_DP} cannot be used as is since the number of cases to consider might be super-polynomial. For this reason, we have to delve deeper in the properties of the rooted MCICP instances arising from node $t$ and its children $t_1$, \ldots, $t_q$.

First, remark that if the weight matrix $W_{t_i}$ is zero for some child $t_i$, then the corresponding rooted MCICP instances $I_{t_i} = I_{t_i}(d_{t_i},\varphi_{t_i})$ are easy to solve since either $d_{t_i} = \zero$ and the $k$ complicating constraints $W_{t_i} x_{t_i} = d_{t_i}$ are trivially satisfied, or the instance is trivially infeasible. This holds more generally when $t_i$ is ``tame'' in the following sense.

\begin{definition}[tame and wild children]
We say that a child $t_i$ of $t$ is \emph{tame} if for all circuits $c, c'$ of $A_{t_i}$ such that $c(\overline{v}_{t_i}) = c'(\overline{v}_{t_i}) = 1$, we have $W_{t_i} c = W_{t_i} c'$. If $t_i$ is not tame, we say that it is \emph{wild}.
\end{definition}

In case $t_i$ is a tame child, as we show below, we can basically get rid of the whole subtree rooted at $t_i$ and modify the configuration $A(t)$ by adding a constant number of copies of $\overline{v}_{t_i}$ with carefully defined profits, see \Cref{lem:gadget}. This directly leads to an improvement in \Cref{lem:2-sum_basic_DP}, since we no longer have to guess $\varphi_{t_i}$ when $t_i$ is a tame child of $t$. A further crucial result that we prove below is that every circuit $c_t \in \circuits(A_t)$ crosses a bounded number of vectors $\overline{v}_{t_i}$ such that $t_i$ is wild, see \Cref{lem:crossing_wild}. Combining these two lemmas, we obtain an efficient algorithm for computing the optimal values $F_t(d_t,\varphi_t)$.

\subsection*{Tame children and gadgets}

We start with a discussion of weight vectors of tame nodes, and follow this with a discussion of gadgets. In order to do this, we need further notions. 

\begin{definition}[linear equivalence, standardized configuration]
Two configurations $A \in \R^{m \times n}$ and $A' \in \R^{m' \times n}$ are said to be \emph{linearly equivalent} if $\ker(A) = \ker(A')$, or equivalently if there exists a bijective linear map $\alpha : \colsp(A) \to \colsp(A')$ that maps the $j$th column of $A$ to the $j$th column of $A'$ for each $j \in [n]$. Let $A \in \R^{m \times n}$ be a rank-$r$  configuration. We say that $A$ is \emph{standardized} if $AP = \begin{bmatrix}D & \identity_r \end{bmatrix}$ for some $D \in \R^{r \times (n-r)}$, and some permutation matrix $P \in \{0,1\}^{n \times n}$.
\end{definition}

Note that any standardized configuration has full row rank.
It is easy to see that every configuration is linearly equivalent to a standardized configuration. This can be achieved by pivoting and deleting zero rows (see for instance \cite{oxley_2006}).

\begin{definition}[dual configuration, cocircuit]
Let $A = \begin{bmatrix}D & \identity_r \end{bmatrix} \in \R^{m \times n}$ denote a standardized rank-$r$ configuration. The \emph{dual} configuration of $A = \begin{bmatrix}D & \identity_r \end{bmatrix}$ is defined as $\begin{bmatrix}\identity_{n-r} & -D^\intercal \end{bmatrix}$. It is known that every configuration has the same connectivity as its dual configuration. A \emph{cocircuit} of A is a circuit of the dual configuration $\begin{bmatrix}\identity_{n-r} & -D^\intercal \end{bmatrix}$. We let
$$
\circuits^*(A) = \circuits^*\left(\begin{bmatrix}D & \identity_r \end{bmatrix}\right) := \circuits\left(\begin{bmatrix}\identity_{n-r} & -D^\intercal \end{bmatrix}\right)
$$
denote the set of cocircuits of $A$. 
\end{definition}

It is easy to check that every circuit of $A$ is orthogonal to every cocircuit of $A$. 

Now consider a regular configuration $A \in \R^{m \times n}$. If $A = \begin{bmatrix}D & \identity_r \end{bmatrix}$, then $D$ is a TU matrix, which implies that the dual configuration $\begin{bmatrix}\identity_{n-r} & -D^\intercal \end{bmatrix}$ is also regular. 

\begin{definition}[equivalence of weight vectors]
We say that two weight vectors $w_1, w_2 \in \Z^n$ are \emph{equivalent} if $w_1^\intercal c = w_2^\intercal c$ for all circuits $c \in \circuits(A)$. Notice that, in a MCICP instance, any row $W(i,\cdot)$ of the weight matrix can be replaced by any equivalent weight vector (transposed) as long as it has polynomial encoding length.
\end{definition}

A \emph{basis} of $A$ is a subconfiguration of $A$ forming a maximally linearly independent set. 

\begin{lemma} \label{lem:zero_on_basis}
Let $B$ denote a basis of $A$. For every weight vector $w_1 \in \Z^n$ there exists an equivalent weight function $w_2 \in \Z^n$ such that $w_2(v) = 0$ for all $v \in B$.
\end{lemma}

\begin{proof}
If $A$ has no cocircuits, then it has also no circuit, in which case the definition of equivalent weight vectors is vacuous and the thesis is trivially satisfied by the zero vector. Hence we can exclude this case.  Notice that if $c^* \in \circuits^*(A)$ is any cocircuit of $A$, then $w_1$ and $w_1 + c^*$ are equivalent. Every basis element $v \in B$ has a corresponding fundamental cocircuit $c^*_v \in \circuits^*(A)$ that uses $v$ and no other element of $B$. By the remark above, $w_1$ and $w_2 := w_1 - \sum_{v \in B} w_1(v) c^*_v$ are equivalent weight vectors. Moreover, $w_2(v) = 0$ for all $v \in B$, as required.
\end{proof}

\begin{lemma} \label{lem:tame_all_zero}
Let $A \in \R^{m \times n}$ be a $2$-connected regular configuration, and let $\overline{v} \in A$. Let $w \in \Z^{n}$ be a weight vector. If every circuit $c \in \circuits(A)$ using $\overline{v}$ has zero weight, then $w$ is equivalent to the zero weight function.
\end{lemma}

\begin{proof}
We claim that every circuit of $A$ has zero weight.

Before proving the claim, we verify that it implies the result. Assuming that the claim holds, pick a basis $B$ and consider any weight vector $w'$ equivalent to $w$ that is zero on all elements of $B$. By \Cref{lem:zero_on_basis}, such a weight function exists. Now consider any $v \in A \setminus B$, and let $c \in \circuits(A)$ denote the unique circuit of $A$ whose support is included in $B \cup \{v\}$ and that uses $v$ positively. We get
$$
w'(v) = \sum_{u \in \supp(c)} w'(u) c(u) = (w')^\intercal c = 0\,.
$$
We conclude that $w'$ is identically zero.

In order to establish the claim, let $c_1 \in \circuits(A)$ be arbitrary. If $c_1$ uses $\overline{v}$, then $w^\intercal c_1 = 0$ by hypothesis. 

From now on, assume that $c_1$ does not use $\overline{v}$, that is, $c_1(\overline{v}) = 0$. Since $A$ is $2$-connected, for each $v$ used by $c_1$, there exists a circuit $c_2 \in \circuits(A)$ using both $\overline{v}$ and $v$. Choose such a vector $v$ and circuit $c_2$ such that the union of the supports of $c_1$ and $c_2$ is inclusionwise minimal. By negating $c_1$ and/or $c_2$ if necessary, we may assume that $c_2(\overline{v}) = 1$, and $c_1(v) = -c_2(v)$.

Now consider the circulation $c_1 + c_2$, and any conformal decomposition $c_1 + c_2 = c_3 + c_4 + \cdots + c_\ell$ as a sum of circuits of $A$. After permuting the indices if necessary, we may assume that $c_3(\overline{v}) = 1$. Let $C_i$ denote the support of $c_i$, for $1 \le i \le \ell$. (The $C_i$'s are circuits in the matroid $M(A)$ represented by $A$.)

Since the decomposition is conformal and $c_1(v) + c_2(v) = 0$, $C_3 \subseteq (C_1 \cup C_2) - v$. In particular, $C_3$ is distinct from $C_2$. Since $C_2$ is a circuit, $C_3 \setminus C_2$ is nonempty. Hence, $C_3$ intersects $C_1$. By choice of $c_2$, we have $C_2 \setminus C_1 \subseteq C_3$. 

Using again the fact that the decomposition of $c_1 + c_2$ is conformal, we have $c_3(u) = c_2(u)$ for all $u \in C_2 \setminus C_1$. This implies that $C_i \subseteq C_1 - v$ for $4 \le i \le \ell$. We conclude that $\ell = 3$ and $c_1 + c_2 = c_3 \in \circuits(A)$. We get $w^\intercal{c_1} = w^\intercal c_3 - w^\intercal c_2 = 0$. The claim is proved.
\end{proof}

We resume the discussion of the rooted MCICP instance $I = (p,A,W,d,\ell,u,\overline{v},\varphi)$ and the corresponding decomposition tree $T$ for $A$. Consider a node $t \in V(T)$. We can resort to the proof of \Cref{lem:tame_all_zero} to check efficiently if $t$ is tame. Notice that being tame or wild does not depend on $W_{t}(\cdot,\overline{v}_t)$. We redefine $W_{t}(i,\overline{v}_t) \in \Z$ for each $i \in [k]$ in such a way that $W_{t} c_t = \zero$, for all $c_t \in \circuits(A_{t})$. By the proof of \Cref{lem:tame_all_zero}, $t$ is tame if and only if the weight vector of every circuit of $A_{t}$ is zero, that is, if and only if $A_{t} x = \zero$ implies $W_{t} x = \zero$ for all vectors $x$ of the appropriate dimension. This can be easily checked by considering any basis $B$ of $A_{t}$ and checking that every fundamental circuit has a zero weight vector relative to $W_{t}$.
Since the fundamental circuits of $A_{t}$ with respect to $B$ form a generating set of the cycle space of $A_{t}$, this is equivalent to checking that $W_{t} c = \zero$ for all circuits $c$ of $A_{t}$.

\begin{lemma} \label{lem:gadget}
Let $t \in V(T)$, and assume that the weight vector of every circuit of $A_t$ with respect to $W_t$ is zero. Consider a new tree $T'$ with a single node $t'$, letting $A(t')$ denote the vector configuration that is formed by the root vector $\overline{v}_t$, $f_{\ref{propReductionIP}}(k,\Delta)$ extra copies of $\overline{v}_t$ and $f_{\ref{propReductionIP}}(k,\Delta)$ further vectors that are all copies of $-\overline{v}_t$, $W_{t'} := \zero$, $\ell_{t'}(v) := 0$ and $u_{t'}(v) := 1$ for every $v \in A_{t'} - \overline{v}_t$. Then for some choice of profits $p_{t'}(v)$ for $v \in A_{t'} - \overline{v}_t$, we have $F_{t'}(d_t,\varphi_t) = F_{t}(d_t,\varphi_t)$ for all $d_t \in B$ and $\varphi_t \in \Phi$.
\end{lemma}

\begin{proof}
First, we notice that if $d_t \neq \zero$ then $F_{t'}(d_t,\varphi_t) = F_{t}(d_t,\varphi_t)= -\infty$ for all $\varphi_t \in \Phi$, hence we can fix $d_t=\zero$ for the rest of the proof.

Consider the function of one variable $f(\varphi_{t}) := F_{t}(\zero,\varphi_{t})$. Since $A_{t}$ is regular, we can omit the integrality condition in the rooted MCICP, and write
\[
f(\varphi_{t}) = \max \left\{ p_{t}^\intercal x : A_{t} x = \zero, x(\overline{v}_t)= \varphi_{t}, \, \ell_{t} \le x \le u_{t} \right\}\,.
\]

It is easy to see that $f$ is piecewise linear and concave. Indeed, given integers $\varphi$ and $\varphi'$, corresponding feasible circulations $x$ and $x'$ respectively attaining $f(\varphi)$ and $f(\varphi')$, and $\lambda \in [0,1]$ such that $\lambda \varphi + (1-\lambda) \varphi'$, we see that the convex combination $\lambda x + (1-\lambda) x'$ is a feasible circulation giving a lower bound on $f(\lambda \varphi + (1-\lambda) \varphi')$. 

Furthermore, $f(0) = 0$. Indeed, $f(0) \ge 0$ since $x = \zero$ is a feasible solution for $\varphi_{t} = 0$. Moreover, $f(0) > 0$ would imply the existence of a feasible circuit of $A$ with zero weight and positive profit.
Recall that we assume that such a circuit does not exist based on the proof of \Cref{propReductionIP}.

For the sake of simplicity, let $\Gamma = \Gamma(k,\ell) := f_{\ref{propReductionIP}}(k,\Delta)$. Label the copies of $\overline{v}_t$ in $A(t')$ as $v_{-1}$, \ldots, $v_{-\Gamma}$ and the copies of $-\overline{v}_t$ as $v_1$, \ldots, $v_{\Gamma}$. We let $p_{t'}(v_i) := f(i) - f(i-1)$ and $p_{t'}(v_{-i}) := f(-i) - f(-i+1)$ for $i = 1, \ldots, \Gamma$. Since $f$ is concave, we have $p_{t'}(v_1) \ge p_{t'}(v_2) \ge \cdots \ge p_{t'}(v_\Gamma)$ and $p_{t'}(v_{-1}) \ge p_{t'}(v_{-2}) \ge \cdots \ge p_{t'}(v_{-\Gamma})$. 

To show that $F_{t'}(\zero,\varphi_t) = f(\varphi_t) $ for $\varphi_t \in \Phi$, consider an optimal solution $x'$ of the rooted instance for node $t'$, and assume that $\varphi_t \geq 0$, the other case being similar. First, one checks that if $x'(v_{i})=1$ for any $i < 0$, there is some $j > 0$ with $x'(v_{j})=1$ and we can set both to zero without decreasing the objective value. Second, if there is some $0 < i < j$ with $x'(v_{i})=0$, $x'(v_{j})=1$, then we can swap the two values without decreasing the objective value. Hence, given that $x'(\bar v_t)=\varphi_t$ and $A_t' x'=\zero$, we conclude that $x'(v_{i})=1$ for $i=1,\dots,\varphi_t$ and $0$ for all other vectors. This implies that $F_{t'}(\zero,\varphi_t)$, the objective value of $x'$, is equal to $\sum_{i=1}^{\varphi_t} f(i)-f(i-1) = f(\varphi_t)$.
\end{proof}
 
Thanks to \Cref{lem:2-sum_basic_DP} and \Cref{lem:gadget}, we can replace the subtree rooted at any tame node of $V(T)$ by a single node whose corresponding configuration is a multiset of parallel vectors without changing any value function $F_t$ for the other nodes. This allows us to assume that every tame node is a leaf. Next, for each tame leaf node $t \in V(T)$ with parent $u \in V(T)$ we replace $A(u)$ by $A(u) \oplus_2 A_t =  A(u) \oplus_2 A(t)$ and delete node $t$ from the decomposition tree. In this way, we obtain a new rooted MCICP instance $I' = (p',A',W',d',\ell',u',\overline{v}',\varphi')$ and decomposition tree $T'$ of $A'$ that is a subtree of $T$ containing the root, in such a way that $T'$ has no tame node, $A(t)$ is \emph{almost} $3$-connected for every $t \in V(T)$, and the value function $F'_t$ with respect to $I'$ equals the value function $F_t$ with respect to the original instance $I$, for every node $t \in V(T')$. 

\subsection*{Final DP}

Consider again a rooted MCICP instance $I = (p,A,W,d,\ell,u,\overline{v},\varphi)$ and decomposition tree $T$ for $A$. Assume that no node of $T$ is tame, and $A(t)$ is almost $3$-connected for all nodes $t \in V(T)$. By the previous section, we may reduce to this case. Fix some non-leaf node $t \in V(T)$ and denote its children by $t_1$, \ldots, $t_q$. The next result is the final piece in our reduction of MCICP to the almost $3$-connected case.

\begin{lemma} \label{lem:crossing_wild}
Every circuit of $c \in \circuits(A_t)$ crosses at most $2 k \Delta$ vectors $\overline{v}_{t_i}$ such that node $t_i$ is a wild child of $t$.
\end{lemma}

\begin{proof}
After renumbering the children of $t$, we may assume that the indices $i \in [q]$ such that $c$ crosses $\overline{v}_{t_i}$ and $t_i$ is wild are $i = 1, \ldots, q'$ for some $q' \le q$. Then $c$ splits into one circuit $c_0$ of $A(t) \oplus_2 A_{t_{q'+1}} \oplus_2 \cdots \oplus_2 A_{t_{q}}$ and $q'$ circuits $c_1$, \ldots, $c_{q'}$ of $A_{t_1}$, \ldots, $A_{t_{q'}}$ respectively. For each $i \in [q']$, let $c'_i$ and $c''_i$ denote two circuits of $A_{t_i}$ such that $c'_i(\overline{v}_{t_i})= c''_i(\overline{v}_{t_i})=1$ and $W_{t_i} c'_i \neq W_{t_i} c''_i$.

Towards a contradiction, suppose that $q'> 2k\Delta$. By the pigeonhole principle, we can find a row index $j \in [k]$ and a set of $2\Delta +1$ indices $i \in [q']$ such that $W_{t_i}(j,\cdot) c'_i \neq W_{t_i}(j,\cdot) c''_i$. After permuting the indices again, we may assume that these indices are $i = 1, \dots, 2\Delta+1$. By exchanging $c'_i$ and $c''_i$ if necessary, we may assume that $W_{t_i}(j,\cdot) c'_i > W_{t_i}(j,\cdot) c''_i$ for all $i \in [2\Delta+1]$.

Notice that replacing $c_i$ for each $i \in [2\Delta + 1]$ with any other circuit $\tilde c_i \in \{\pm c'_i, \pm c''_i\}$ of $A_{t_i}$ such that $c_i(\overline{v}_{t_i}) = \tilde{c}_i(\overline{v}_{t_i})$ in the decomposition of $c$ results in a new circuit of $A_t$. Consider the circuit $c^+$ obtained from $c$ by replacing $c_i$ with $c'_i$ if $c_i(\overline{v}_{t_i}) = 1$, and with $-c''_i$ if $c_i(\overline{v}_{t_i}) = -1$, for $i \in [2\Delta+1]$. Consider the circuit $c^-$ obtained similarly by exchanging the roles of $c'_i$ and $c''_i$. In this way, $W_t(j,\cdot)c^+$ and $W_t(j,\cdot) c^-$ differ by at least $2\Delta+1$, contradicting the fact that both weights should be at most $\Delta$ in absolute value.
\end{proof}

This last result directly yields an improvement to \Cref{lem:2-sum_basic_DP}. Let $(d^0_{t}, d_{t_1}, \ldots, d_{t_q}, \varphi_{t_1},\ldots,\varphi_{t_q}) \in B^{q+1} \times \Phi^q$ be a \emph{guess} such that
$$
d^0_{t} + \sum_{i=1}^q \left( d_{t_i} - \left(W_t(\cdot,\overline{v}_{t_i}) - W^0_t(\cdot,\overline{v}_{t_i})\right) \varphi_{t_i} \right) = d_t\,,
$$
see \eqref{eq:DP_2-sum}. Recall that every MCICP instance has an optimal solution that is a sum of at most $f_{\ref{propReductionIP}}(k,\Delta)$ circuits. Together with \Cref{lem:crossing_wild}, this implies that we can restrict to the guesses such that
$$
\sum_{i=1}^q |\varphi_{t_i}| \le 2 k \Delta f_{\ref{propReductionIP}}(k,\Delta)\,.
$$
Moreover, we can further require that there are at most $f_{\ref{propReductionIP}}(k,\Delta)$ indices $i$ such that $\varphi_{t_i} = 0$ and $d_{t_i} \neq \zero$. Let $\Gamma_t$ denote the set of all guesses $(d^0_{t}, d_{t_1}, \ldots, d_{t_q}, \varphi_{t_1},\ldots,\varphi_{t_q}) \in B^{q+1} \times \Phi^q$ satisfying the three conditions above. Then we can change \eqref{eq:DP_2-sum} to
\begin{align*}
F_t(d_t,\varphi_t) = \max \Bigg\{ F^0_t(d^0_{t},\varphi_{t},-\varphi_{t_1},\ldots,-\varphi_{t_q}) + \sum_{i=1}^q F_{t_i}(d_{t_i},\varphi_{t_i}) :\, & (d^0_{t}, d_{t_1}, \ldots, d_{t_q}, \varphi_{t_1},\ldots,\varphi_{t_q}) \in \Gamma_t\Bigg\}
\end{align*}
where $d_t \in B$ and $\varphi_t \in \Phi$ are arbitrary. We point out that the DP can be once again extended to store an optimal solution for each table entry $F_t(d_t,\varphi_t)$. Putting everything together, we obtain the following result, which concludes this section.

\begin{theorem} \label{thm:2-sum}
Let $k,\Delta \in \Z_{\ge 1}$ be constants, and let $f_{\ref{thm:2-sum}}(k,\Delta) := k \Delta f_{\ref{propReductionIP}}(k,\Delta)$. Any $2$-connected instance of the MCICP on a matrix with $n$ columns can be solved after solving at most $(2\Delta f_{\ref{propReductionIP}}(k,\Delta)+1)^{O(kf_{\ref{thm:2-sum}}(k,\Delta))}n^{O(f_{\ref{thm:2-sum}}(k,\Delta))}$ instances of the MCICP that are almost $3$-connected and on matrices with at most $n$ columns, and performing extra work in time $(\Delta f_{\ref{propReductionIP}}(k,\Delta)n)^{O(kf_{\ref{thm:2-sum}}(k,\Delta))}$.
\end{theorem}
\begin{proof}
    It remains to prove that the reported number of instances is correct.

    Clearly $q\le n$, so the total number of integer values to guess is in $O(kn)$.
    Further, there are at most $O(f_{\ref{thm:2-sum}}(k,\Delta))$ non-zero values in $\phi$, summing up in absolute value to at most $O(f_{\ref{thm:2-sum}}(k,\Delta))$, which gives $n^{O(f_{\ref{thm:2-sum}}(k,\Delta))}$ options for this part of the vector.
    Each non-zero value in $\phi$ corresponds to a potential non-zero $k$-dimensional vector in $d$, giving us an additional factor of $(2\Delta f_{\ref{propReductionIP}}(k,\Delta)+1)^{O(kf_{\ref{thm:2-sum}}(k,\Delta))}$ guesses.
    Finally, there are at most $f_{\ref{propReductionIP}}(k,\Delta)$ entries where $\phi=0$ and $d\neq0$, leading to an additional factor of $n^{kf_{\ref{propReductionIP}}(k,\Delta)}(2\Delta f_{\ref{propReductionIP}}+1)^{kf_{\ref{propReductionIP}}(k,\Delta)}$.
    The total number of instances is thus bounded by 
    \begin{equation*}
        (2\Delta f_{\ref{propReductionIP}}(k,\Delta)+1)^{O(kf_{\ref{thm:2-sum}}(k,\Delta))}n^{O(f_{\ref{thm:2-sum}}(k,\Delta))}.
    \end{equation*}
\end{proof}


\section{Algorithm for the cographic case}\label{sec:cographic}

This section provides a strongly polynomial-time algorithm for solving cographic instances, relying on structural results that are established in \Cref{sec:structure}. 
In \Cref{sec:cographic_circuit}, we discuss circuit vectors of (co)-graphic instances in terms of the corresponding graphs. 
In \Cref{sec:cov}, we recall how to reformulate any cographic instance in the vertex space through a change of variables that transforms the main part of the coefficient matrix into the transpose of an incidence matrix. 
By this change of variables, the $k$ additional rows of the constraint matrix yield a $k$-dimensional weight vector for each vertex. 
Defining the set of roots as the set of vertices whose weight vector is nonzero, we turn the underlying graph into a rooted graph. 
In \Cref{sec:pumpkins}, we prove that these graphs do not have a rooted $K_{2,t}$-minor for $t := 4k\Delta + 1$. This in turn implies the strong structural properties stated in \Cref{SpecialDecomposition}, which are used in \Cref{sec:algorithm} to prove \Cref{thm:main_IP}.

\subsection{Circuits and connectivity}\label{sec:cographic_circuit}

Let $G$ be any directed graph. We allow parallel and anti-parallel directed edges, but do not allow loops. Seen as a vector configuration, the incidence matrix of $G$ associates to each directed edge $(v,w) \in E(G)$ the vector $\chi^v - \chi^w$ in $\R^{V(G)}$, where $\{\chi^v : v \in V(G)\}$ denotes the canonical basis of $\R^{V(G)}$. 

\begin{definition}[graphic and cographic configurations]
Let $A \in \R^{m \times n}$ be a configuration. We call $A$ \emph{graphic} whenever $A$ is linearly equivalent to the configuration defined by the incidence matrix of some directed graph $G$, and \emph{cographic} whenever $A$ is linearly equivalent to the dual of some graphic configuration.
\end{definition}

Up to a permutation of its columns, every rank-$r$ graphic configuration is linearly equivalent to a standardized configuration of the form $\begin{bmatrix}D & \identity_r \end{bmatrix}$, where $D \in \{0,\pm 1\}^{r \times (n-r)}$ is a network matrix. 
Conversely, every network matrix can be obtained in this way. 
Recall that the dual configuration of $\begin{bmatrix}D & \identity_r \end{bmatrix}$ is $\begin{bmatrix}\identity_{n-r} & -D^\intercal \end{bmatrix}$.
 
We remark that we may restrict to directed graphs $G$ that are (weakly) connected. Indeed, if $G$ has $\ell > 1$ connected components then we can pick any vertex $v_i$ in the $i$th component for each $i \in [\ell]$, and merge $v_1$, \ldots, $v_\ell$ in a single vertex. The resulting directed graph is weakly connected and the corresponding configuration is linearly equivalent to the original one.

In \Cref{lem:circuits_graphic_cographic} below, we give the well-known characterization of circuits of graphic and cographic configurations. Before doing this, we need more terminology regarding directed graphs (we generally follow~\cite{schrijver_2003}).
Note that the lemma follows from~\cite[Proposition~2.3.1]{oxley_2006}.

Let $G$ be a connected directed graph. For $X \subseteq V(G)$, we let $\delta(X)$ denote the set of edges with one end in $X$ and the other end in $\overline{X} := V(G) \setminus X$, $\delta^\mathrm{out}(X)$ denote the set of edges leaving $X$, and $\delta^\mathrm{in}(X)$ denote the set of edges entering $X$. Notice that $\delta(X) = \delta^\mathrm{out}(X) \cup \delta^\mathrm{in}(X)$. The set $\delta(X)$ is called an \emph{undirected cut}, while $\delta^\mathrm{out}(X)$ and $\delta^\mathrm{in}(X)$ are \emph{directed cuts}. If $X$ is a proper and nonempty subset of $V(G)$, we say that $\delta(X)$, $\delta^\mathrm{out}(X)$ and $\delta^\mathrm{in}(X)$ are \emph{nontrivial} cuts. It is easy to see that an inclusionwise minimal nontrivial undirected cut, sometimes called a \emph{bond}, is an undirected cut $\delta(X)$ such that each of $X$ and $\overline{X}$ induces a nonempty connected subgraph of $G$. To every such cut there correspond two opposite \emph{signed incidence vectors} in $\{0,\pm 1\}^{E(G)}$, namely, $\chi^{\delta^\mathrm{out}(X)} - \chi^{\delta^\mathrm{in}(X)}$ and $\chi^{\delta^\mathrm{in}(X)} - \chi^{\delta^\mathrm{out}(X)}$.

An \emph{undirected circuit} of $G$ is a cycle of its underlying undirected graph, that is, a sequence $C = (v_0,e_1,v_1,\ldots,e_k,v_k)$ where $k \ge 1$, $v_0$, $v_1$, \ldots, $v_k$ are vertices, $e_i$ is an edge incident to both $v_{i-1}$ and $v_i$ for each $i \in [k]$, $v_0 = v_k$, $v_1$, \ldots, $v_k$ are pairwise distinct. To every undirected circuit there corresponds a \emph{signed incidence vector} in $\{0,\pm 1\}^{E(G)}$, defined as $\chi^{C^+} - \chi^{C^-}$ where $C^+$ denotes set of edges of the circuit that are traversed forwards and $C^-$ the set of edges traversed backwards.

\begin{lemma} \label{lem:circuits_graphic_cographic}
Let $A \in \R^{m \times n}$ be a graphic configuration, and let $G$ be a corresponding connected directed graph, with $n = |E(G)|$. Then, the rank of $A$ equals $|V(G)| - 1$, the circuits of $A$ are the signed incidence vectors of undirected circuits of $G$, and the cocircuits of $A$ are the signed incidence vectors of minimal nontrivial undirected cuts of $G$.
\end{lemma}

The next lemma relates the connectivity of a graphic or cographic configuration to the (undirected) vertex connectivity of the corresponding graph, see \cite[Chapter~8]{oxley_2006}. Below, we say that directed graph $G$ is \emph{simple} if it contains no parallel or antiparallel directed edges.

\begin{lemma} \label{lem:connectivity} 
Let $A \in \R^{m \times n}$ be a rank-$r$ configuration that is graphic or cographic, and let $G$ be a corresponding (weakly) connected directed graph, with $r = |V(G)| - 1$ and $n = |E(G)|$. Assuming $r \ge 2$, $A$ is $2$-connected if and only if $G$ is $2$-connected. Assuming $n \ge 4$, $A$ is $3$-connected if and only if $G$ is $3$-connected and simple. Now assume that $A$ is a cographic configuration. Then $A$ is almost $3$-connected if and only if $G$ can be obtained from a $3$-connected simple directed graph by subdividing some of its edges.
\end{lemma}

\subsection{Change of variables}\label{sec:cov}

Given a (feasible) cographic instance of \Cref{problemEquality}, we can assume without loss of generality that $A = \begin{bmatrix}\identity_{n-r} & -D^\intercal \end{bmatrix}$, where $D \in \{0,\pm 1\}^{r \times (n-r)}$ is a network matrix. Next, using \Cref{propReductionIP} and applying a translation mapping $z$ to $\zero$, we reduce to an instance of \Cref{problemMCC} (MCICP), with the same matrix $A$ (regarded as a vector configuration, $A$ is a cographic configuration) and $b = \zero$. Finally, by the results of \Cref{sec:1sum2sum}, we further reduce this instance to a polynomial number of instances on almost $3$-connected, cographic configurations. We point out that given the network matrix $D$ (or its transpose $D^\intercal$), one can efficiently find a corresponding directed graph $G$, see for instance~\cite{tutte_1960}.

%

Recall from \Cref{sec:MCIPP} that by letting $x = M^\intercal y$ where $M = \begin{bmatrix} N & B\end{bmatrix}$ is the incidence matrix of $G$ with the row corresponding to some $v_0 \in V(G)$ removed and $B$ is a basis of $M$, we can rewrite the cographic instance of \Cref{problemMCC} at hand almost as an instance of the maximal constrained integer potential problem. More precisely, we transform \eqref{eqMCC} for $A = \begin{bmatrix}\identity_{n-r} & -D^\intercal \end{bmatrix}$ into \eqref{eq:IP1-cographic}. By the following lemma, we see that circuits of $A$ bijectively correspond to docsets of $G$.

\begin{lemma} \label{lem:circuits_cographic_are_docsets}
A vector $x \in \R^n \cong \R^{E(G)}$ is a circuit of $A = \begin{bmatrix} \identity_{n-r} & -D^\intercal \end{bmatrix}$ if and only if $x = \pm M^\intercal \chi^S$ where $S \subseteq V(G-v_0)$ is a nontrivial docset of $G$, and the incidence vector is taken in $\R^{V(G-v_0)} \cong \R^r$.
\end{lemma}

\begin{proof}
The map $\alpha : \ker(A) \to \R^r$ such that $\alpha(x) = B^{-\intercal} x_B$ for all $x = \smallmat{x_N\\ x_B} \in \ker(A)$ is an isomorphism whose inverse is given by $\alpha^{-1}(y) = M^\intercal y$ for all $y \in \R^r$. If $S$ is a nontrivial docset of $G$ with $v_0 \notin S$, then $\pm M^\intercal \chi^S = \pm \chi^{\delta^\mathrm{out}(S)} \mp \chi^{\delta^\mathrm{in}(S)}$, which are both circuits of $A$. Moreover, every circuit of $A$ can be obtained in this way.
\end{proof}


Recall also that the problem we will solve, namely the maximum constrained integer potential problem, can be obtained from \eqref{eq:IP1-cographic} by adding one extra column to $M^\intercal$ and one extra row to $y$, both corresponding to vertex $v_0$, in such a way that $M^\intercal$ becomes the transpose of the (full) incidence matrix of $G$. 
Notice that letting $y(v_0) := 0$ transforms back \eqref{eqMCIPP} into \eqref{eq:IP1-cographic}.

In order to ease the task of the reader, we state formally below the precise version of the MCIPP that we have to solve, including the condition on the weight of docsets. 

\begin{problem}\label{problemMCP}
    Let $k,\Delta \in \Z_{\ge 1}$ be constants.    
    Given a connected directed graph $G$, profit vector $p \in \Z^{V(G)}$ with $\sum_{v \in V(G)} p(v) = 0$, weight matrix $W \in \Z^{[k] \times V(G)}$ such that each row of $W$ sums up to $0$ and $\|W \chi^S\|_\infty \le \Delta$ for all docsets $S$, target vector $d \in \Z^k$, bounding vectors $\ell, u \in \Z^{E(G)}$, solve 
    \begin{equation}
        \label{eqMCP}
        \max \left\{p^\intercal y : \ell(v,w) \le y(v) - y(w) \le u(v,w) \, \forall (v,w) \in E(G), \, Wy = d, \, y \in \Z^{V(G)}\right\}\,.\tag{$\mathrm{IP}_3$}
    \end{equation}
\end{problem}

Using the results of \Cref{sec:1sum2sum} and \Cref{lem:connectivity}, we may assume without loss of generality that the input graph $G$ is a $3$-connected directed graph some of whose edges are subdivided.
In fact, we can also assume that no vertex of degree $2$ is a root.
This follows from \Cref{lem:cov-weights} and the fact that the edges $e$ incident to vertices of degree $2$ have $W(\cdot,e)=\zero$, see \Cref{lem:gadget}.
Moreover, we can still argue that \eqref{eqMCP} either is infeasible, or has an optimal solution $y$ that is a sum of at most $f_{\ref{propReductionIP}}(k,\Delta)$ incidence vectors of docsets of $G$. 
This follows from \Cref{propReductionIP} and \Cref{lem:circuits_cographic_are_docsets}.
Notice that for any solution $y$ to \Cref{problemMCP}, $y':=y + \one$ is a solution to \Cref{problemMCP} with $p^\intercal y'=p^\intercal y$.

The following lemma follows directly from the definition of the incidence matrix, and gives a concrete characterization of the profit vector $p$ and weight matrix $W$ in \Cref{problemMCP} after our change of variables.

\begin{lemma}\label{lem:cov-weights}
    Let $M\in\{0,\pm1\}^{V(G)\times E(G)}$ be the incidence matrix of a directed graph $G$ and $w\in\Z^{E(G)}$ a weight vector.
    Then for each $v\in V(G)$, 
    \begin{equation*}
        (M^\intercal w)(v)=\sum_{e\in\delta^\mathrm{out}(v)}w(e)-\sum_{e\in\delta^\mathrm{in}(v)}w(e).        
    \end{equation*}
\end{lemma}

\subsection{Rooted $K_{2,t}$-models}\label{sec:pumpkins}

In this section we consider (vertex)-weighted undirected graphs $(G,a)$ without a docset of large weight, and study related properties of rooted $K_{2,t}$-minors in such graphs, where the roots are the vertices $v$ with $a(v) \neq 0$.
Note that the definition of docsets extends naturally to undirected graphs by requiring that for a set of vertices $S\subseteq V(G)$ both $G[S]$ and $G[\overline{S}]$ are connected.
We use the following notation.

\begin{definition}[maximum weight of a docset $\beta(G,a)$]
    Let $G$ be a connected graph and $a \in \Z^{V(G)}$ be vertex weights such that $\sum_{v \in V(G)} a(v) = 0$. For $S \subseteq V(G)$, we let $a(S) := \sum_{v \in S} a(v)$. We define 
    \[
    \beta(G,a) := \max \left\{ a(S) : S \subseteq V(G),\  S \text{ is a docset of } G\right\}\,.
    \]
\end{definition}

Notice that $a\big(\overline{S}\big) = -a(S)$, hence $\beta(G,a) \ge 0$. If $G$ is $2$-connected, then $\{v\}$ is a docset for every vertex $v$, which implies $\beta(G,a) \ge \max \left\{|a(v)| : v \in V(G)\right\}$. Every instance of \Cref{problemMCP} has $\beta(G,a) \leq \Delta$ for all weight vectors $a$ which are (the transpose of) a row of $W$.

The main result we prove is that no weighted graph $(G,a)$ with $\beta(G,a) \leq \Delta$ can contain a rooted $K_{2,t}$-minor for $t > 4\Delta$, taking as roots the vertices $v$ with $a(v) \neq 0$. 
We need the following technical lemma. 
Below, we use $N(X)$ to denote the open neighborhood of $X \subseteq V(G)$ in $G$, that is, we let $N(X) := \{v \in V(G-X) : \exists w \in X : vw \in E(G)\}$. 
For $X \subseteq V(G)$ and $v \in V(G)$ we denote by $X + v$ the set $X \cup \{v\}$.

\begin{lemma}\label{lemma:addition_bis}
    Let $G$ be a $2$-connected graph and let $X$, $Y$ be two docsets of $G$. If $X \subsetneq Y$, then there exists some vertex $v \in Y \setminus X$, such that $X + v$ is a docset.
\end{lemma}
\begin{proof}
    Let $u \in \overline{Y}$ be arbitrary. 
    Since $Y$ is connected and $X$ is a proper subset of $Y$, $N(X) \cap Y$ is nonempty. 
    Let $v$ denote any vertex of $N(X) \cap Y$ whose distance from $u$ in $G[\overline{X}] = G-X$ is maximum. 
    Clearly, $G[X + v]$ is connected. We claim that $G[\overline{X + v}]$ is also connected. 
    
    For the sake of contradiction, assume that $G[\overline{X+v}]$ has several (connected) components. Since $G[\overline{Y}]$ is connected, it is contained in a unique component of $G[\overline{X+v}]$, say $K$. Let $w$ be any vertex of $N(X) \cap Y$ distinct from $v$. By choice of $v$, no shortest path from $u$ to $w$ in $G[\overline{X}]$ contains $v$. This proves that $w \in K$. 
    
    Let $K'$ be any component of $G[\overline{X+v}]$ distinct from $K$. Observe that $V(K') \subseteq Y$. Since $K'$ is a component of $G[\overline{X+v}]$, we have $N(K') \subseteq X+v$. Since $G$ is $2$-connected, $K'$ should have at least two neighbors, and at least one distinct from $v$. But this implies that $K'$ contains at least one vertex $w \in N(X)\cap Y$ distinct from $v$. By the argument in the previous paragraph, $w\in K$, hence we get a contradiction to the assumption that $K$ and $K'$ are different connected components in $G[\overline{X + v}]$. Hence, the claim is proved and $X + v$ is a docset, as required.
\end{proof}


\begin{definition}[model]
    Given a graph $H$, an {\em $H$-model $M$} in a graph $G$ consists of one connected vertex subset $M(v) \subseteq V(G)$ for each vertex $v\in V(H)$, called a {\em branch set}, and one edge $M(uv) \in E(G)$ with one endpoint in $M(u)$ and the other in $M(v)$ for each edge $uv \in E(H)$ such that all branch sets are pairwise vertex disjoint.  
Note that $G$ contains $H$ as a minor if and only if $G$ has an $H$-model. 
\end{definition}

We stress that in this section and Section \ref{sec:structure} we use $M$ to denote an $H$-model, whereas previously we used $M$ for a matrix. The meaning of $M$ will always be clear from context.

\begin{definition}[rooted $K_{2,t}$-model]
    Given a rooted graph $(G, R)$, a {\em rooted $K_{2,t}$-model} in $(G, R)$ is a model of $K_{2,t}$ in $G$ such that each of the $t$ branch sets corresponding to vertices on the ``$t$ side'' of $K_{2,t}$ contains a vertex from $R$. 
    (Such branch sets will be called \emph{central branch sets}.)
    Note that $(G, R)$ contains a rooted $K_{2,t}$-minor if and only if $(G, R)$ contains a rooted $K_{2,t}$-model. 
\end{definition}

In the proof below, we say that two vertex subsets $X, Y$ of a graph $G$ \emph{touch} if their distance in $G$ is at most $1$.

\begin{lemma}\label{lemma:pumpkin_bis}
    Let $(G,a)$ be a $2$-connected weighted graph, where $a \in \Z^{V(G)}$ satisfies $\sum_{v \in V(G)} a(v) = 0$. 
    Let $R = \{v\in V(G): a(v) \neq 0\}$. 
    If $(G, R)$ contains a rooted $K_{2,t}$-model, then $\beta(G,a)\ge t/4$. 
\end{lemma}

\begin{proof}
    Notice that we may assume that $t\ge 4$ without loss of generality. 
    Notice further that any rooted $K_{2,t}$-model 
    in $(G, R)$ that avoids some vertex can be modified in order to cover the whole vertex set of $G$. 
    This follows easily from the connectivity of $G$. 
    We call such models \emph{full}. 
    Let $M$ denote a full rooted $K_{2,t}$-model 
    in $(G, R)$ that maximizes the number of nonzero weight central branch sets. 
    In other words, letting $V(K_{2,t}) := \{-1,0\} \cup \{1,\ldots,t\}$, we assume that $M$ maximizes 
    $$
        f(M) := |\{i \in [t] : a(M(i)) \neq 0\}|
    $$
    among all full rooted $K_{2,t}$-models in $(G, R)$.

    We claim that $f(M) \ge t/2$, that is, at least $t/2$ of the central branch sets have nonzero weight.
		
    For notational convenience, we let $B_i := M(i)$ (where $i \in \{-1,0\} \cup [t]$) denote the branch sets of $M$. Assuming that the claim holds, consider the docsets
    $$
    X^+ := B_{-1} \cup \bigcup_{i \in [t] : a(B_i) > 0} B_i \quad \text{and} \quad 
    X^- := B_{-1} \cup \bigcup_{i \in [t] : a(B_i) < 0} B_i\,.
    $$
    Notice that the claim implies $a(X^+) - a(X^-) \ge t/2$. Hence, $a(X^+) \ge t/4$ or $a(X^-) \le -t/4$. This implies that $\beta(G,a) \ge t/4$.

    Consider any central branch set $B_i$ of zero weight. 
    Since $M$ is full, $B_i$ is a docset. 
    We claim that $B_i$ has a partition into docsets $U_i, V_i$, which both are of nonzero weight. 
    Let $u_i \in B_i$ denote any vertex adjacent to some vertex of $B_{-1}$. 
    By applying iteratively \Cref{lemma:addition_bis} starting from docsets $\{u_i\}$ and $B_i$, we can find a docset $W_i \subsetneq B_i$ that contains $u_i$ and a unique root. 
    Thus, both $a(W_i) \neq 0$ and $a(B_i \setminus W_i) \neq 0$. Observe that $G[B_i \setminus W_i]$ decomposes into connected components, at least one of which has nonzero weight, which we denote by $V_i$. 
    Let $U_i := B_i \setminus V_i$. 
    By construction, $U_i$ and $V_i$ induce connected subgraphs and $a(V_i) \neq 0$ and thus $a(U_i) \neq 0$. 
    Observe that also $\overline{U_i}$ and $\overline{V_i}$ induce connected subgraphs. 
    Indeed, the set $\overline{V_i} \cap B_i$ is $U_i$ and therefore connected, the set $U_i$ contains $W_i$ and therefore is touching $B_{-1}$. 
    The set $\overline{U_i}$ induces a connected subgraph because $G[\overline{W_i}]$ is connected and hence $V_i$ is touching at least one $B_j$ for $j\neq i$ (note that $j$ can be also an index of a central branch set different from $B_i$).
    Thus, $U_i$ and $V_i$ are docsets.

    By maximality of $M$, $V_i$ touches neither $B_{-1}$ nor $B_0$ and hence $U_i$ touches both $B_{-1}$ and $B_0$. Otherwise for some $j \in \{-1,0\}$ we could replace $B_i$ by $U_i$ and $B_j$ by $B_j \cup V_i$ and strictly increase $f(M)$.
    
    Recall that $U_i$ is a docset, hence $V_i$ touches $B_j$ for some $j \in [t]$ different from $i$. If $a(V_i) + a(B_j) \neq 0$, we could merge $B_j$ and $V_i$ and contradict the maximality of $M$. In particular, all the central branch sets $B_j$ with $j \neq i$ touched by $V_i$ have nonzero weight.
      
    Assume that there is an index $i'\in [t]$ with $i'\neq i$, such that $a(B_i) = a(B_{i'}) = 0$ and $V_i$, $V_{i'}$ touch the same central branch set $B_j$ with $a(B_j) \neq 0$. By the argument given above, we have $a(V_i) + a(B_j) = 0$. This implies that $a(V_i) + a(V_{i'}) + a(B_j) = a(V_{i'}) \neq 0$, in which case we can move all the vertices of $V_i$ and $V_{i'}$ inside $B_j$ to strictly increase $f(M)$, which contradicts again the maximality of $M$.
    
    It follows that the number of central branch sets $B_j$ with nonzero weight are at least as numerous as the central branch sets $B_i$ with zero weight. 
    That is, $f(M) \ge t/2$. 
    This proves the claim. 

    Using the fact that $f(M) \ge t/2$ and that the model $M$ is full, it is easily seen that $\beta(G,a)\ge t/4$. The lemma follows.     
\end{proof}

We remark that the bound of $t/4$ in \Cref{lemma:pumpkin_bis} is tight, as shown in~\cite[Remark~4.39]{Kob23}.

\Cref{lem:4kDelta+1} from \Cref{sec:outline} follows directly from the previous lemma.

\begin{proof}[Proof of \Cref{lem:4kDelta+1}]
    Note that the set of roots $R$ in $G$ is defined as in \Cref{lemma:pumpkin_bis}.
    Assume to the contrary that $(G,R)$ contains a rooted $K_{2,t}$-minor, and consider a corresponding rooted $K_{2,t}$-model. 
    By the pigeonhole principle, there exists a row $w_i$, $i\in[k]$ of $W$, such that at least $4\Delta+1$ central branch sets of the model contain a vertex with a non-zero entry in $w_i$.
    By \Cref{lemma:pumpkin_bis}, this implies that $\beta(G,w_i)\ge\Delta+1$, a contradiction.
\end{proof}

We conclude with two simple results describing how the size of the largest rooted $K_{2,t}$-model in a graph can change by adding or contracting an edge.
\begin{lemma}\label{adding_an_edge}
    Let $(G,R)$ be a rooted graph, and let $e = uv$ where $u, v \in V(G)$. If $(G+e,R)$ contains a rooted $K_{2,t}$-model (with $t \ge 2$), then $(G, R)$ contains a rooted $K_{2,\lceil t/2 \rceil}$-model.
\end{lemma}

\begin{proof}
    Let $M$ be a rooted $K_{2,t}$-model in $(G+e, R)$. If the vertices $u$ and $v$ are contained in different branch sets of the model, then $(G, R)$ contains a rooted $K_{2,t-1}$-model, and hence a rooted $K_{2,\lceil t/2 \rceil}$-model since $t \ge 2$. If $e$ is contained in one of the central branch sets $M(i)$ (with $i \in [t]$), then again $(G, R)$ contains a rooted $K_{2,t-1}$-model. Finally, suppose that $e$ is contained in one of the branch sets $M(-1)$ or $M(0)$. Without loss of generality, assume that $e$ is contained in $M(-1)$. Then $G[M(-1)]$ has at most two connected components and therefore $(G, R)$ contains a rooted $K_{2,\lceil t/2 \rceil}$-model. 
\end{proof}

The following lemma is straightforward. Below, we let $(G, R) \contract e$ denote the rooted graph obtained from $(G,R)$ by contracting edge $e = uv \in E(G)$: if none of $u$ or $v$ is a root then we keep the same roots, otherwise we remove $u$ and/or $v$ from the set of roots and replace them by the vertex arising from the contraction of $e$.

\begin{lemma}\label{contraction}
    Let $(G,R)$ be a rooted graph, and let $e \in E(G)$. If $(G, R) \contract e$ contains a rooted $K_{2,t}$-model, then $(G,R)$ also contains a rooted $K_{2,t}$-model.
\end{lemma}
 
\subsection{The algorithm}\label{sec:algorithm}
In this section, we describe how to solve instances of \Cref{problemMCP}, assuming the structural results from \Cref{sec:structure}.
We remark that the dynamic programming (DP) algorithm in this section is an adaptation of the DP algorithm in \Cref{sec:reduction_3-connected_bis}.
Here, we consider local instances for each bag of the tree-decomposition and each possible way a superposition of $f_{\ref{propReductionIP}}(k,\Delta)$ docsets of $G$ can intersect the roots within the bag, and then solving an IP with a TU constraint matrix for each fixed choice. The optimal solutions to the local instances are then composed recursively to an optimal solution of the whole instance.

\subsection*{Structural foundations}
We start by recalling the relevant definitions and result from \Cref{sec:outline} that underlie our approach.

\begin{definition}[tree-decomposition]\label{def:td}
A {\em tree-decomposition} of a graph $G$ is a pair $(T,\mathcal{B})$ where $T$ is a tree and $\mathcal{B}=\{B_t : t \in V(T)\}$ is a collection of subsets of $V(G)$, called \emph{bags}. Moreover, we require the following conditions:
\begin{enumerate}
    \item $\bigcup_{t \in V(T)} B_t = V(G)$,
    \item \label{def:td_edge} for every edge $e \in E(G)$, there is $t \in V(T)$, such that $e \subseteq B_t$,
    \item \label{def:td_connected} let $t, t', t''\in V(T)$ be such that $t'$ lies on the path between $t$ and $t''$ in $T$, then $B_t \cap B_{t''}\subseteq B_{t'}$. 
\end{enumerate}
\end{definition}

We always assume that the decomposition tree \(T\) is rooted at one of its nodes, called the \emph{root node} of $T$ (not to be confused with the root vertices of $G$).

\begin{definition}[$\ell$-special tree-decomposition]
Given a graph $G$ and an integer $\ell \ge 1$, we say that a tree-decomposition $(T, \mathcal{B})$ is $\ell$-special if
\begin{enumerate}
    \item for each $tt' \in E(T)$ we have $|B_t \cap B_{t'}| \le \ell$, and
    \item every node $t \in V(T)$ has at most $\ell$ children.
\end{enumerate}
\end{definition}

For any given node $t\in V(T)$, we define $T_t$ as the subtree of $T$ rooted at $t$, and denote by $G_t$ the subgraph of $G$ induced by $\bigcup_{u\in T_t} B_u$. If $t$ is not the root node, let $t'$ denote its parent. We define the {\em adhesion} along the edge $e = tt'$ as the intersection of the corresponding bags and let $\adh{t} := B_t\cap B_{t'}$. If $t$ is the root node, we let $\adh{t} := \varnothing$.


Our final definition concerns the possible intersections of docsets of $G$ with a subset of the roots.

\begin{definition}[docset profiles and superprofiles]
Let $\mathcal{S}(G)$ denote the collection of all docsets in $G$.
The \emph{docset profile} of a subset $R' \subseteq R$ is the collection of sets $\mathcal{P}(G,R'):=\{R'\cap S :  S\in \mathcal{S}(G)\}$. 
A \emph{docset superprofile} of $R'$ is a collection of subsets of $R'$ that contains the docset profile $\mathcal{P}(G,R')$.
\end{definition}



Consider an instance of \Cref{problemMCP} defined by a tuple $(p,\ell,u,G,W,d)$.
Recall that \Cref{SpecialDecomposition} (see \Cref{sec:outline}) states that, for each fixed $s \in \Z_{\ge 1}$, we can efficiently obtain an $f_{\ref{SpecialDecomposition}}(s)$-special tree-decomposition of a $3$-connected rooted graph $(G,R)$ with no rooted $K_{2,s}$-minor, together with a docset superprofile of $R_t := R \cap B_t$ for each $t \in V(T)$. 
We prove this in \Cref{sec:structure} and extend the result to subdivisions of $3$-connected graphs without roots of degree $2$, see \Cref{SpecialDecompositionExtended}.


\subsection*{The local instances}

We proceed to define the {\em local docset completion problem} that our algorithm will solve within the DP, see \Cref{def:LDCP} below.

Recall that since $G$ is $2$-connected, $\{v\}$ is a docset for every $v \in V(G)$, implying that $\|W\|_\infty \le \Delta$. Hence, if $y \in \Z^{V(G)}$ is the sum of (at most) $f_{\ref{propReductionIP}}(k,\Delta)$ incidence vectors of docsets, and $y' \in \Z^{V(G)}$ is obtained from $y$ by zeroing an arbitrary set of coordinates, we have $\|W y'\|_\infty\le n\Delta f_{\ref{propReductionIP}}(k,\Delta)$. In view of this, we let $\mathcal{D} := [-n \Delta f_{\ref{propReductionIP}}(k,\Delta), n \Delta f_{\ref{propReductionIP}}(k,\Delta)]^k \cap \Z^k$.
Thus $\mathcal{D}$ is a polynomial-size set containing the weight vector of any partial solution we consider in the context of the DP. Finally, we let $\varPhi:=[0, f_{\ref{propReductionIP}}(k,\Delta)] \cap \Z$. Every time we need to guess the value of a variable $y(v)$, we take the guess in $\varPhi$. \Cref{propReductionIP} and \Cref{lem:circuits_cographic_are_docsets} imply that this is safe. Observe that $\varPhi$ is a constant-size set.

\begin{definition}[local docset completion problem]\label{def:LDCP}
    Let $k,\Delta\in\Z_{\ge1}$ be constants, let $t\in V(T)$ be a fixed node of the decomposition tree and let $t_1$, \ldots, $t_q$ denote the children of $t$ (in case $t$ is a leaf, $q = 0$). Let $\ell_t$ and $u_t$ denote the respective restrictions of $\ell$ and $u$ to $E(G[B_t])$. Let $p_t$ and $W_t$ denote the restrictions of $p$ and $W$ to $B_t$, redefining $p_t(v) = W_t(i,v) := 0$ for all $v \in \adh{t}$ and all $i \in [k]$. Finally, let $Y_t := \adh{t} \cup \bigcup_{i \in [q]} \adh{t_i} \cup R_t$. The {\em local docset completion problem} with respect to node $t$ is as follows: given $\varphi_t \in \varPhi^{Y_t}$, solve
    \begin{equation}\label{eq:LDCP}
        \max\left\{p_t^\intercal y:\ell_t(v,w)\le y(v)-y(w)\le u_t(v,w)\ \forall (v,w) \in E(G[B_t]),\ y(v)=\varphi_t(v)\ \forall v\in Y_t,\ y\in \Z^{B_t}\right\}\,.
    \end{equation}
    We denote an instance of this problem by $I^0_t = I^0_t(\varphi_t)$, and define $F_t^0(\varphi_t):=\OPT(I^0_t)$ as the optimum value of~\eqref{eq:LDCP} (if $I^0_t$ is infeasible, then $\OPT(I^0_t) =-\infty$).
    For each feasible instance $I^0_t = I^0_t(\varphi_t)$, we let $y_t^0(\varphi_t) \in \Z^{B_t}$ denote any optimal solution of $I^0_t$. We call this solution an \emph{optimal local solution}.
\end{definition}

Notice that for every $t \in V(T)$ and $\varphi_t \in \varPhi^{Y_t}$, we can compute $F_t^0(\varphi_t)$ as well as an optimal local solution $y_t^0(\varphi_t)$ (in case the instance $I_0^t(\varphi_t)$ is feasible) in strongly polynomial time, since the constraint matrix of \eqref{eq:LDCP} is TU. Given a polynomial-size docset superprofile $\mathcal{X}_t$ for node $t$, we can enumerate all relevant guesses $\varphi_t \in \varPhi^{Y_t}$ in polynomial time and consider them one after the other. In this way, we obtain a complete set of local optimal solutions for each $t \in V(T)$.

Moreover, for every optimal solution $y^*$ of \eqref{eqMCP} and every $t \in V(T)$, it should be clear that the restriction of $y^*$ to $B_t$ is an optimal solution of the local docset completion problem in which $\varphi_t$ is the restriction of $y^*$ to $Y_t$. In other words, we can think of any optimal solution $y^*$ of \eqref{eqMCP} as being obtained by composing\footnote{Let $V$ be any finite set and $V_1, \ldots, V_\ell$ be subsets of $V$. We say that a vector $y \in \R^{V}$ is obtained by \emph{composing} vectors $y_1 \in \R^{V_1}, \ldots, y_\ell \in \R^{V_
\ell}$ if the restriction of $y$ to each $V_i$ is exactly $y_i$. In particular, this implies that the vectors $y_i$, $i \in [\ell]$ agree on their common coordinates.} optimal local solutions $y_t^0(\varphi_t)$ for some choice of $\varphi_t$ for $t \in V(T)$. Our DP finds an optimal way to compose the optimal local solutions in a recursive way, starting at the leaves of the decomposition tree and going up to the root. We formalize this below.

\subsection*{The dynamic program}

Our DP is based on the concept of a \emph{rooted solution}, which is a vector $y_t \in \Z^{V(G_t)}$ obtained by composing (optimal) local solutions $y^0_{t'}(\varphi_{t'})$ where $t'$ is a descendant of some fixed node $t$. Given a guess $\varphi_t^{\text{up}} \in \varPhi^{\adh{t}}$ on the variables corresponding to the vertices in the adhesion to the parent, and a target $d_t \in \mathcal{D}$ for the weight vector of the rooted solution, the DP seeks a maximum profit rooted solution $y_t$ that \emph{complies} with $\varphi_t^{\text{up}}$ and $d_t$, in the sense that $y_t(v) = \varphi_t^{\text{up}}(v)$ for all $v \in \adh{t}$ and the weight vector of $y_t$ with respect to the vertices of $V(G_t - \adh{t})$ is coordinate-wise equal to $d_t$.

\begin{definition}[DP for composing optimal local solutions]\label{def:DP_MCP}
    Let $t\in V(T)$ be a node of the rooted decomposition tree, and let $t_1$, \ldots, $t_q$ denote the children of $t$ (where $q = 0$ whenever $t$ is a leaf). Let $\mathcal{X}_t$ be a polynomial-size docset superprofile with respect to $R_t = R \cap B_t$.
    For any tuple $\mathfrak{S}=(S_i)_{i\in[f_{\ref{propReductionIP}}(k,\Delta)]}\in\mathcal{X}_t^{f_{\ref{propReductionIP}}(k,\Delta)}$, we let $\sigma^{\mathfrak{S}}:=\sum_{i=1}^{f_{\ref{propReductionIP}}(k,\Delta)}\chi^{S_i} \in \varPhi^{B_t}$ denote the sum of the corresponding incidence vectors.
    For any given $\varphi_t^{\text{up}}\in\varPhi^{\adh{t}}$ and $d_t\in \mathcal{D}$, we define



    \begin{align*}
        F_t(\varphi_t^{\text{up}},d_t):=\max\bigg\{F_t^0(\varphi_t)+\sum_{i\in[q]}F_{t_i}(\varphi_{t_i}^{up},d_{t_i}) :\ 
        & \varphi_t\in\varPhi^{Y_t},\ \varphi_{t_i}^{\text{up}} \in \varPhi^{\adh{t_i}}\ \forall i\in[q],\ d_{t_i}\in \mathcal{D}\ \forall i\in[q], \\[-1.5ex]
        &\varphi_t(v) = \varphi_t^{\text{up}}(v)\ \forall v \in \adh{t},\ \varphi_{t_i}^{\text{up}}(v) = \varphi_t(v)\ \forall v\in\adh{t_i}\ \forall i\in[q],\\
        &\exists \mathfrak{S} \in \mathcal{X}_t^{f_{\ref{propReductionIP}}(k,\Delta)} : \varphi_t(v)=\sigma^{\mathfrak{S}}(v)\ \forall v\in R_t,\  \\
        &W_t \sigma^{\mathfrak{S}} + \sum_{i \in [q]} d_{t_i} = d_t\bigg\}.
    \end{align*}
    Furthermore, we define $y_t(\varphi_t^{\text{up}},d_t)\in\Z^{V(G_t)}$ as the rooted solution obtained by composing $y_t^0(\varphi_t)$ and $y_{t_i}(\varphi_{t_i}^{\text{up}},d_{t_i})$ for $i\in[q]$, where $\varphi_t\in\varPhi^{Y_t}$, $\varphi_{t_i}^{\text{up}} \in \varPhi^{\adh{t_i}}$ and $d_{t_i} \in \mathcal{D}$ for $i\in[q]$ achieve the maximum.
\end{definition}

Notice that there is a constant number of choices for $\varphi^\text{up}_t \in \varPhi^{\adh{t}}$ and a polynomial number of choices for $d_t \in \mathcal{D}$. All in all, there are a polynomial number of table entries $F_t(\varphi_t^{\text{up}},d_t)$ for each $t \in V(T)$. For every $\varphi_t^{\text{up}} \in \varPhi^{\adh{t}}$ and $d_t \in \mathcal{D}$, we can compute $F_t(\varphi_t^{\text{up}},d_t)$ in strongly polynomial time, from the table entries for the children $t_1$, \ldots, $t_q$ (note that $q\le f_{\ref{SpecialDecomposition}}(s)$) and the polynomially many optimal values $F^0_t(\varphi_t)$, which can all be computed in strongly polynomial time before running the DP.


\begin{lemma} \label{lem:DP_technical}
The following holds for every $t \in V(T)$, $\varphi_t^{\text{up}} \in \varPhi^{\adh{t}}$ and $d_t \in \mathcal{D}$:
\begin{enumerate}
\item the profit of $y_t(\varphi_t^{\text{up}},d_t)$ with respect to vertices in $V(G_t - \adh{t})$ equals $F_t(\varphi_t^{\text{up}},d_t)$, 
\label{item:i_lem:DP_technical}
\item the weight vector of $y_t(\varphi_t^{\text{up}},d_t)$ with respect to vertices in $V(G_t - \adh{t})$ is equal to $d_t$, and 
\label{item:ii_lem:DP_technical}
\item $F_t(\varphi_t^{\text{up}},d_t)$ is the maximum profit of a rooted solution complying with $\varphi_t^{\text{up}}$ and $d_t$.
\label{item:iii_lem:DP_technical}
\end{enumerate}
\end{lemma}

\begin{proof}
All three results are easily proved by induction, starting from the leaves of $T$, using the fact that $p_t(v) = W_t(i,v) = 0$ for all $v \in \adh{t}$ and $i \in [k]$. We leave the details to the reader.
\end{proof}

We are now in position to prove the main algorithmic result of this paper.

\begin{proof}[Proof of \Cref{thm:main_IP}]
We showed in \Cref{overviewEquality} that any integer program \eqref{eqIPgeneral} on a totally $\Delta$-modular constraint matrix which becomes the transpose of a network matrix after removing a constant number of rows and columns can be efficiently reduced to an instance of \Cref{problemEquality} in which $A$ is the transpose of a network matrix.
Using the results of \Cref{sec:1sum2sum} we further reduce the problem to polynomially many instances of \Cref{problemMCP}, in which the graph is a subdivision of a $3$-connected graph whose roots all have degree at least $3$.


Consider any instance of \Cref{problemMCP}, defined by $I := (p,\ell,u,G,W,d)$, and let $\mathrm{OPT}$ denote the optimal value of instance $I$. Observe that, by \Cref{lemma:pumpkin_bis}, $G$ does not contain any rooted $K_{2,4k\Delta+1}$-model. Letting $s:=4k\Delta+1$, we apply \Cref{SpecialDecompositionExtended} in order to obtain a $f_{\ref{SpecialDecomposition}}(s)$-special tree-decomposition $(T, \mathcal{B})$ of $G$ and a collection $\{\mathcal{X}_t : t \in V(T)\}$ of docset superprofiles.

As observed before, we can compute the table entry $F_t(\varphi_t^{\text{up}},d_t)$ for all $t \in V(T)$, $\varphi_t^{\text{up}} \in \varPhi^{\adh{t}}$ and $d_t \in \mathcal{D}$, in strongly polynomial time. Let $t_0$ be the root of $T$ (remember that $\adh{t_0} = \varnothing$), and let $\varepsilon$ denote the null vector in $\Z^\varnothing$. We claim that $y_{t_0}(\varepsilon,d)$ is an optimal solution of $I$ whenever it is defined, and $I$ is infeasible whenever $y_{t_0}(\varepsilon,d)$ is not defined. We split the proof in two parts: (i) $y_{t_0}(\varepsilon, d)$ is feasible for \eqref{eqMCP} whenever it is defined, (ii) $F_{t_0}(\varepsilon, d) = \mathrm{OPT}$.

For (i), assume that $y_{t_0}(\varepsilon, d)$ is defined, that is, $F_{t_0}(\varepsilon, d) > -\infty$. For simplicity, let $y := y_{t_0}(\varepsilon, d)$. 

Observe that $y$ is obtained by composing optimal local solutions. Hence, there exist $\varphi_t \in \varPhi^{Y_t}$ for each $t \in V(T)$ such that $y$ is obtained by composing $y^0_t(\varphi_t)$ for $t \in V(T)$. Each optimal local solution $y^0_t(\varphi_t)$ is the restriction of $y$ to $B_t$. Moreover, each $y^0_t(\varphi_t)$ is in particular feasible for the corresponding instance of the local docset completion problem. Since each edge of $G$ is contained in $B_t$ for at least one $t \in V(T)$ by \Cref{def:td}.\ref{def:td_edge}, we see that $\ell(v,w) \le y(v)-y(w) \le u(v,w)$ for all $(v,w) \in E(G)$. 
By \Cref{lem:DP_technical}.\ref{item:ii_lem:DP_technical}, $Wy = d$. We conclude that $y$ is a feasible solution of \eqref{eqMCP}. Using \Cref{lem:DP_technical}.\ref{item:i_lem:DP_technical}, this implies $\mathrm{OPT} \ge F_{t_0}(\varepsilon, d)$. 

For (ii), it remains to prove that $F_{t_0}(\varepsilon, d) \ge \mathrm{OPT}$. Without loss of generality, assume that \eqref{eqMCP} is feasible and consider an optimal solution $\bar{y}$ that is the sum of at most $f_{\ref{propReductionIP}}(k,\Delta)$ incidence vectors of docsets of $G$. For each $t \in V(T)$, let $\bar{\varphi}_t \in \varPhi^{Y_t}$ denote the restriction of $\bar{y}$ to $Y_t$, let $\bar{\varphi}^\text{up}_t \in \varPhi^{\adh{t}}$ similarly denote the restriction of $\bar{y}$ to $\adh{t}$, and let $\bar{d}_t$ denote the weight vector of $\bar{y}$ restricted to the nodes of $V(G_t - \adh{t})$. Notice that, for every fixed node $t$, redefining $\bar{y}(v) := y^0_t(\bar{\varphi}_t)$ for $v \in B_t$ preserves both the feasibility and the optimality of $\bar{y}$. Now $\bar{y}$ is obtained by composing optimal local solutions, hence $\bar{y}$ is a rooted solution with respect to $t = t_0$ that complies with $\varepsilon$ and $d$. By \Cref{lem:DP_technical}, we conclude that $F_{t_0}(\varepsilon, d)$ is at least the profit of $\bar{y}$, that is, at least $\mathrm{OPT}$.
\end{proof}
\section{Structure of graphs without a rooted $K_{2,t}$-minor}\label{sec:structure}

This section has two main goals. 

First, we prove our decomposition theorem, see \Cref{thm:P1-P6_informal} in \Cref{sec:intro}, or \Cref{P1-P6} below for a precise statement. Roughly, our decomposition theorem states that every \(3\)-connected rooted graph \((G, R)\) without a rooted \(K_{2,t}\)-minor has a tree-decomposition $(T,\mathcal{B})$ in which every bag has a ``nice structure''. 

Second, we prove that given such a tree-decomposition, it is possible to find in polynomial time a collection $\mathcal{X}=\{\mathcal{X}_u : u \in V(T)\}$ of subsets of vertices of $G$ where each $\mathcal{X}_u$ is a docset superprofile of $R\cap B_u$, of polynomial size. This is the content of \Cref{SpecialDecomposition}.

Before stating \Cref{P1-P6}, we need to introduce some terminology and notations about graphs embedded in surfaces. 
A \emph{surface} $\surf$ is a non-empty compact connected Hausdorff topological space in which every point has a neighborhood that is homeomorphic to the plane~\cite{MT01}. Let $\surf(h,c)$ denote the surface obtained by removing $2h+c$ open disks with disjoint closures from the sphere and gluing $h$ cylinders and $c$ M\"obius strips onto the boundaries of these disks. As stated by the classification theorem of surfaces, every surface is topologically equivalent to $\surf (h,c)$, for some $h$ and $c$. The \emph{Euler genus} $\eg(\surf)$ of a surface $\surf \cong \surf(h,c)$ is $2h+c$. 
A \emph{curve} in $\surf$ is the image of a continuous function $f:[0,1]\rightarrow \surf$. The curve is \emph{simple} if $f$ is injective. The curve $e=f([0,1])$ \emph{connects} the endpoints $f(0)$ and $f(1)$.   
We say that a graph $H$ is \emph{embedded} in a surface $\surf$ if the vertices of $H$ are distinct points in $\surf$ and every edge $uv \in E(H)$ is a simple curve in $\surf$ which connects in $\surf$ the points which are the images of $u$ and $v$, such that its interior is disjoint from other edges and vertices.
We also consider embeddings in a disk, and the definitions above extend from surfaces to a disk.
For simplicity, we sometimes identify an embedded  graph with the corresponding subset of the surface when no confusion can occur. 

Consider a graph $H$ embedded in $\surf$. 
The {\em faces} of the embedded graph $H$ are the regions of $\surf - H$\footnote{That is, the equivalence classes of the equivalence relation where two points of $\surf - H$ are equivalent if there is a simple curve in $\surf - H$ connecting them.}.   
A curve $C$ in $\surf$ is said to be {\em $H$-normal} if it intersects  $H$ only at vertices.
The {\em distance in $\surf$} between two points $x,y\in \surf$ is the minimum of $|C\cap V(H)|$ over all $H$-normal curves $C$ linking $x$ to $y$. 
The distance in $\surf$ between two vertex subsets $A, B$ of $H$ is the minimum distance in $\surf$ between a vertex of $A$ and a vertex of $B$. 
A {\em noose} is a simple, closed, $H$-normal and non-contractible curve in $\surf$. For a surface $\surf$ that is not a sphere, the {\em facewidth} (also known as {\em representativity}) of the embedding of $H$ in $\surf$ is the minimum of $|C\cap V(H)|$ over all nooses $C$. (The facewidth of any graph embedded on the sphere is set to be infinity.) 

Let $(T,\mathcal{B})$ be a tree-decomposition of a graph $G$. 
Recall that the tree $T$ is assumed to be rooted at some node of $T$. 
With respect to this root, the \emph{weak torso} of a node $u\in V(T)$ is the graph $G[B_u]$ where we add edges between all (previously non-adjacent) pairs of vertices in each of the adhesion sets of $B_u\cap B_v$ where $v$ is a child of $u$. We denote the weak torso of $u$ as $G^{\#}[B_u]$.

\begin{theorem}\label{P1-P6}
For every integer \(t \ge 1\), there exists an integer \(\ell = \ell(t) \ge 1\) such that
the following holds.
Let $(G,R)$ be a $3$-connected rooted graph without a rooted $K_{2,t}$-minor. Then $G$ has a tree-decomposition $(T,\mathcal{B})$ with the following properties.
\begin{enumerate}
\item the adhesion of $(T,\mathcal{B})$ is at most \(\ell\), and
\item for every node $u\in V(T)$, all, but at most \(\ell\), children of \(u\) are leaves \(v\) with \(B_v \cap R \subseteq B_u\), and
\item for every node $u\in V(T)$, at least one of the following is true:
\begin{enumerate}
\item \(|B_u| \le \ell\), 
\item $u$ is a leaf of $T$ and $B_{u}\cap R \subseteq B_{u'}$, where $u'$ is the parent of $u$ in $T$, or 
\item \label{item:c}
there is a set $Z\subseteq B_u$ such that $|Z|\le \ell$, 
$G^{\#}[B_u]-Z$ is $3$-connected, does not contain a rooted $K_{2,t}$-minor with respect to the set of roots $B_u \cap R - Z$,
and has an embedding in a surface of Euler genus at most $\ell$ such that every face is bounded by a cycle of the graph, and all the vertices in $B_u \cap R - Z$ can be covered using at most $\ell$ facial cycles. 
\end{enumerate}
\end{enumerate}
Furthermore, for fixed $t$, the above tree-decomposition and the embeddings described in \ref{item:c} can be found in polynomial time given $(G,R)$ as input. 
\end{theorem}

We restate \Cref{SpecialDecomposition} here, for the reader's convenience.  

\SpecialDecomposition*

Recall that a tree-decomposition $(T,\mathcal{B})$ of a graph $G$ is $\ell$-special for some integer $\ell \ge 1$ if $|B_u\cap B_v|\le \ell$ for each edge $uv\in E(T)$, and every node $u\in V(T)$ has at most $\ell$ children. 

\subsection{Building the special tree-decomposition}\label{sec:star-decomposition}

A \emph{star-decomposition} of a graph $H$ is a tree-decomposition $H$ where the tree is a star. 
Such a tree-decomposition is denoted \((B_0; B_1, \ldots, B_m)\), where \(B_0\) is the center bag, and \(B_1, \ldots, B_m\) are the leaf bags.

For a fixed rooted graph \((G, R)\) without a rooted \(K_{2, t}\)-minor, we construct the tree-decomposition of \(G\) as in \Cref{P1-P6} recursively, one node at a time. 

For any subgraph \(H \subseteq G\), we denote by \(\partial_G(H)\) the set of all vertices of \(H\) that are adjacent in \(G\) to a vertex not in \(H\). When the graph \(G\) is clear from the context, we drop the subscript \(G\).

Given a subgraph \(H \subseteq G\), we construct a tree-decomposition of $H$ recursively. 
Firstly, we find a star-decomposition \((B_0; B_1, \ldots, B_m)\) of $H$ and secondly we find the subtrees attached to the root bag $B_0$ by applying recursion to the subgraphs \(H_1, \ldots, H_m\), where \(H_i = H[B_i]\). 
To do so, we need to introduce some definitions and technical lemmas. 

A set of vertices \(X \subseteq V(H)\) is \emph{well-linked} in a graph \(H\) if for every \(X_1, X_2 \subseteq X\) there exist \(\min\{|X_1|, |X_2|\}\) vertex-disjoint \(X_1\)--\(X_2\) paths in \(H\).
Given a rooted graph \((G, R)\) and an integer \(k \ge 1\), a subgraph \(H \subseteq G\) is said to be \emph{\(k\)-interesting} if there exists a set \(X \subseteq V(H)\) such that
\begin{itemize}
  \item \(|X| = k\),
  \item \(\partial(H) \subseteq X\),
  \item \(X\) is well-linked in \(H\), and
  \item there exist \(k\) vertex-disjoint \(R_H\)--\(X\) paths in \(H\) where $R_H=R\cap V(H)$.
\end{itemize}
Such a set $X$ is called a {\em core} of $H$.  
In particular, for every \(k\)-interesting subgraph \(H\), we have \(|\partial(H)| \le k\). 
A subgraph \(H \subseteq G\) is \emph{\(k\)-boring} if \(|\partial(H)| \le k\) and \(H\) is not \(k\)-interesting. 

We will prove \Cref{P1-P6} using the following two lemmas.
\begin{lemma}\label{boring}
  Let \(k \ge 1\) be an integer, let \((G, R)\) be a rooted graph, and let \(H \subseteq G\) be a \(k\)-boring subgraph. Then there exists a star-decomposition \((B_0; B_1, B_2)\) of \(H\) with \(\partial(H) \subseteq B_0\) such that
  \begin{enumerate}
    \item \(|B_0| \le 2k\),
    \item \(B_i \neq V(H)\) for \(i \in \{1, 2\}\).
    \item \(|B_0 \cap B_1| \le k\), and
    \item either \(|B_0 \cap B_2| \le k\), or \(B_2 \cap R \subseteq B_0\).
  \end{enumerate}
  Furthermore, for fixed $t$, the star-decomposition can be found in polynomial time given $(G,R)$ as input. 
\end{lemma}

\begin{lemma}\label{interesting}
  For every integer \(t \ge 1\), there exist integers \(k = k(t), k'=k'(t) \ge 1\) such that for every \(3\)-connected rooted graph \((G, R)\) without a rooted \(K_{2,t}\)-minor, and
  every \(k\)-interesting subgraph \(H \subseteq G\) with a core \(X\), there exists a star-decomposition \((B_0; B_1, \ldots, B_m)\) of $H$ such that \(X \subseteq B_0\), for every \(i \in \{1, \ldots, m\}\), we have \(|B_0 \cap B_i| \le k\), \(B_i \neq V(H)\), and if \(i > k'\), then \(B_i \cap R \subseteq B_0\), and there is a set \(U \subseteq B_0\) such that
  \begin{enumerate}
    \item \(|U| \le k'\), and
    \item \label{item:interesting_embeddding} \((H^\#[B_0] - U, R \cap B_0 - U)\) is \(3\)-connected, has no rooted \(K_{2,t}\)-minor, and admits an embedding in a surface with Euler genus at most $k$ such that every face is bounded by a cycle, and the set \(R \cap B_0 - U\) can be covered using at most \(k'\) facial cycles.
  \end{enumerate}
  Furthermore, for fixed $t$, the star-decomposition and the embedding described in \ref{item:interesting_embeddding} can be found in polynomial time given $(G,R)$ as input. 
\end{lemma}

Before proving these two lemmas, let us show that they easily imply \Cref{P1-P6}. 

\begin{proof}[Proof of \Cref{P1-P6}]
  Fix an integer \(t \ge 1\), and let \(k=k(t)\) and $k'=k'(t)$ be as in \Cref{interesting}. We prove the theorem for \(\ell = \max \{ 2k, k'\}\).
  Fix a $3$-connected rooted graph \((G, R)\) that does not contain a rooted \(K_{2, t}\)-minor. 
  
  We will prove the following claim: For every subgraph
  \(H \subseteq G\) with \(|\partial(H)| \le k\), there exists a tree-decomposition \((T, \mathcal{B})\) of $H$ such that \(\partial(H)\) is contained in the root bag, and
  \begin{enumerate}
  \item the adhesion of $(T,\mathcal{B})$ is at most \(\ell\), and
  \item for every node $u\in V(T)$, all but at most \(\ell\) children of \(u\) are leaves \(v\) with \(B_v \cap R \subseteq B_u\), and
  \item for every node $u\in V(T)$, at least one of the following is true:
  \begin{enumerate}
  \item \(|B_u| \le \ell\), 
  \item $u$ is a leaf of $T$ and $B_{u}\cap R \subseteq B_{u'}$, where $u'$ is the parent of $u$ in $T$, or 
  \item there is a set $Z\subseteq B_u$ such that $|Z|\le \ell$, $H^{\#}[B_u]-Z$ is $3$-connected, does not contain a rooted $K_{2,t}$-minor with respect to the set of roots $B_u \cap R - Z$, and has an embedding in a surface of Euler genus at most $\ell$ such that every face is bounded by a cycle of the graph, and all the vertices in $B_u \cap R - Z$ can be covered using at most $\ell$ facial cycles.
  \end{enumerate}
  \end{enumerate}

  Since \(\partial_G(G) = \emptyset\), the theorem follows from the claim with $H=G$.  

  We prove the claim by induction on \(|V(H)|\). If \(|V(H)| = 0\), the theorem is satisfied by a tree-decomposition with only one, empty bag.
  For the induction step suppose that \(|V(H)| \ge 1\), and let \((B_0; B_1, \ldots, B_m)\) be a star-decomposition of \(H\) obtained by applying \Cref{boring} or \Cref{interesting}, depending on whether \(H\) is \(k\)-boring or \(k\)-interesting. 
  Let $H_i \coloneqq H[B_i]$ for each $i\in \{0, 1, \dots, m\}$.  
  Note that we can remove from $B_0$ the set \(\cup_{i\in [m]} (B_0\cap B_i-\partial(H_i))\) while preserving the properties of the star-decomposition. This way we get that $\partial(H_i)=B_0\cap B_i$ for each $i\in [m]$, since \(\partial(H) \subseteq B_0\), 
  and thus
  either \(|\partial(H_i)| = |B_0 \cap B_i| \le 2k\) and \(B_i \cap R \subseteq B_0\), or \(|\partial(H_i)| = |B_0 \cap B_i| \le k\). 
  For each \(i \in [m]\) with  \(|\partial(H_i)| = |B_0 \cap B_i| \le 2k\) and \(B_i \cap R \subseteq B_0\), let simply \((T_i, \mathcal{B}_i)\) be the tree-decomposition of $H_i$ consisting of the unique bag $B_i$. 
  For each \(i \in [m]\) with \(|\partial(H_i)| = |B_0 \cap B_i| \le k\), apply the induction hypothesis to \(H_i = H[B_i]\)
  to obtain a tree-decomposition \((T_i, \mathcal{B}_i)\) of $H_i$.  
  Let \((T, \mathcal{B})\) be the tree-decomposition of \(H\) obtained from the disjoint union of these tree-decompositions by adding a new root node \(u\) whose children are the roots of \(T_1, \ldots, T_m\), and whose bag is the center bag \(B_0\) of the star-decomposition.
  It is easily checked that the resulting tree-decomposition of $H$ satisfies the claim.
\end{proof}

A \emph{separation} in a graph $G$ is an ordered pair $(A,B)$ of vertex subsets of $G$ such that $A\cup B=V(G)$ and there are no edges with one endpoint in $A \setminus B$ and the other in $B\setminus A$. 
The \emph{order} of a separation is $|A\cap B|$. 
We conclude this subsection with the proof of \Cref{boring}. 

\begin{proof}[Proof of \Cref{boring}]
  Let \(H\) be a \(k\)-boring subgraph of \((G, R)\).
  If \(|V(H)| \leq 2k\), then the lemma is satisfied by the star-decomposition \((V(H); \emptyset, \emptyset)\).    
  So we assume that \(|V(H)| \ge 2k+1\). Let \(X\) be any set with
  \(\partial(H) \subseteq X \subseteq V(H)\) and \(|X| = k\).

  Suppose that \(X\) is not well-linked in $H$, and let \(X_1, X_2 \subseteq X\) be such that there are no \(\min\{|X_1|, |X_2|\}\) vertex-disjoint \(X_1\)--\(X_2\) paths in \(H\).
  By Menger's Theorem, there exists a separation \((B_1, B_2)\) of \(H\) with \(X_i \subseteq B_i\) for $i\in \{1,2\}$ and \(|B_1 \cap B_2| < \min\{|X_1|, |X_2|\}\). Let \(B_0 = X \cup (B_1 \cap B_2)\), and observe that
  \begin{itemize}
    \item \(|B_0| \le |X| + |B_1 \cap B_2| < 2k\), and
    \item for $i\in \{1,2\}$, \(|B_0 \cap B_i| \le |X \setminus X_{3-i}| + |B_1 \cap B_2| < |X \setminus X_{3-i}| + |X_{3-i}| \le |X| = k\),
  \end{itemize} 
  so \((B_0; B_1, B_2)\) satisfies the lemma.

  Now suppose that \(X\) is well-linked in $H$.
  Since \(H\) is \(k\)-boring, this implies that there are no \(k\) vertex-disjoint \(R_H\)--\(X\) paths in \(H\), where $R_H=R\cap V(H)$. By Menger's theorem,
  there exists a separation \((B_1, B_2)\) of \(H\) with \(R_H \subseteq B_1\), \(X \subseteq B_2\), and \(|B_1 \cap B_2| < k\). Let \(B_0 = X \cup (B_1 \cap B_2)\), and observe that
  \begin{itemize}
    \item \(|B_0| \le |X| + |B_1 \cap B_2| < 2k\), and
    \item \(|B_0 \cap B_1| = |B_1 \cap B_2| < k\), and
    \item \(B_2 \cap R = B_2 \cap R_H \subseteq B_1 \cap B_2 \subseteq B_0\).
  \end{itemize}
  Thus, \((B_0; B_1, B_2)\) satisfies the lemma.

  Finally, we remark that in both cases above, the resulting star-decomposition of $H$ is easily computed in polynomial time. 
\end{proof}

The following subsections are devoted to the proof of~\Cref{interesting}.

\subsection{Finding a large wall}

Let $(G,R)$ be a rooted graph and let $k\ge 1$ be an integer. 
Suppose that $H$ is a $k$-interesting subgraph of $(G, R)$ and let $X\subseteq V(H)$ be a core of $H$. Thus, $|X|=k$, $\partial(H)\subseteq X$, $X$ is well-linked in $H$, and there are $k$ vertex-disjoint paths between $R_H=R\cap V(H)$ and $X$ in $H$.   
The goal of this section is to show that $H$ must contain a large \emph{wall} as a subgraph and that there are many vertex-disjoint paths from $R_H$ to any \emph{transversal} set of vertices in the wall. 
We proceed with the necessary definitions. 

Given an integer $h\geq 2$, an {\em elementary wall of height $h$} is the graph obtained from the $2h \times h$ grid with vertex set $[2h] \times [h]$ by removing all edges with endpoints $(2i-1,2j-1)$ and $(2i-1,2j)$ for all $i\in [h]$ and $j\in [\lfloor h/2 \rfloor]$, all edges with endpoints $(2i,2j)$ and $(2i,2j+1)$ for all $i\in [h]$ and $j\in [\lfloor (h-1)/2 \rfloor]$, and removing the two vertices of degree at most $1$ in the resulting graph.

In an elementary wall $W$ of height $h$, there is a unique set of $h$ vertex-disjoint paths linking the bottom row (vertices of the form $(i,1)$) to the top row (vertices of the form $(i,h)$). 
These paths are called the {\em vertical paths} of $W$. We enumerate the vertical paths as $Q_1, \dots, Q_h$ so that the first coordinates of their vertices are increasing. There is also a unique set of $h$ vertex-disjoint paths linking $Q_1$ to $Q_h$, called the {\em horizontal paths} of $W$. 

A subdivision of an elementary wall of height $h$ is called a {\em wall of height $h$}. Vertical paths, and horizontal paths of a wall are defined as expected, as the subdivided version of their counterparts in the elementary wall. 

Let $W$ be a wall. Let $U$ be the set of vertices of degree $3$ in $W$, we call the vertices in $U$ \emph{branch vertices}. A \emph{transversal set of vertices in $W$} is an inclusion-wise maximal subset of $U$ such that no two vertices are in the same row or column of $W$. 
Observe that such a set necessarily has size $h$, where $h$ is the height of $W$. 
One way to obtain a transversal set in $W$ is to choose all the vertices on the ``diagonal'' of $W$. 


\begin{lemma}\label{well_linked_wall}
For every integer $h\ge 1$ there is an integer $k=k_{\ref{well_linked_wall}}(h)\ge 1$ so that the following holds. 
For every rooted graph $(G,R)$, every $k$-interesting subgraph $H$ of $G$, and every core $X$ of $H$, there exists a wall $W$ of height $h$ in $H$ such that, for every transversal set $Y$ of $W$, there are $h$ vertex-disjoint $X$--$Y$ paths and $h$ vertex-disjoint $R_H$--$Y$ paths in $H$, where $R_H = R \cap V(H)$.    
\end{lemma}

For the proof of the lemma we need a few definitions and a lemma from \cite{KTW20}. 
Let $s\in \mathbb{N}$. 
An \emph{$s$-tangle} in a graph $G$ is a collection $\mathcal{T}$ of (ordered) separations of $G$ such that,
\begin{itemize}
\item for every separation $(A,B)$ of $G$ of order at most $s-1$, 
exactly one of $(A,B), (B,A)$ is included in $\mathcal{T}$, and 
\item there are no three separations $(A_1,B_1),(A_2,B_2),(A_3,B_3)\in \mathcal{T}$ with $G[A_1]\cup G[A_2] \cup G[A_3]=G$. 
\end{itemize}
The parameter $s$ is called the \emph{order} of the tangle.

Let $W$ be a wall of height $h$ in a graph $G$. 
The tangle $\mathcal{T}_W$ induced by $W$ is the collection of separations $(A,B)$ of $G$ of order at most $h-1$ such that $B\setminus A$ contains the vertex set of both a horizontal and vertical path of $W$. 
A tangle $\mathcal{T}'$ is a {\em restriction} of $\mathcal{T}$ if $\mathcal{T}'\subseteq \mathcal{T}$. 

\begin{lemma}[Lemma 14.6 in~\cite{KTW20}] \label{KTWtangle}
Let $f_{\text{wall}}$ be a function such that for every positive integer $s$, every graph with tree-width at least $f_{\text{wall}}(s)$ contains an $s$-wall as a subgraph.
Let $h\ge 1$ be an integer. Let $G$ be a graph and let $\mathcal{T}$ be a tangle of order at least $3f_{\text{wall}}(6h^2)+1$ in $G$. Then there exists a wall $W$ of height $h$ in $G$ such that $\mathcal{T}_W$ is a restriction of $\mathcal{T}$.
\end{lemma}

Note that such a function $f_{\text{wall}}$ exists by the Grid Minor Theorem of Robertson and Seymour~\cite{RS86}, and is computable (see~\cite{CT21} for the current best known bound).

We may now turn to the proof of \Cref{well_linked_wall}. 

\begin{proof}[Proof of \Cref{well_linked_wall}]
Let $h\ge 1$ and let $k=6\cdot (3f_{\text{wall}}(6h^2)+1)$. 
Let $H$ be a $k$-interesting subgraph of a rooted graph $(G, R)$. 
Let $X\subseteq V(H)$ be a core of $H$. 
We define a $k/6$-tangle $\mathcal{T}_X$ in $H$ with respect to $X$: 
For every separation $(A,B)$ of $H$  
of order at most $(k/6-1)$ we add either $(A,B)$ or $(B, A)$ to $\mathcal{T}_X$ as follows. 
Since $X$ is well-linked in $H$, either $A\setminus B$ or $B \setminus A$ contains at most $k/6-1$ elements from $X$. We add $(A,B)$ to $\mathcal{T}_X$ if $|(B\setminus A)\cap X|\ge k/6$, otherwise we add $(B,A)$ to $\mathcal{T}_X$. Note that for each separation $(A,B)\in \mathcal{T}_X$, $|A\cap X|\le k/3-2$, 
and thus it is easily checked that $\mathcal{T}_X$ is a tangle.  

By \Cref{KTWtangle}, there is a wall $W$ of height $h$ in $H$ such that the tangle $\mathcal{T}_W$ induced by $W$ is a restriction of $\mathcal{T}_X$. 

Let $Y$ be some transversal set of $W$.  
Thus, $|Y|=h$. First, suppose that there are no $h$ vertex-disjoint $X$--$Y$ paths in $H$. Then, by Menger's Theorem, there is a separation $(A,B)$ in $H$ such that $|A\cap B|\le h-1$, $Y\subseteq A$ and $X\subseteq B$. Note that $h\le k/6$ and hence $(A,B)\in \mathcal{T}_X$. 
Since $|A\cap B|\le h-1$, there is a horizontal path of $W$ and a vertical path of $W$ that have an empty intersection with $A\cap B$. Each of these two paths contains some vertex in $Y$ (otherwise we get a contradiction to the maximality of $Y$ in the definition of a transversal in a wall), and thus they must be contained in $A\setminus B$. Hence, $(B,A)\in \mathcal{T}_W$. But this contradicts the fact that $\mathcal{T}_W$ is a restriction of $\mathcal{T}_X$. 
Therefore, we conclude that there are $h$ vertex-disjoint $X$--$Y$ paths in $H$.

Assume now that there are no $h$ vertex-disjoint $R_H$--$Y$ paths in $H$. 
Then, similarly as before, let $(A,B)$ be a separation in $H$ such that $|A\cap B|\le h-1$, $Y\subseteq A$ and $R_H\subseteq B$. 
For the same reason as before, we must have $(B,A)\in \mathcal{T}_W$, and thus also $(B,A)\in \mathcal{T}_X$ since $\mathcal{T}_W$ is a restriction of $\mathcal{T}_X$. 
By the definition of $\mathcal{T}_X$, this implies that $|(B\setminus A)\cap X|\le k/6-1$. 
Since we also have $|A\cap B|\le h-1 \le k/6-1$, 
it follows that $((B\setminus A)\cap X) \cup (A\cap B)$ has size at most $k/3-2 < k$ and intrsects all $R_H$--$X$ paths in $H$, which contradicts the existence of $k$ vertex-disjoint $R_H$--$X$ paths in $H$. 
Therefore, we conclude that there are $h$ vertex-disjoint $R_H$--$Y$ paths in $H$, as desired.
\end{proof}

\subsection{Excluding a complete graph minor grasped by a wall}


Given a wall $W$ in $G$, and an $H$-model $M$ in $G$ for a graph $H$ with $|V(H)|=s$, we say that the wall $W$ {\em grasps} the model $M$ if for every $v\in V(H)$ there are $s$ distinct horizontal paths $R^v_1, \dots, R^v_s$ of $W$ and $s$ distinct vertical paths $Q^v_1, \dots, Q^v_s$ of $W$ such that $V(R^v_i) \cap V(Q^v_i) \subseteq M(v)$ for each $i\in [s]$. 

\begin{lemma}\label{minor_grasped}
Let $h,t\in \mathbb{N}$ be such that $h\ge 4t+2$. Let $(G,R)$ be a rooted graph with $|R|\ge h$ and having no rooted $K_{2,t}$-minor. Assume that $G$ contains an $h$-wall $W$ such that, for every transversal set $Y$ of $W$, there are $h$ vertex-disjoint $R$--$Y$ paths in $G$.  
Then $G$ does not contain a $K_{4t+2}$-model grasped by the wall $W$. 
\end{lemma}

\begin{proof}
Assume to the contrary that $G$ contains a $K_{4t+2}$-model $M$ grasped by the wall $W$. 
First, we show that one can find a set $Y'$ of $2t$ branch vertices of $W$ contained in distinct branch sets of $M$ and such that each horizontal path of $W$ and each vertical path of $W$ contains at most one vertex in $Y'$.
This can be done iteratively, marking the new paths found in each step. Let $s=4t+2$, and consider $2t$ distinct vertices $v_1, \dots, v_{2t}$ of $K_{4t+2}$.  
For $j = 1, 2, \dots, 2t$, consider $s$ distinct horizontal paths $R^{v_j}_1, \dots, R^{v_j}_s$ of $W$ and $s$ distinct vertical paths $Q^{v_j}_1, \dots, Q^{v_j}_s$ of $W$ such that $V(R^{v_j}_i) \cap V(Q^{v_j}_i) \subseteq M(v_j)$ for each $i\in [s]$. Since at most $2j-2$ of these paths were marked previously, we can find $i\in [s]$ such that neither $R^{v_j}$ nor $Q^{v_j}$ is marked. 
We mark these two paths and choose a branch vertex of $W$ in $V(R^{v_j}_i) \cap V(Q^{v_j}_i)$. 
Finally, we let $Y'$ be the set of chosen branch vertices.

Next, extend the set $Y'$ to a transversal set $Y$ in $W$. By our assumptions, there are $h$ vertex-disjoint $R$--$Y$ paths in $G$, and thus in particular there are $2t$ vertex-disjoint $R$--$Y'$ paths in $G$.

In what follows, we call {\em feasible path system} a collection of $2t$ vertex-disjoint paths in $G$ such that all the paths have one endpoint (called its {\em source}) in $R$, and their other endpoint (called its {\em destination}) in distinct branch sets of $M$. 
Note that there is at least one feasible path system in $G$, since there are $2t$ vertex-disjoint $R$--$Y'$ paths in $G$.  
A path $P$ in a feasible path system ${\cal P}$ is {\em good} if the branch set of $M$ containing the destination of $P$ avoids all the other paths in ${\cal P}$, and is {\em bad} otherwise. 
With respect to ${\cal P}$, we classify branch sets of $M$ into four categories: A branch set $B$ of $M$ is 
\begin{itemize}
  \item {\em good} if $B$ contains the destination of a {\em good} path of ${\cal P}$; 
  \item {\em bad} if $B$ contains the destination of a {\em bad} path of ${\cal P}$; 
  \item {\em ugly} if $B$ intersects some path in ${\cal P}$ but $B$ contains no destination of paths in ${\cal P}$; 
  \item {\em free} if $B$ intersects no path in ${\cal P}$. 
\end{itemize}
We let $g({\cal P}), b({\cal P}), u({\cal P}), f({\cal P})$ denote the number of branch sets in $M$ that are respectively good, bad, ugly, and free with respect to ${\cal P}$. 
Finally, we let $i({\cal P})$ denote the sum over every path $P\in {\cal P}$ of the number of branch sets of $M$ intersected by $P$. 
(The letter $i$ stands for `intersection'.) 

To each feasible path system ${\cal P}$ in $G$ we associate a corresponding {\em vector} with two entries, defined as follows:
\begin{equation}
  \label{eq:valueofP}
  (i({\cal P}) + b({\cal P}), b({\cal P})). 
\end{equation}
Now, choose a feasible path system ${\cal P}$ in $G$ whose vector is 
lexicographically minimum. 
Let us make a few observations about ${\cal P}$. 
First, we claim that 
\begin{equation}
  \label{eq:ugly_vs_good}
  u({\cal P}) \leq g({\cal P}).  
\end{equation}
To see this, observe first that no bad path $P$ of ${\cal P}$ intersects an ugly branch set $B$ of $M$, since otherwise we could shorten $P$ so that it ends in $B$, which would improve the vector of ${\cal P}$.    
Thus, ugly branch sets of $M$ only intersects good paths of ${\cal P}$. 
If a good path $P$ of ${\cal P}$ intersects two ugly branch sets $B_1, B_2$ of $M$ in this order from its source, then one could shorten $P$ so that it ends in $B_1$. 
This possibly increases $b({\cal P})$ by $1$ (since $P$ might become bad) but decreases $i({\cal P})$ by at least $2$ (since $P$ no longer intersects $B_2$ nor the good branch set that contained its destination), thus resulting in a feasible path system with a lexicographically smaller vector, a contradiction. 
Hence, every good path intersects at most one ugly branch set, and \eqref{eq:ugly_vs_good} follows. 

Observe that 
\begin{equation}
  \label{eq:free}
  f({\cal P}) \geq 2  
\end{equation}
since $b({\cal P}) + g({\cal P}) + u({\cal P}) + f({\cal P})=4t+2$ and  $b({\cal P}) + g({\cal P}) + u({\cal P}) 
= 2t + u({\cal P}) \leq 2t + g({\cal P}) \leq 4t$ by \eqref{eq:ugly_vs_good}. 

Next, we claim that 
\begin{equation}
  \label{eq:bad_vs_good}
  b({\cal P}) \leq g({\cal P}).  
\end{equation}
To show this, we are going to associate to each bad branch set $B$ of $M$ a corresponding good path of ${\cal P}$, in an injective manner. 
Since the number of good paths equals to the number of good branch sets, this will imply \eqref{eq:bad_vs_good}.  
So, suppose $B$ is a bad branch set of $M$. 
Let $F$ be a free branch set of $M$, which exists by \eqref{eq:free}, and let $vw$ be an edge linking $B$ and $F$ in $G$, with $v$ in $B$.  
Consider a spanning tree $T$ of $G[B]$ and let $P$ be the path in ${\cal P}$ closest to $v$ in the tree $T$.  
(Possibly $v$ is in $P$.) 
If $P$ is bad, then one could reroute $P$ towards $v$ in the tree $B$ and use the edge $vw$ so that it ends in $F$. 
Then $P$ would become good, so $b({\cal P})$ would decrease by $1$, and $i({\cal P})$ would increase by at most $1$ (due to the intersection of $P$ with $F$). 
Overall, this would result in a lexicographically smaller vector for ${\cal P}$, which is a contradiction. 
Thus $P$ must be good. 
Moreover, if $B$ is not the last bad branch set intersected by $P$ on its way to its destination, then performing the same rerouting gives again a contradiction, so $B$ is the last such branch set. 
Hence, we may associate $B$ to the good path $P$, and this way $P$ will not be chosen by any other bad branch set, as desired.  
This shows \eqref{eq:bad_vs_good}. 

Since $b({\cal P}) + g({\cal P}) = 2t$, it follows from \eqref{eq:bad_vs_good} that one can find $t$ good branch sets in $M$. 
Using the union of these branch sets with the corresponding good paths in ${\cal P}$ plus two additional free branch sets of $M$  (which exist by \eqref{eq:free}), we find a rooted $K_{2,t}$-model in $(G, R)$. 
This final contradiction concludes the proof. 
\end{proof}

\subsection{The structure of graphs excluding a complete graph minor grasped by a wall}
\label{sec:structure of graphs excluding a minor}

We combine results in \cite{KTW20} and \cite{DiestelKMW12} to derive a theorem that describes the structure of graphs without a complete graph minor grasped by a wall. Before we do so, we introduce the necessary definitions. 

A {\em vortex} is a pair $V=(J,\Omega)$ where $J$ is a graph and $\Omega=\Omega(V)$ is a linearly ordered set $(u_1,u_2,\ldots,u_q)$ of a subset of vertices of $J$. 
With a slight abuse of notation, we will denote the unordered set of vertices $\{u_1,u_2,\ldots,u_q\}$ by $\Omega$ as well. 

Given a vortex $V=(J,\Omega)$ and a vertex subset $Z$, we denote by $V-Z$ the vortex obtained by deleting the vertices in $Z$ from $J$ and \(\Omega\). 
If $\mathcal{V}$ is a set of vortices, we let $\mathcal{V}-Z$ denote the set of vortices $V-Z$ with $V\in \mathcal{V}$ such that $V-Z$ has at least one vertex. 

Let $\surf$ be a surface and let $D$ be a closed disk in the surface. We denote by $\bd(D)$ the boundary of $D$.

Let $V=(J,\Omega)$ be a vortex with $\Omega=(u_1,u_2,\ldots,u_q)$. 
A {\em linear decomposition} of $V$ is a collection of sets $(X_1,X_2,\ldots, X_q)$ such that 
\begin{itemize}
\item for each $i\in [q]$, $X_i \subseteq V(J)$ and $u_i\in X_i$, and moreover $\cup_{i=1}^q X_i = V(J)$; 
\item for every $uv\in E(J)$, there exists $i\in [q]$ such that $\{u,v\}\subseteq X_i$, and
\item for every $x\in V(J)$, the set $\{i:x\in X_i\}$ is an interval in $[q]$. 
\end{itemize}
The {\em adhesion} of the linear decomposition is $\max(|X_i\cap X_{i+1}|: 1\le i \le q-1)$. 

Let $\alpha_0,\alpha_1,\alpha_2\in \mathbb{N}$, and let $\surf$ be a surface. A graph $G$ is {\em $(\alpha_0,\alpha_1,\alpha_2)$-nearly embeddable in $\surf$} if there exist $Z\subseteq V(G)$ with $|Z|\le \alpha_0$ and an integer $\alpha'\le \alpha_1$ such that $G-Z$ can be written as the union of $p+1$ edge-disjoint graphs $G_0,G_1,\ldots,G_p$ with the following properties:
\begin{enumerate}
\item For each $i\in [p]$, $V_i:=(G_i,\Omega_i)$ is a vortex where the set $\Omega_i$ is $V(G_i\cap G_0)$ (its linear ordering is specified in \ref{item:ordering} below). For all $1\le i < j \le p$, $G_i\cap G_j\subseteq G_0$.
\item The vortices $V_1,\ldots,V_{\alpha'}$ have a linear decomposition of adhesion at most $\alpha_2$. Let $\mathcal{V}$ be the collection of those vortices.
\item The vortices $V_{\alpha'+1},\ldots,V_p$ satisfy $|\Omega_i|\le 3$ for each $\alpha'+1\le i \le p$. 
Let $\mathcal{W}$ be the collection of those vortices.
\item \label{item:ordering}
There are closed disks $D_1,\ldots,D_p$ in $\surf$ with disjoint interiors and an embedding $\sigma: G_0 \hookrightarrow \surf - \cup_{i=1}^p \inter(D_i)$ such that $\sigma(G_0)\cap \bd(D_i)=\sigma(\Omega_i)$ for each $i\in [p]$ and the linear ordering of $\Omega_i$ is compatible with the cyclic ordering of $\sigma(\Omega_i)$.
\end{enumerate} 
We call the tuple $(\sigma,G_0,Z,\mathcal{V},\mathcal{W})$ an {\em $(\alpha_0,\alpha_1,\alpha_2)$-near embedding} of $G$.\footnote{We note that $\sigma$ admits a polynomial-size combinatorial description, see~\cite{MT01}.}. 
We will denote by $D(V_i)$ the disk $D_i$ corresponding to vortex $V_i$ in the above definition. We call vortices in $\mathcal{V}$ {\em large vortices} and vortices in $\mathcal{W}$ {\em small vortices}. 
Given $\alpha \in \mathbb{N}$, a tuple $(\sigma,G_0,Z,\mathcal{V},\mathcal{W})$ is said to be an {\em $\alpha$-near embedding} if it is an $(\alpha,\alpha,\alpha)$-near embedding.
We note that if the constants are clear from the context, we sometimes omit them and simply write {\em near embedding}. 

Given a near embedding $(\sigma,G_0,Z,\mathcal{V},\mathcal{W})$ of a graph $G$ in a surface $\surf$, we define a corresponding graph $G_0'$, obtained from $G_0$ by adding an edge $vw$ for every pair of non-adjacent vertices $v,w$ in $G_0$ that are in a common small vortex $V\in \mathcal{W}$. 
These extra edges are drawn without crossings, in the disks accommodating the corresponding vortices. 
We will refer to these edges as the {\it virtual edges} of $G_0'$. 
If $H'$ and $H$ are subgraphs of $G_0'$ and $G-Z$, respectively, such that $H$ is obtained from $H'$ by replacing each virtual edge $uv$ of $H'$ with a $u$--$v$ path contained in some vortex in $\mathcal{W}$, in such a way that all the paths are internally vertex disjoint, then we call $H$ a {\it lift} of $H'$ (with respect to the near embedding $(\sigma,G_0,Z,\mathcal{V},\mathcal{W})$), and say that $H'$ {\it can be lifted to} $H$. 

A cycle $C$ in a graph $H$ embedded in a surface $\surf$ is {\em flat} if $C$ bounds a disk in $\surf$. 
A wall $W$ in $H$ is {\em flat} if the boundary cycle of $W$ (which is defined in the obvious way) bounds a closed disk $D(W)$ with the wall $W$ drawn inside it.  

    
Vertex-disjoint cycles $C_1, \dots, C_s$ of $H$ are {\em concentric} if they bound closed disks $D_1, \dots, D_s$ in $\Sigma$ with $D_1 \supseteq \cdots \supseteq D_s$ in $\surf$. 
These cycles {\em enclose} a vertex subset $\Omega$ if $\Omega \subseteq D_s$. 
They {\em tightly enclose} $\Omega$ if moreover, for every $i\in[s]$ and every point $v\in \bd(D_i)$, there is a vertex $w\in \Omega$ at distance at most $s-i+2$ from $v$ in $\surf$ with respect to $H$. 
(See the beginning of \Cref{sec:structure} for the definition of `distance'.) 

In the context of a near-embedding $(\sigma,G_0,Z,\mathcal{V},\mathcal{W})$ of a graph $G$ in a surface $\surf$, concentric cycles $C_1, \dots, C_s$ in $G'_0$ {\em (tightly) enclose} a vortex $V\in \mathcal{V}$ if they (tightly) enclose $\Omega(V)$. 

For integers $3\leq s \leq h$, an $(\alpha_0,\alpha_1,\alpha_2)$-near embedding $(\sigma,G_0,Z,\mathcal{V},\mathcal{W})$ of a graph $G$ in a surface $\surf$ is {\em $(s, h)$-good} if the following properties are satisfied. 
\begin{enumerate}
\item \label{P:flat}
$G'_0$ contains a flat wall $W'_0$ of height $h$. 
\item \label{P:facewidth}
If $\surf$ is not the sphere, then the facewidth of $G'_0$ in $\surf$ is at least $s$. 
\item \label{P:concentric}
For every vortex $V\in \mathcal{V}$ there are $s$ concentric cycles $C_1(V), \dots, C_s(V)$ in $G'_0$ tightly enclosing $V$ and bounding closed disks $D_1(V) \supseteq \cdots \supseteq D_{s}(V)$, such that $D_{s}(V)$ contains $\Omega(V)$ and $D(W'_0)$ does not intersect $D_1(V)$. 
For distinct $V, V'\in \mathcal{V}$, the disks $D_1(V)$ and $D_1(V')$ are disjoint. 
\end{enumerate}
We call the above wall $W'_0$ a \emph{good wall} with respect to the $(s,h)$-good $(\alpha_0,\alpha_1,\alpha_2)$-near embedding. 
Also, we let $G(D_i(V))$ denote be the subgraph of $G$ contained in the disk $D_i(V)$ for each $i\in [s]$.

By combining the main result of~\cite{KTW20} (Theorem 2.11) and results in~\cite{DiestelKMW12}, it is possible to deduce the following theorem. 
We remark that some extra properties of the embedding could be derived as well from these two papers, we only state here the properties that we will need.  

\begin{theorem}[{\cite[Theorem 2.11]{KTW20}, \cite{DiestelKMW12}}]\label{GMST}
Let $t',s,h$ be positive integers with $3 \leq s \leq h$. Then there exist $h'=h_{\ref{GMST}}'(t',s,h)$, $\alpha_0=\alpha_0(t',s)$, and $\alpha=\alpha(t')$ such that the following holds. Let $G$ be a graph and let $W$ be a wall of height $h'$ in $G$. Then either $G$ has a $K_{t'}$-model grasped by $W$, or $G$ has an $(s,h)$-good $(\alpha_0,\alpha,\alpha)$-near embedding $(\sigma,G_0,Z,\mathcal{V},\mathcal{W})$ in a surface $\surf$ of Euler genus at most $\alpha$, such that $G'_0$ contains a good wall $W'_0$ with respect to the embedding that can be lifted to a subwall $W_0$ of $W$.

Furthermore, for some computable function $T$, there is an algorithm with running time $T(t',s,h)\cdot n^{O(1)}$ that, given an $n$-vertex graph $G$ and a wall $W$ as above, finds one of the two structures guaranteed by the two outcomes of the theorem. 
\end{theorem}

Using this theorem, we prove the following lemma. 

\begin{lemma}\label{embedding_H}
Let $t, s, h$ be positive integers with $3 \leq s \leq h$. 
Let $t'= 4t+2$, $h'=h_{\ref{GMST}}'(t',s,h)$, $k=k_{\ref{well_linked_wall}}(h')$, and let $\alpha_0=\alpha_0(t',s)$ and $\alpha=\alpha(t')$ be as in \Cref{GMST}.  
Then, for every  rooted graph $(G,R)$ without a rooted $K_{2,t}$-minor and every $k$-interesting subgraph $H$ of $(G, R)$ with core $X$,  
\begin{enumerate}
\item \label{item:near_embeddding}
there exists an $(s,h)$-good $(\alpha_0,\alpha,\alpha)$-near embedding $(\sigma,H_0,Z,\mathcal{V},\mathcal{W})$ of $H$ in a surface $\surf$ of Euler genus at most $\alpha$, and
\item \label{item:Hprime_0}
$H'_0$ contains a good wall $W'_0$ with respect to the embedding that can be lifted to a wall $W_0$ in $H$, 
and such that for every transversal $Y'$ of $W'_0$ there are $h$ vertex-disjoint $X$--$Y'$ paths in $H$, and there are $h$ vertex-disjoint $R_H$--$Y'$ paths in $H$, where $R_H = R \cap V(H)$. 
\end{enumerate}
Furthermore, for some function $T'$, there is an algorithm with running time $T'(t,s,h)\cdot n^{O(1)}$ that, given an $n$-vertex rooted graph $(G, R)$ and a $k$-interesting subgraph $H$ of $G$ with core $X$ as above, finds this embedding and the wall $W'_0$.  
\end{lemma}

\begin{proof}
Let $(G,R)$ be a rooted graph without a rooted $K_{2,t}$-minor and let 
$H$ be a $k$-interesting subgraph of $(G, R)$ with core $X$. 
By \Cref{well_linked_wall}, there is a wall $W$ of height $h'$ in $H$ such that for every transversal set $Y$ of $W$ there are $h'$ vertex-disjoint $X$--$Y$ paths in $H$ and $h'$ vertex-disjoint $R_H$--$Y$ paths in $H$, where $R_H = R \cap V(H)$.  

It follows from \Cref{minor_grasped} that $H$ does not contain a $K_{t'}$-model grasped by $W$. 
\Cref{GMST} implies then that $H$ has an $(s,h)$-good $(\alpha_0,\alpha,\alpha)$-near embedding $(\sigma,H_0,Z,\mathcal{V},\mathcal{W})$ in a surface $\surf$ of Euler genus at most $\alpha$. This shows property \ref{item:near_embeddding}. 

Moreover, by \Cref{GMST}, $H'_0$ contains a good wall $W'_0$ with respect to the embedding that can be lifted to a subwall $W_0$ of $W$. 
Finally, suppose $Y'$ is a transversal set of $W'_0$. 
Since $W'_0$ has height $h$, it follows that $|Y'| = h$.  
Since $W'_0$ can be lifted to the subwall $W_0$ of $W$, it follows that $Y'$ can be extended to a transversal set $Y$ of $W$. 
We know that there are $h'$ vertex-disjoint $X$--$Y$ paths in $H$ and $h'$ vertex-disjoint $R_H$--$Y$ paths in $H$. 
Keeping only the paths that have an endpoint in $Y'$ give the desired paths in property \ref{item:Hprime_0}. 
\end{proof}

\subsection{Discarding large vortices}

For an integer $q$, the \emph{shallow vortex grid of order $q$} is the graph obtained from the Cartesian product of a cycle $u_1,u_2,\ldots,u_{4q}$ with a path $v_1,v_2,\ldots,v_q$ by adding the edges $(u_{4(i-1)+1},v_1)(u_{4(i-1)+3},v_1)$ and $(u_{4(i-1)+2},v_1)(u_{4(i-1)+4},v_1)$ for every $i\in [q]$. 
For  $i\in [q]$, the cycle $(u_1,v_i),(u_2,v_i),\ldots, (u_{4q}, v_i)$ is called the {\em $i$-th cycle} of the shallow vortex grid. 
A \emph{sliced shallow vortex grid of order $q$} is defined as above with the difference that instead of taking the Cartesian product of a cycle and a path, we take the Cartesian product of a path with a path.

Let $G$ be a graph and assume that there is an $(s,h)$-good near-embedding $(\sigma,G_0,Z,\mathcal{V},\mathcal{W})$ of $G$ in a surface $\surf$. Let $V=(J,\Omega)\in \mathcal{V}$ be a large vortex of the near-embedding.  
We say that a model of a shallow vortex grid of order $q$ \emph{surrounds} $V$ if for each $i\in [q-1]$, the union of the vertex sets of the branch sets in the model that correspond to vertices of the $i$-th cycle separates that for the $q$-th cycle from $\Omega$ in $G$. 

Thilikos and Wiederrecht~\cite{TW_arxiv_2022} developed a useful tool that, given a near embedded graph, allows either to `remove' a large vortex from the near embedding, or to find a shallow vortex grid minor of large order in the graph. 
This is Lemma~31 in~\cite{TW_arxiv_2022}. 
While this is not stated explicitly in~\cite{TW_arxiv_2022}, 
it can be checked that, in the second case, the proof in~\cite{TW_arxiv_2022} produces a model of the shallow vortex grid that surrounds the vortex, which will be important for our purposes. 
(The authors are grateful to Sebastian Wiederrecht for helpful discussions regarding this matter.) 
One then can check that applying Lemma~31 from~\cite{TW_arxiv_2022} with this extra ``surrounding'' property to the good near embeddings considered in this paper gives the following result.

\begin{theorem}[{Corollary from \cite[Lemma~31]{TW_arxiv_2022}}]
\label{killing_a_vortex}
For every positive integers $q, \alpha_2$ with $q\le \alpha_2$, there exists a positive integer $s'=s'_{\ref{killing_a_vortex}}(q)$ such that for every positive integers $s\ge s'$ and $h$ with $3\leq s \leq h$ the following holds. 
Let $G$ be a graph that has an $(s,h)$-good $(\alpha_0,\alpha_1,\alpha_2)$-near embedding $(\sigma,G_0,Z,\mathcal{V},\mathcal{W})$ in a surface $\surf$ of Euler genus at most $\alpha$ for some integers $\alpha_0,\alpha_1,\alpha$.
Let $V=(J,\Omega)\in \mathcal{V}$ and let $G_V=G(D_1(V))$, where $D_1(V)$ is defined as in \ref{P:concentric} in the definition of $(s,h)$-good embeddings. 
Then at least one of the following two outcomes holds.
\begin{enumerate}[label=(\alph*)]
\item \label{item:outcome_a}
There is a vertex subset $S$ of $G_V$ with $|S|\le 12 \alpha_2 q$ such that there is an $(s-|S|,h)$-good $(\alpha_0+|S|,\alpha_1-1,\alpha_2)$-near embedding $(\widetilde{\sigma},\widetilde{G}_0,Z\cup S,\mathcal{V}\setminus \{V\},\widetilde{\mathcal{W}})$ of $G$ in $\surf$ such that the embedding $\widetilde{\sigma}$ is an extension of $\sigma$ to the graph $J - S$ and with $\mathcal{W}\subseteq \widetilde{\mathcal{W}}$.

\item \label{item:outcome_b}
$G_V$ contains a model of a shallow vortex grid  of order $q$ that surrounds $V$.  
\end{enumerate}
\end{theorem}

The following lemma shows that a sliced shallow vortex grid of order $t$ with roots in the bottom row of the grid contains a rooted $K_{2,t}$-model.

\begin{lemma}\label{K2t_in_SVGM}
Let $S$ be a sliced shallow vortex grid of order $t$, that is, the Cartesian product of a path $u_1,u_2,\ldots,u_{4t}$ with a path $v_1,v_2,\ldots,v_{t}$ plus the additional edges on the outerface of the grid. 
Let \(R = \{u_1,u_2,\ldots,u_{4t}\} \times \{v_{t}\}\). Then $(S, R)$ contains a rooted $K_{2,t}$-model. 
\end{lemma}

\begin{proof}
Using the additional edges on the boundary, we find two disjoint intertwining paths \(P_1\) and \(P_2\) with vertices in \(\{u_1,u_2,\ldots,u_{4t}\} \times \{v_1, v_2\}\) 
so that \(P_1\) contains all vertices \((u_i, v_2)\) with \(i \bmod 8 \in \{1, 2, 7, 8\}\), and \(P_2\) contains all vertices \((u_i, v_2)\) with \(i \bmod 8 \in \{3, 4, 5, 6\}\). 
For each \(j \in [t]\), let \(B_j = \{u_{4j-3}, u_{4j-2}, u_{4j-1}, u_{4j}\} \times \{v_3, \ldots, v_{t}\}\). Then the desired rooted \(K_{2, t}\)-model has branch sets 
\(V(P_i)\) for \(i \in \{1, 2\}\), and \(B_{i}\) for \(i \in [t]\).
\end{proof}

\begin{lemma}\label{PathsInGrid2}
  Let \(t, m, n\) be positive integers with \(\min\{m, n\} \ge 3t\), and let \(J\) be the \(n \times m\) grid. Let \(X\) denote the set of all \(t^2\) vertices with coordinates \((i, j)\) such that \(1 \le i, j \le t\), and let \(Y\) be any set of \(4t^2\) vertices of \(J\). Then there exist \(2t-1\) vertex-disjoint \(X\)--\(Y\) paths in the grid.  
\end{lemma}
\begin{proof}
  Suppose towards a contradiction, that the lemma does not hold. Then, by Menger's Theorem, there exists a set of vertices \(S\) with \(|S| < 2t-1\) such that after removing \(S\) from the grid, no component intersects both \(X\) and \(Y\).
  For each \(i \in [t]\), let \(P_i'\) denote the path induced by the vertices of the grid with coordinates in the set \(\{i\} \times \{t, \ldots, m\}\),
  and for each \(j \in [t-1]\), let \(P_j''\) denote the path induced by the vertices of the grid with coordinates in the set \(\{t, \ldots, n\} \times \{j\}\).
  Note that the paths \(P_i'\) and \(P_j''\) are all pairwise vertex-disjoint, and each of them intersects \(X\).
  Since \(|S| < 2t-1\), there exists \(P_X \in \{P_1', \ldots, P_t', P_{1}'', \ldots, P_{t-1}''\}\) disjoint from \(S\).

  The number of vertices in the set \(Y\) is at most the number of rows of the grid intersecting \(Y\) times the number of columns of the grid intersecting \(Y\).
  As \(|Y| = (2t)^2\), there exists a row or column \(P_Y\) that intersects \(Y\) and that is disjoint from \(S\). 

  Since \(\min\{m, n\} \ge 3t\), and the set \(X\) intersects \(t\) rows and \(t\) columns, there exist a column \(Q'\) and a row \(Q''\) of the grid, that are disjoint from \(X \cup S\).
  Thus, \(P_X \cup P_Y \cup Q' \cup Q''\) is a connected subgraph of the grid which is disjoint from \(S\), and intersects both \(X\) and \(Y\), a contradiction.
\end{proof}

Note that the definition of a transversal set in a wall naturally extends to a transversal set in a grid. For a model $M$ of a shallow vortex grid in some graph $G$, let a \emph{transversal set of $M$} is a set $Y$ of vertices chosen such that no two vertices are taken from the same branch set of $M$, and the branch sets which contain vertices from $Y$ correspond to a transversal set of the grid we obtain by contracting the branch sets of $M$.

\begin{lemma}\label{vortex_free_decomposition}
Let $t, s, h$ be positive integers with $3 \leq s \leq h$ and $h\ge 3t^2$, and let 
$t'=4t+2$. 
Then there exist positive integers $s'=s'(t, s)$ and $k=k(t,s,h)$ such that, letting  $\alpha_0=\alpha_0(t',s')$ and $\alpha=\alpha(t')$ 
be as in \Cref{GMST}, the following holds.  
For every rooted graph $(G,R)$ without a rooted $K_{2,t}$-minor, 
and every $k$-interesting subgraph $H$ of $(G, R)$ with core $X$, 
\begin{enumerate}
\item there exists an $(s,h)$-good $(\alpha_0,0,\alpha)$-near embedding $(\sigma,H_0,Z,\varnothing,\mathcal{W})$ of $H$ in a surface $\surf$ of Euler genus at most $\alpha$, and
\item $H'_0$ contains a good wall $W'_0$ with respect to the embedding that can be lifted to a wall $W_0$ in $H$ and such that for every transversal $Y'$ of $W'_0$ there are $h$ vertex-disjoint $X$--$Y'$ paths in $H$.  
\end{enumerate}
Furthermore, for some function $T''$, there is an algorithm with running time $T''(t,s,h)\cdot n^{O(1)}$ that, given an $n$-vertex rooted graph $(G, R)$ and a $k$-interesting subgraph $H$ of $G$ with core $X$ as above, finds this embedding and the wall $W'_0$.  
\end{lemma}

\begin{proof}
Let 
\begin{align*}
p&=4t^2\\
q&=2t(p+1). 
\end{align*}
It can be checked that the constant $\alpha = \alpha(t')$ from \Cref{GMST} is (much) larger than $q$; in particular 
\[
\alpha \geq q/2, 
\]
which will be used in the proof.  
Let also
\begin{align*}
\beta&=12q\alpha^2+3p\alpha \\
s'&= s+s'_{\ref{killing_a_vortex}}(q)+\beta \\
h''&= h+s'_{\ref{killing_a_vortex}}(q)+\beta \\
h'&=h_{\ref{GMST}}'(t',s',h'') \\ 
k&=k_{\ref{well_linked_wall}}(h'). 
\end{align*}

Observe that $s' \leq h''$ holds since $s\leq h$ (indeed, ensuring that his holds is the only reason for adding the term $s'_{\ref{killing_a_vortex}}(q)$ in the definition of $h''$). 
With these constants, \Cref{embedding_H} implies that 
\begin{enumerate}
\item there exists an $(s',h'')$-good $(\alpha_0,\alpha,\alpha)$-near embedding $(\sigma,H_0,Z,\mathcal{V},\mathcal{W})$ of $H$ in a surface $\surf$ of Euler genus at most $\alpha$, and
\item $H'_0$ contains a good wall $W'_0$ with respect to the embedding that can be lifted to a wall $W_0$ in $H$, 
and such that for every transversal $Y'$ of $W'_0$ there are $h''$ vertex-disjoint $X$--$Y'$ paths in $H$, and there are $h''$ vertex-disjoint $R_H$--$Y'$ paths in $H$, where $R_H = R \cap V(H)$. 
\end{enumerate}

If $\mathcal{V}=\emptyset$ then the above near embedding is the required embedding. Indeed, any $(s',h'')$-good embedding is also $(s,h)$-good since $s'\ge s$ and $h''\ge h$. 
Hence, we may assume that $\mathcal{V} \neq \emptyset$. 

The main idea of the proof below is to show that there is no shallow vortex grid model surrounding a vortex $V\in \mathcal{V}$ such that there are many vertex-disjoint paths between $R_H$ and $\Omega$. 
This is because, if such a structure existed, then we could find a rooted $K_{2,t}$-minor in $(H, R_H)$, and thus also in $(G, R)$, which is a contradiction to our assumption about $(G, R)$. 
Using the non-existence of such models, we use \Cref{killing_a_vortex} to find a small set of vertices, the removal of which ``kills'' the vortex $V$. 
We then iterate over all remaining vortices in $\mathcal{V}$. 

First, assume that there is some vortex $V\in \mathcal{V}$ such that 
\begin{enumerate}
  \item \label{item:SVGM}
  there is a shallow vortex grid model $M$ of order $q$ in $H$ that surrounds $V$, and
  \item \label{item:disjoint_paths}
  there are $p$ vertex-disjoint paths in $H$ between $R_H$ and $p$ distinct branch sets of $M$.
\end{enumerate} 
We will show that the existence of this model and these $p$ paths together imply that $(G, R)$ has a rooted $K_{2,t}$-minor, a contradiction.

In what follows, we use a terminology similar to the one in the proof of \Cref{minor_grasped}. 
We call a {\em feasible path system} a collection of $p$ vertex-disjoint paths in $H$ such that all the paths have one endpoint (called its {\em source}) in $R_H$, and their other endpoint (called its {\em destination}) in distinct branch sets of $M$. 
Note that there is at least one feasible path system in $H$, by our assumption \ref{item:disjoint_paths} above. 
A path $P$ in a feasible path system ${\cal P}$ is {\em good} if the branch set of $M$ containing the destination of $P$ avoids all the other paths in ${\cal P}$, and is {\em bad} otherwise. 
With respect to ${\cal P}$, we classify branch sets of $M$ into four categories: A branch set $B$ of $M$ is 
\begin{itemize}
  \item {\em good} if $B$ contains the destination of a {\em good} path of ${\cal P}$; 
  \item {\em bad} if $B$ contains the destination of a {\em bad} path of ${\cal P}$; 
  \item {\em ugly} if $B$ intersects some path in ${\cal P}$ but $B$ contains no destination of paths in ${\cal P}$; 
  \item {\em free} if $B$ intersects no path in ${\cal P}$. 
\end{itemize}
We let $g({\cal P}), b({\cal P}), u({\cal P}), f({\cal P})$ denote the number of branch sets in $M$ that are respectively good, bad, ugly, and free with respect to ${\cal P}$. 
Finally, we let $i({\cal P})$ denote the sum over every path $P\in {\cal P}$ of the number of branch sets of $M$ intersected by $P$. 

Let ${\cal P}$ be a feasible path system that minimizes $i({\cal P})$. Observe that this implies  
\[
  u({\cal P})=0,
\] 
since otherwise some path $P$ in ${\cal P}$ could be shortened so as to end in an ugly branch set, which would decrease $i({\cal P})$. 
Thus,  
\[
  b({\cal P})+g({\cal P})=|{\cal P}|=p.
\]
The shallow vortex grid of order \(q=2t(p+1)\) contains \(p+1\) disjoint copies of the sliced shallow vortex grid of order \(2t\),  
so one of these copies, which we denote by \(S\), has the property that all branch sets \(M(u)\) with \(u \in V(S)\) are free. 
In our shallow vortex grid of order \(q\), we find a \(4q \times q\) grid $J$ such that the vertex set of \(S\) is in the corner of the grid; 
that is, for every \(i\in [8t]\) and \(j\in [2t]\), the vertex of \(S\) corresponding to \((u_i, v_j)\) coincides with the vertex of \(J\) with coordinates \((i, j)\) (where $u_i, v_j$ refer to the notations from \Cref{K2t_in_SVGM}).  
Let \(S'\) be a sliced shallow vortex grid of order \(t\) contained in \(S\), which contains the same corner of \(J\) as \(S\), and thus, is induced by the vertices \((u_i, v_j)\) of \(S\) with \(i\leq 4t\) and \(j \le t\). 
We will build a rooted \(K_{2,t}\)-model in $(G, R)$ 
by rerouting some paths in \(\cal P\) so that they hit \(S'\) in a way that allows applying \Cref{K2t_in_SVGM}.

Let \(U\) denote the set of all \(p=4t^2\) vertices \(u \in V(J)\) such that the branch set \(M(u)\) contains the destination of a path of \(\cal P\) (i.e., such that \(M(u)\) is good or bad).

We will now focus on paths in the grid \(J\), and then we will use the model \(M\) to build a rooted \(K_{2, t}\)-minor in $(G, R)$. 
By \Cref{PathsInGrid2}, there exists a set \(\mathcal{Q}\) of \(2t-1\) disjoint \(U\)--\(V(S')\) paths in \(J\). 
Recall that the sliced shallow vortex grid \(S'\) has order \(t\), and that its vertex set resides in the corner of \(J\) and is disjoint from \(U\), so each path in \(\mathcal{Q}\) has its destination in one of the \(2t-1\) 
vertices of \(S'\) that are adjacent to vertices outside \(S'\) in \(J\). Since \(J\) is bipartite, we can find \(Q_1, \ldots, Q_t \in \mathcal{Q}\) such that if \(w_c\) denotes the end of \(Q_c\) in \(U\), then the vertices \(w_1, \ldots, w_t\) are pairwise nonadjacent in \(J\).

Let \(s^*\) denote the vertex of \(S\) that corresponds to the corner \((u_{8t}, v_{2t})\), and 
for each \(c \in [t]\), let \(s_c'\) denote the end of \(Q_c\) in \(V(S')\), and let \(s_c\) be the vertex of \(Q_c\) belonging to \(V(S) \setminus \{s^*\}\) that is closest to \(w_c\) on \(Q_c\) (so that the \(s_c\)--\(w_c\) subpath of \(Q_c\)
intersects \(V(S) \setminus \{s^*\}\) only in \(s_c\)).
Thus, each vertex \(s_c\) corresponds to some vertex \((u_i, v_j)\) with either $i=8t$ or $j=2t$. 
See \Cref{fig:lemma63} for an illustration. 
 \begin{figure}
    \centering
    \includegraphics[width=0.6\textwidth]{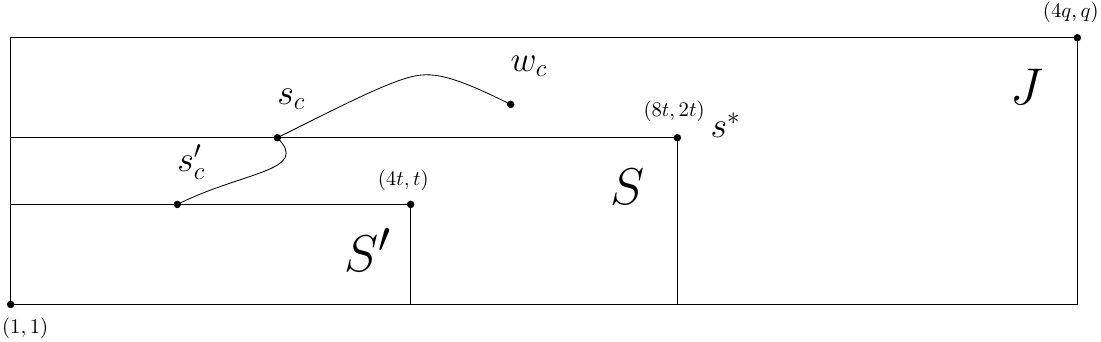}
    \caption{An illustration of $J,S,S'$ and a path $Q_c$.}
    \label{fig:lemma63}
  \end{figure}

Let \(R_{S'}\) denote the set of the vertices of \(S'\) corresponding to the vertices \((u_i, v_j)\) with \(j=t\).
After reordering the paths \(Q_1, \ldots, Q_t\) and rerouting them using the buffer zone of \(S\), we may assume that each \(s_c'\) is the vertex corresponding to the vertex \((u_c, v_t) \in R_{S'}\) of \(S'\). 

Let \(c \in [t]\), and let \(w_c'\) denote the vertex of \(Q_c\) adjacent to its end \(w_c\). Thus, the graph \(H\) contains an edge between \(M(w_c)\) and \(M(w_c')\). 
Since \(w_c \in U\), the branch set \(M(w_c)\) intersects at least one path from \(\cal P\). Let us pick a path \(P_c \in \cal P\) such that the subgraph of \(H\) induced by \(M(w_c)\) contains a path \(P'_c\) such that one end of \(P'_c\) is in \(P_c\), the other end of \(P'_c\) is a vertex adjacent to \(M(w_c')\), and no internal vertex of \(P'_c\) is in any of the paths in \(\cal P\).  

Since the vertices \(w_1, \ldots, w_t\) are pairwise non-adjacent in \(J\), the paths \(P_1, \ldots, P_t\) are pairwise distinct: Indeed, if some \(P_{c_1}\) and \(P_{c_2}\) with \(c_1 \neq c_2\) were equal, then we could reroute that path to either \(M(w_{c_1}')\) or \(M(w_{c_2}')\) to make $i({\cal P})$ smaller.

Let \(Q_1^*, \ldots, Q_t^*\) be vertex-disjoint paths in \(H\) such that each \(Q_c^*\) is an \(M(s_c')\)--\(R_H\) path contained in \(\bigcup_{w \in V(Q_c)} M(w) \cup V(P_c') \cup V(P_c)\). 
By \Cref{K2t_in_SVGM}, there exists a rooted \(K_{2,t}\)-model \(M'\) in \((S', R_{S'})\). We can then transform it into a rooted \(K_{2, t}\)-model \(M''\) in \((G, R)\), where for each \(a \in V(K_{2, t})\),
the branch set \(M''(a)\) is defined as the union of all branch sets \(M(u)\) of the shallow vortex grid model with \(u \in M'(a)\), and the union over all \(c \in [t]\) with \(s_c' \in M'(a)\) of the vertex sets of the paths \(Q_c^*\). The existence of this rooted \(K_{2,t}\)-model is a contradiction.

Therefore, we may assume that for each vortex $V\in \mathcal{V}$, at least one of the two conditions \ref{item:SVGM} and \ref{item:disjoint_paths} above fails to hold.  

Let $V\in \mathcal{V}$.  
If \ref{item:SVGM} does not hold for $V$, then we simply update the current near embedding using \ref{item:outcome_a} in \Cref{killing_a_vortex}, allowing thus to get rid off $V$, at the price of a small increase in the size of the apex set (at most $12q\alpha$), and a small decrease in the first parameter for the goodness of the near embedding, which is perfectly fine as $s'$ was chosen large enough at the beginning to have room for up to $\alpha$ such reductions. 
This way, we may remove iteratively all vortices  $V\in \mathcal{V}$ for which \ref{item:SVGM} does not hold, updating the near embedding each time.

Next, assume that there is a vortex $V=(J,\Omega)\in \mathcal{V}$ for which condition \ref{item:SVGM} holds but not condition \ref{item:disjoint_paths}. 
As we explain below, in this case there is a transversal set $Y'$ of $\widetilde{W}_0$, where $\widetilde{W}_0$ is the subwall of $W_0$ obtained after the previous steps of the near embedding updates, 
such that there is a separation $(A,B)$ of $H$ where $\Omega\subseteq A$, $Y'\subseteq B$ and $|A\cap B|<3p$. In this case we update the near embedding as follows. 
We add $A\cap B$ to the apex set, remove $V$ from $\mathcal{V}$ and add $(A,\emptyset)$ to $\mathcal{W}$. 
Again, this results in a small increase of the apex set, and a decrease of at most $3p$ in the first parameter for the goodness of the near embedding, which we have budgeted for thanks to the initial choice of $s'$. 
Proceeding iteratively this way with every remaining vortex in $\mathcal{V}$, we eventually get rid of all of them, as desired. 

Hence, it only remains to show the existence of the transversal set $Y'$ of $\widetilde{W}_0$ and the separation $(A,B)$ of $H$ mentioned above.  
Assume to the contrary that they do not exist, that is, for each transversal set $Y'$ in $\widetilde{W}_0$ there is no such separation $(A,B)$ of $H$. 
We will show that condition \ref{item:disjoint_paths} then holds for $V$, which is a contradiction. 
Let $M$ be a model of a shallow vortex grid of order $q$ that surrounds $V$, as in \ref{item:SVGM}. 
Let $Y'$ be a transversal set in $\widetilde{W}_0$ and let $\mathcal{P}'$ be a collection of at least $3p$ vertex-disjoint paths in $H$ between $Y'$ and $\Omega$ (they exist by Menger's theorem and our assumption that there is no separation between them of a smaller size). 
Using the facts that $D(\widetilde{W}'_0)$ does not intersect $D_1(V)$ and that $M$ surrounds $V$, we deduce that each path $P\in \mathcal{P}'$ intersects the union of the vertex sets of the branch sets in the model that correspond to vertices of the $i$-th cycle for each $i\in [q]$, we refer to this collection as the \emph{$i$-th cycle of the model of $M$} or simply \emph{a cycle of $M$}. 
Indeed, if $P$ avoids some cycle of $M$, then this cycle would separate $P$ from $\Omega$, in contradiction to the definition of $P$.  

Let $Y''$ be some transversal set in $M$.  
Next, we show that there are $3p$ vertex-disjoint paths between $Y'$ and $Y''$. 
Arguing by contradiction, assume otherwise, and let $(A,B)$ be a separation of $H$ such that  $Y'\subseteq B, Y''\subseteq A$ and $|A\cap B|\le 3p-1$. As $|\mathcal{P}'|\ge 3p$ and $M$ is of order $q>3p$, there is a path $P\in \mathcal{P}'$ and a cycle $C$ of $M$ that are missed by $A\cap B$. 
The path $P$ contains a vertex from $Y'$, 
the cycle $C$ contains a vertex from $Y''$, 
and $P$ intersects $C$. 
We deduce that $A\cap B$ does not separate $Y'$ from $Y''$ in $H$, a contradiction.  

Finally, we show that there are $p$ vertex-disjoint paths between $R_H$ and $Y''$, so that \ref{item:disjoint_paths} holds. Assume that this is not the case, and let $(A,B)$ be a separation of $H$ such that $Y''\subseteq A, R_H\subseteq B$ and $|A\cap B|<p$. This implies that $|A\cap Y'|\ge 2p$ but this contradicts the fact that there are at least $|Y'|\ge 3p$ vertex-disjoint paths between $R_H$ and $Y'$. 
\end{proof}

\subsection{Further refinement of the embedding}
For two near embeddings \((\sigma, G_0, A, \mathcal{V}, \mathcal{W})\) and \((\tilde{\sigma}, \tilde{G_0}, \tilde{A}, \tilde{\mathcal{V}}, \tilde{\mathcal{W}})\) of a graph \(G\) in a surface \(\Sigma\), we say that \((\tilde{\sigma}, \tilde{G_0}, \tilde{A}, \tilde{\mathcal{V}}, \tilde{\mathcal{W}})\) is \emph{finer} than \((\sigma, G_0, A, \mathcal{V}, \mathcal{W})\) if \(G_0 \subseteq \tilde{G_0}\), \(\tilde{\sigma}\) extends \(\sigma\), \(A = \tilde{A}\), \(\mathcal{V} = \tilde{\mathcal{V}}\), and for every \(\tilde{V} \in \tilde{\mathcal{W}}\), there is \(V \in {\mathcal{W}}\) with \(\tilde{V} \subseteq V\); 
moreover, it is  \emph{strictly finer}  if the  two near embeddings are not the same. 
A small vortex \((J, \Omega)\in \mathcal{W}\) is \emph{properly attached} if \(J\) is connected, every component of \(J - \Omega\) is adjacent to all vertices in \(\Omega\), and either \(|\Omega| \le 2\), or \(V(J) \setminus \Omega \neq \emptyset\) and for every three distinct vertices \(x, y, z \in \Omega\) and \(w \in V(J) \setminus \Omega\), there are an \(x\)--\(y\) path and a \(z\)--\(w\) path in \(J\) that are vertex disjoint. Note that our notion of proper attachment is stronger than the one in~\cite{DiestelKMW12}.

\begin{lemma}\label{ProperAttachments}
  For every near embedding \((\sigma, G_0, A, \mathcal{V}, \mathcal{W})\) of a graph \(G\) in a surface \(\Sigma\), there exists a finer near embedding \((\tilde{\sigma}, \tilde{G_0}, \tilde{A}, \tilde{\mathcal{V}}, \tilde{\mathcal{W}})\) of \(G\) in \(\Sigma\) such that every small vortex \(\tilde{V} \in \tilde{\mathcal{W}}\) is properly attached.
  Furthermore, the finer near embedding can be found in polynomial time given $G$ and \((\sigma, G_0, A, \mathcal{V}, \mathcal{W})\) as input. 
\end{lemma}

To prove \Cref{ProperAttachments} we need the following theorem.

\begin{theorem}[Two Disjoint Paths Theorem~\cite{jung_1970,RS90,seymour_1980a,shiloach_1980,thomassen_1980}]\label{TwoPaths}
  Let \(x,y,z,w\) be distinct vertices in a graph \(J\). Then either there exist an \(x\)--\(y\) path and a \(z\)--\(w\) path in \(J\) that are vertex disjoint, or there is a near embedding \((\sigma^*, G_0^*, A^*, \mathcal{V}^*, \mathcal{W}^*)\) of \(J\) into a disk such that \(A^* = \emptyset\),  \(\mathcal{V}^*=\emptyset\), and the vertices \(x, w, y, z\) are mapped by \(\sigma^*\) into the boundary of the disk in that cyclic order. 
  Furthermore, there is a polynomial-time algorithm that finds either the two disjoint paths or the near embedding. 
\end{theorem}

We now turn to the proof of \Cref{ProperAttachments}. 

\begin{proof}[Proof of \Cref{ProperAttachments}] 
  Let \((\tilde{\sigma}, \tilde{G_0}, \tilde{A}, \tilde{\mathcal{V}}, \tilde{\mathcal{W}})\) be a near embedding of \(G\) in \(\Sigma\) that is finer than \((\sigma, G_0, A, \mathcal{V}, \mathcal{W})\), and such that there is no near embedding of \(G\) in \(\Sigma\) that is strictly finer than this one. 
  (Observe that there is such a near embedding, since there cannot be an infinite sequence of near embeddings of \(G\) in \(\Sigma\) such that each one is strictly finer than the previous one.)

  We claim that \((\tilde{\sigma}, \tilde{G_0}, \tilde{A}, \tilde{\mathcal{V}}, \tilde{\mathcal{W}})\) satisfies the lemma. 
  Suppose towards a contradiction that some vortex \((J, \Omega) \in \tilde{\mathcal{W}}\) is not properly attached.

  If \(J\) is disconnected, say \(J\) is the disjoint union of nonempty graphs \(J_1\) and \(J_2\), we can obtain a strictly finer near embedding by replacing \((J, \Omega)\)
  by two small vortices \((J_1, \Omega_1)\) and \((J_2, \Omega_2)\) where each \(\Omega_i\) is the restriction of \(\Omega\) to \(\Omega \cap V(J_i)\), a contradiction.  
  (Observe that we can find a disk in $\Sigma$ accommodating \((J_i, \Omega_i)\) for $i=1,2$ since $|\Omega_i|\leq 2$.)

  If \(V(J) \setminus \Omega = \emptyset\), then we obtain a strictly finer near embedding by removing the small vortex \((J, \Omega)\) from \(\tilde{\mathcal{V}}\), replacing
  \(G_0\) by \(G_0 \cup J\), and extending \(\tilde{\sigma}\) to an embedding of \(G_0 \cup J\) by drawing the edges of \(J\) along the boundary of the disk \(\Delta\) corresponding to \((J, \Omega)\), which is a contradiction.
  Thus, \(V(J) \setminus \Omega \neq \emptyset.\)

  If some connected component $C$ of $J - \Omega$ is only adjacent to a strict subset $\Omega_1$ of the vertices in $\Omega$, 
  then letting $J_1 = J[V(C) \cup \Omega_1]$, we can obtain a strictly finer near embedding by replacing \((J, \Omega)\)
  by two small vortices \((J_1, \Omega_1)\) and \((J - V(C), \Omega)\), a contradiction.  
  Thus, every connected component of $J - \Omega$ is adjacent to all vertices in $\Omega$.

  Finally, suppose that there are distinct vertices \(x,y,z \in \Omega\) and \(w\in V(J)\setminus\Omega\) such that there do not exist an \(x\)--\(y\) path and a \(z\)--\(w\) path in \(J\) that are vertex disjoint.
  Then, by \Cref{TwoPaths}, there exists a near embedding of \(J\) in a disk with no apices and no large vortices such that the vertices \(x,w,y,z\) appear on the boundary of the disk in that cyclic order.
   In such case we can make the near embedding strictly finer by replacing the disk \(\Delta_i\) representing \((J, \Omega)\) with the near embedding of \(J\) in the disk.

   We remark that this proof can be turned into a polynomial-time algorithm by iteratively applying each of the improvement steps above, and observing that the total number of improvements that can be made is bounded by a polynomial in $|V(G)|$. 
 \end{proof}

\subsection{Ensuring \(3\)-connectivity}
In this subsection, we describe how to transform a given near embedding \((\sigma, G_0, A, \mathcal{V}, \mathcal{W})\) of a graph $G$ in surface $\Sigma$ with a large flat wall \(W\) 
so as to make  \(G_0'\) \(3\)-connected while keeping many properties of the original near embedding (in particular, we want that a large subwall of \(W\) is still flat). 
We use the notion of SPQR trees that has been introduced by Di Battista \& Tamassia~\cite{battista_1989}. We remark that there is a linear-time construction algorithm due to Gutwenger \& Mutzel~\cite{gutwenger_2000}.

In what follows, \emph{multigraphs} may have parallel edges between pairs of vertices, but no loops.
Let \(G_1\) and \(G_2\) be multigraphs such that \(V(G_1) \cap V(G_2) = \{x, y\}\) for some distinct vertices \(x\) and \(y\), and each of \(G_1\) and \(G_2\) has at least one edge between \(x\) and \(y\). Then, the \emph{\(2\)-sum} of \(G_1\) and \(G_2\) is the graph \(G\) obtained by removing one edge between \(x\) and \(y\) from each of \(G_1\) and \(G_2\), and taking the union of the resulting subgraphs. A \emph{dipole} is a multigraph with \(2\) vertices (and any number of parallel edges between them). For instance, the graph \(K_{2, t}\) can be obtained by \(2\)-summing a dipole with \(t\) edges with \(t\) triangles.

A multigraph is \emph{\(2\)-connected} if every edge lies on a cycle (possibly of length \(2\)). For every \(2\)-connected multigraph \(G\), the \emph{SPQR tree} of \(G\) is a canonical representation of \(G\) as the \(2\)-sum of multigraphs of three types: dipoles, cycles (without parallel edges) and \(3\)-connected graphs (without parallel edges). We do not provide a full definition of the SPQR tree, and only describe its properties required in our proof (see \cite{battista_1989,gutwenger_2000} for more background). The SPQR tree of \(G\) is uniquely defined. It is a tree \(T\) in which every node \(u\) is associated with a graph \(G_u\), and has one of four types
\begin{itemize}
  \item if \(u\) is an \emph{S node}, \(G_u\) is a cycle,
  \item if \(u\) is a \emph{P node}, \(G_u\) is a dipole with at least three edges,
  \item if \(u\) is a \emph{Q node}, \(G_u\) is a dipole with two edges, and
  \item if \(u\) is an \emph{R node}, \(G_u\) is a \(3\)-connected graph. 
\end{itemize}
For each edge \(uv \in E(T)\), the graphs \(G_u\) and \(G_v\) share two common vertices that are adjacent in both graphs, and thus their \(2\)-sum is well defined. 
We can contract any edge \(u v\) of $T$ into a new vertex \(w\) and associate with it the graph \(G_w\) which is the \(2\)-sum of \(G_u\) and \(G_v\).
If we contract the edges of \(T\) in any order, then regardless the order in which we contract the edges, the graph \(G_r\) associated with the node \(r\) into which we contracted the tree is equal to \(G\). Since a dipole with \(2\) edges is a neutral element for \(2\)-sum, if the SPQR tree has a \(Q\) node, then it is the only node, and \(G\) is a dipole with \(2\) edges. 

\begin{theorem}[Mohar~\cite{mohar_1997}]\label{thm:SPQRfacewidth}
Let \(\Sigma\) be a surface with strictly positive Euler genus, and let \(G\) be a \(2\)-connected graph that can embedded in \(\Sigma\) with facewidth \(w \ge 3\).
Then, there is exactly one node \(u\) in the SPQR tree of \(G\) such that \(G_u\) is not planar. Further, \(G_u\) can be embedded in \(\Sigma\) with facewidth \(w\).
\end{theorem}

\begin{lemma}\label{ensuring_3c}
Let \(r \ge 3\), let \(G\) be a graph, let \((\sigma, G_0, Z, \mathcal{V}, \mathcal{W})\) be an $(\alpha,0,\alpha)$-near embedding of \(G\) in a surface \(\Sigma\) with all small vortices properly attached, let \(W_0'\) be an \((r+2)\)-wall in \(G_0'\) that can be lifted to an \((r+2)\)-wall \(W\) in \(G - Z\). Then, there is an $(\alpha,0,\alpha)$-near embedding \((\tilde{\sigma}, \tilde{G_0}, \tilde{Z}, \tilde{\mathcal{V}}, \tilde{\mathcal{W}})\) of \(G\) in a surface \(\tilde{\Sigma}\) with \(\eg(\tilde{\Sigma}) \le \eg(\Sigma)\) such that the facewidth of \(\tilde{\sigma}\) is not smaller than the facewidth of \(\sigma\), all small vortices are properly attached,  \(\tilde{G_0}'\) contains an \(r\)-wall \(\tilde{W}_0'\) that can be lifted to an \(r\)-subwall \(\tilde{W}\) of \(W\), and \(\tilde{G_0}'\) is \(3\)-connected.
Furthermore,  the near embedding can be found in polynomial time given $G$ and \((\sigma, G_0, Z, \mathcal{V}, \mathcal{W})\) as input. 
\end{lemma}
\begin{proof}
  First, suppose that \(G_0'\) is not connected, and let \(J_0'\) be a connected component of \(G_0'\) distinct from the one containing \(W_0'\).
  Let \(\mathcal{W}'\) denote the set of all small vortices \((J, \Omega) \in \mathcal{W}\) with \(\Omega \cap V(J_0') \neq \emptyset\) (and thus \(\Omega \subseteq V(J_0')\)),
  and let \(J_1\) denote the union of \(J_0' \cap G_0\) and all small vortices in \(\mathcal{W}'\). Then, after replacing \(G_0\) with \(G_0 - V(J_0')\) and \(\mathcal{W}\) with \((\mathcal{W} \setminus \mathcal{W}') \cup \{(J_1,\emptyset)\}\), we obtain a near embedding, where \(J_0'\) is no longer a connected component of \(G_0'\). Hence, after removing similarly all connected components of \(G_0'\) which do not contain \(W_0'\), we may assume that \(G_0'\) is connected.

  Similarly, we argue that we may assume \(G_0'\) to be \(2\)-connected.
  Suppose that \(G_0'\) is not \(2\)-connected.
  Let \(G_1'\) be the block of \(G_0'\) containing \(W_0'\), and let \(x\) be a cut-vertex of \(G_0'\) contained in \(G_1'\).
  Let \(J_0'\) denote the subgraph of \(G_0'\) induced by \(x\) and all connected components of \(G_0' - V(G_1')\) that are adjacent to \(x\).
  Let \(\mathcal{W}'\) denote the set of all small vortices \((J, \Omega) \in \mathcal{W}\) with \(\emptyset \neq \Omega \subseteq V(J_0')\), and
  let \(J_1\) denote the union of \(J_0' \cap G_0\) and all small vortices in \(\mathcal{W}'\).
  Then, after replacing \(G_0\) with \(G_0 - (V(J_0') \setminus \{x\})\) and \(\mathcal{W}\) with \((\mathcal{W} \setminus \mathcal{W}') \cup \{(J_1,\{x\})\}\), we obtain a near embedding, where \(x\) is no longer a cut-vertex of \(G_0'\). Hence, we may assume that \(G_0'\) is \(2\)-connected.

  Let \(T\) be the SPQR tree of \(G_0'\). With each edge \(u_1 u_2 \in E(T)\), we associate a separation \((A_1, A_2)\) of \(G_0'\), defined as follows. For each \(i \in \{1, 2\}\), let \(T_i\) denote the connected component of \(T - u_1 u_2\) that contains \(u_i\), and let \(A_i = \bigcup_{u \in V(T_i)} V(G_u)\) 
  (where $G_u$ denotes the graph corresponding to node $u$ in the SPQR tree of \(G_0'\)). 
  Since \(W_0'\) is a subdivision of a \(3\)-connected graph, for some \(i \in \{1, 2\}\), the set of vertices of \(W_0'\) in \(A_i\) is either empty or induces a path, and \(A_{3-i}\) contains all branch vertices of \(W_0'\). We refer to \(A_i\) as the \emph{small side} and to \(A_{3-i}\) as the \emph{big side}. Let us orient every edge
  \(u_1 u_2\) towards the vertex \(u_i\) such that \(A_i\) is the big side of the separation, and let \(u_0\) be a sink in this orientation. Hence, \(V(G_{u_0})\) contains all branch vertices of \(W_0'\).
  Moreover, the graph \(G_{u_0}\) contains a graph obtained from \(W_0'\) by suppressing some degree-\(2\) vertices. In particular, \(G_{u_0}\) contains an \(r\)-wall (which can be lifted to a subwall of \(W\)), so \(u_0\) must be an R node of the SPQR tree.

  To complete the proof, we will construct a near embedding with \(\tilde{G_0} = G_0 \cap G_{u_0}\). This choice uniquely determines the small vortices, and we will have \(G_0' = G_{u_0}\), so \(G_0'\) will be \(3\)-connected. By \Cref{thm:SPQRfacewidth}, either \(G_u\) has an embedding in \(\Sigma\) with the same facewidth as \(\sigma\), or \(G_u\) is planar (and thus has an embedding in the sphere with infinite facewidth). In both cases, there exists a surface \(\tilde{\Sigma}\) with Euler genus not greater than the one of \(\Sigma\) and an embedding \(\tilde{\sigma}\) of \(G_u\) in \(\tilde{\Sigma}\) with facewidth not smaller than the one of \(\sigma\). Let \(\tilde{G_0} = G_0 \cap G_{u_0}\), let \(\tilde{Z} = Z\), let \(\tilde{\mathcal{V}} = \mathcal{V} = \emptyset\), and let \(\tilde{\mathcal{W}}\) be defined as follows: for each small vortex \((J, \Omega) \in \mathcal{W}\) with \(\Omega \subseteq V(G_{u_0})\), include it in \(\tilde{\mathcal{W}}\). Next, turn each connected component \(T'\) of \(T - u_0\) into a small vortex \((J', \Omega')\) defined as follows: Let \(J'_0\) denote the subgraph of \(G_0\) induced by all vertices in \(\bigcup_{u \in V(T')} V(G_u)\), and let \(J'\) denote the union of \(J_0'\) and all small vortices \((J, \Omega) \in \mathcal{W}\) with \(\Omega \subseteq V(J_0')\). Finally, let \(\Omega'\) denote the only cyclic ordering of the two-element set \(V(G_{u_0}) \cap V(J_0')\). 
  This defines the set of small vortices \(\tilde{\mathcal{W}}\).
  Note that neither the Euler genus nor $Z$ were increased during the construction.
  Therefore, we have \(\tilde{G_0} = G_{u_0}\), so \((\tilde{\sigma}, \tilde{G_0}, \tilde{Z}, \tilde{\mathcal{V}}, \tilde{\mathcal{W}})\) is an $(\alpha,0,\alpha)$-near embedding satisfying the lemma.
\end{proof}

\subsection{Constraining the roots}\label{sec:3-connectivity}

In this section, we finally prove \Cref{interesting}. By \Cref{vortex_free_decomposition}, 
for every rooted graph $(G,R)$ without a rooted $K_{2,t}$-minor and every $k$-interesting subgraph $H$ of $(G, R)$ with core $X$, there exists an \((s, h)\)-good \((\alpha, 0, \alpha)\)-near embedding \((\sigma, G_0, A, \mathcal{V}, \mathcal{W})\) in a surface of Euler genus at most \(\alpha\) with a flat wall \(W'\) of height \(h\) in \(G_0'\) that can be lifted to a wall \(W_0'\) in \(G\), and there are \(h\) vertex-disjoint paths between \(X\) and any transversal set of \(W_0'\). It remains to show that
\begin{enumerate}
  \item\label{g0'rc} only a bounded number of small vortices in \(\mathcal{W}\) contain a root from the set \(R\) that is not in \(G_0'\), and 
  \item\label{g0'fc} the roots in \(G_0'\) have a small face cover.
\end{enumerate}
First, we will prove that \ref{g0'rc} holds for the small vortices that have at most two vertices in \(G_0'\), and then we will use \Cref{BKMM} to simultanously handle the small vortices with three vertices in \(G_0'\) and prove \ref{g0'fc}.

\begin{lemma}\label{EasterMondayLemma}
  Let \(t \ge 2\) be an integer, let \(G\) be a \(3\)-connected graph, and let \(D \subseteq E(G)\) be a set of edges with \(|D| \ge 3t -3\). Then there exists a subset \(D' \subseteq D\) with \(|D'|=t\) such that \(G - D'\) is connected.
\end{lemma}
\begin{proof}
  We assume that \(G - D\) is disconnected since otherwise the lemma is satisfied by any set \(D' \subseteq D\) with \(|D'| = t\).
  Let \(c \ge 2\) denote the number of components of \(G - D\). Since \(G\) is \(3\)-connected, each component of \(G - D\) is linked to the remaining components with at least three edges from \(D\), so \(|D| \ge \frac32c\). Combining this with the assumption \(|D| \ge 3t-3\), we deduce that \(|D| \ge \frac23(\frac32c) + \frac13(3t-3) = c + t-1\). 
  Let \(D_0 \subseteq D\) be a set of \(c-1\) edges such that \((G - D) + D_0\) is connected. We have \(|D \setminus D_0| = |D| - |D_0| \ge (c+t-1)-(c-1)=t\).
  Let \(D' \subseteq D \setminus D_0\) be any subset with \(|D'| = t\). Then \(G -D'\) is a supergraph of the graph \((G - D) + D_0\), so \(D'\) satisfies the lemma.
\end{proof}

\begin{lemma}\label{Omega2}
  Let \((G, R)\) be a \(3\)-connected rooted graph without a rooted \(K_{2, t}\)-minor, let \(H\) be a \(k\)-interesting subgraph, and let \((\sigma,H_0,Z,\mathcal{V},\mathcal{W})\) be an $(\alpha,0,\alpha)$-near embedding of \(H\) with \(H_0'\) \(3\)-connected and all small vortices properly attached. Then there are at most \(k+(t-1)\binom{\alpha}{2}+(3t-4)\alpha\) small vortices \((J, \Omega)\in \mathcal{W}\) with \(|\Omega| \le 2\) and \((V(J) - \Omega) \cap R \neq \emptyset\).
\end{lemma}
\begin{proof}
  Partition the set of all small vortices \((J, \Omega)\in \mathcal{W}\) with \(|\Omega| \le 2\) and \((V(J) - \Omega) \cap R \neq \emptyset\) into three sets \(\mathcal{W}_{\partial}\), \(\mathcal{W}_{\le1}\), and \(\mathcal{W}_{2}\), by assigning \((J, \Omega)\) to \(\mathcal{W}_{\partial}\) if \((V(J) - \Omega) \cap \partial(H) \neq \emptyset\), to \(\mathcal{W}_{\le1}\) if \((V(J) - \Omega) \cap \partial(H) = \emptyset\) and \(|\Omega| \le 1\), and to \(\mathcal{W}_2\) if \((V(J) - \Omega) \cap \partial(H) = \emptyset\) and \(|\Omega| = 2\).
  Since \(H\) is \(k\)-interesting, we have \(|\partial(H)| \le k\), and since the graphs \(J - \Omega\) are pairwise disjoint for \((J, \Omega) \in \mathcal{W}\), we have \(|\mathcal{W}_{\partial}| \le k\).

  Next, we claim that \(|\mathcal{W}_{\le 1}| \le (t-1)\binom{\alpha}{2}\). Suppose to the contrary that \(|\mathcal{W}_{\le 1}| \ge t\binom{\alpha}{2}\). 
  For each \((J, \Omega)\in \mathcal{W}_{\le 1}\), choose a component \(C(J, \Omega)\) of \(J - \Omega\) that intersects \(R\). Since \(G\) is \(3\)-connected, \(|\Omega| \le 1\), and \((V(J) - \Omega) \cap \partial(H) = \emptyset\), the component \(C(J, \Omega)\) must be adjacent to at least two vertices in \(Z\). We have \(|Z| \le \alpha\), so by the pigeonhole principle, there exist \(t\) small vortices \((J_1, \Omega_1), \ldots, (J_t, \Omega_t) \in \mathcal{W}_{\le 1}\) and distinct vertices \(z_1, z_2 \in Z\) such that \(C(J_i, \Omega_i)\) is adjacent to both \(z_1\) and \(z_2\) for each \(i \in [t]\). Therefore, \(K_{2, t}\) is a rooted minor of \((G, R)\), as witnessed by the model with branch sets \(\{z_i\}\) for \(i \in \{1, 2\}\), and \(V(C(J_j, \Omega_j))\) for \(j \in [t]\), a contradiction.

  Finally, we show that \(\mathcal{W}_2\) has size at most $(3t-4)\alpha$.
  For each  \((J, \Omega) \in \mathcal{W}_2\), choose a component \(C(J, \Omega)\) of \(J - \Omega\) that intersects the root set \(R\). 
  Since \((J, \Omega)\) is properly attached, \(C(J, \Omega)\) is adjacent to both vertices in \(\Omega\).
  Since \(G\) is \(3\)-connected and \((V(J) - \Omega) \cap \partial(H) = \emptyset\), the component \(C(J, \Omega)\) is adjacent to a vertex in \(Z\).
  Suppose towards a contradiction that \(|\mathcal{W}_2| > (3t-4)\alpha\). 
  By the pigeonhole principle, there exist a subset \(\mathcal{W}'_2 \subseteq \mathcal{W}_2\) and a vertex \(z \in Z\) such that \(|\mathcal{W}'_2| \ge 3t-3\) and for each \((J, \Omega) \in \mathcal{W}'_2\), the component \(C(J, \Omega)\) is adjacent to \(z\). 
  Let \(D\) denote the set of all edges \(xy \in E(H_0')\) such that \(\Omega = \{x, y\}\) for some \((J, \Omega) \in \mathcal{W}'_2\). 
  By \Cref{EasterMondayLemma} (note that we can assume that $t\ge2$ without loss of generality), there exists a subset \(D' \subseteq D\) with \(|D'| = t\), such that \(H_0' - D'\) is connected.
  Let \((J_1, \Omega_1), \ldots, (J_t, \Omega_t)\) be distinct small vortices in \(\mathcal{W}'_2\) such that each \(\Omega_i\) consists of the endpoints of a different edge in \(D'\).
  Let \(H'\) denote the union of $H_0$ with all small vortices \((J, \Omega) \in \mathcal{W}\) except those with \(\Omega = \{x, y\}\) for some \(xy\in D'\), or $\Omega = \emptyset$. 
  Since the small vortices are properly attached, the graph \(H'\) is connected analogously to \(H_0' - D'\). 
  In this case we obtained a rooted model of \(K_{2,t}\) in \((G, R)\), with branch sets \(\{z\}\), \(V(H')\), and \(V(C(J_{i}, \Omega_{i}))\) for \(i \in [t]\), contradiction.
\end{proof}

\begin{lemma}\label{BKMMappl}
  Let \((G, R)\) be a \(3\)-connected rooted graph without a rooted \(K_{2,t}\)-minor, let \(H\) be a \(k\)-interesting subgraph, and let \((\sigma,H_0,Z,\mathcal{V},\mathcal{W})\) be a near embedding of \(H\) in a surface \(\Sigma\) with Euler genus at most \(g\) such that \(H_0'\) is \(3\)-connected and all small vortices are properly attached. Suppose further that the facewidth of \(\sigma\) is at least \(f_{\ref{BKMM}}(g, t)\). Then, there are at most \(f_{\ref{BKMM}}(g, t)\) small vortices \((J, \Omega) \in \mathcal{W}\) such that \(|\Omega| = 3\) and \((V(J) - \Omega) \cap R \neq \emptyset\),
  and there is a set of at most \(f_{\ref{BKMM}}(g, t)\) facial cycles in \(\sigma\) that cover the set \(R \cap V(H_0)\).

  Furthermore, for fixed $t$, the set of facial cycles can be found in polynomial time given $(G,R)$, $H$, and \((\sigma,H_0,Z,\mathcal{V},\mathcal{W})\) as input. 
\end{lemma}
\begin{proof}
  Let \(H_0''\) be obtained from \(H_0'\) by adding, for each \(V = (J, \Omega) \in \mathcal{W}\) with \(|\Omega|=3\) and \((V(J) - \Omega) \cap R \neq \emptyset\), a new vertex \(x_V\) that is adjacent to all vertices in \(\Omega\), and extend the embedding \(\sigma\) of \(H_0'\) to an embedding \(\sigma'\) of \(H_0''\) by embedding each \(x_V\) in the interior of the triangular face containing the three vertices in \(\Omega\). Note that \(H_0''\) is still \(3\)-connected, and the facewidth of \(\sigma'\) is at least \(f_{\ref{BKMM}}(g, t)\). Let \(R'\) be the union of \(R \cap V(H_0)\) and the set of all
  new vertices \(x_V\). Since the small vortices are properly attached, we can see that there is no rooted \(K_{2,t}\)-minor in \((H_0'', R')\) (otherwise we could lift it to a rooted $K_{2,t}$-minor in $(G,R)$). By \Cref{BKMM}, there is a set of at most \(f_{\ref{BKMM}}(g, t)\) facial cycles in \(\sigma'\) that cover all vertices in \(R'\). Since each face of \(\sigma'\) contains at most one vertex \(x_V\), we conclude that there are at most \(f_{\ref{BKMM}}(g, t)\) vertices \(x_V\), and thus there are at most \(f_{\ref{BKMM}}(g, t)\) small vortices \((J, \Omega) \in \mathcal{W}\) with \(|\Omega| = 3\) and \((V(J) - \Omega) \cap R \neq \emptyset\). We can also transform the collection of at most \(f_{\ref{BKMM}}(g, t)\) facial cycles in \(\sigma'\) covering \(R'\) into a collection of at most \(f_{\ref{BKMM}}(g, t)\) facial cycles in \(\sigma\) covering \(R \cap V(H_0)\) by replacing any facial cycle containing a vertex \(x_V\) with \(V = (J, \Omega)\) by the facial cycle \(H_0'[\Omega]\).
\end{proof}

Recall that a {\em block} of a graph $G$ is an inclusionwise maximal subgraph of $G$ that is $2$-connected, or an edge, or a vertex. 
An embedding of a graph $G$ in a surface $\surf$ is {\em cellular} if the closure of every face is homeomorphic to a disk. A graph $G$ is {\em cellularly embedded} in $\surf$ if the embedding is cellular. 
 
\begin{theorem}[{Robertson and Vitray~\cite{RV90}, \cite[Proposition 5.5.2]{MT01}}]\label{thm:cycle_faces}
Let $\surf$ be a surface with $\eg(\surf) > 0$ and let $G$ be a graph that is embedded in $\surf$ with facewidth at least $2$. Then there is precisely one block $Q$ of $G$ that contains a noncontractible cycle. Moreover, $Q$ is cellularly embedded in $\surf$ and all its faces are bounded by cycles. Each block $Q'$ of $G$ distinct from $Q$ is a planar subgraph of $G$ contained in the closure of some face of $Q$. Finally, the facewidth of $Q$ is equal to the facewidth of $G$.
\end{theorem}

A corollary of this result is that if $G$ is a $2$-connected graph that is embedded in $\surf$ with facewidth at least $2$, then all the faces of the embedding are bounded by cycles.

\begin{proof}[Proof of \Cref{interesting}]
  Let \(\alpha = \alpha(t)\) be as in \Cref{vortex_free_decomposition}, and let \(h = \max\{ 2\alpha+8, 3t^2\}\).
  Let \(w = f_{\ref{BKMM}}(\alpha, t)\).
  Let \(k=k(t,w,3t^2)\) (as in \Cref{vortex_free_decomposition}), and let \(k'=k+(t-1)\binom{\alpha}2+(3t-4)\alpha+w\).
  Let \((G, R)\) be a rooted graph without a rooted \(K_{2,t}\)-minor, and let \(H\) be a \(k\)-interesting subgraph.
  By \Cref{vortex_free_decomposition}, \(H\) admits a \((w, h+2)\)-good \((\alpha, 0, \alpha)\)-near embedding \((\sigma,H_0,Z,\mathcal{V},\mathcal{W})\) (thus, \(\mathcal{V} = \emptyset\)), and there is a good \((h+2)\)-wall \(W_0'\) in \(H_0'\) such that any transversal set of \(W_0'\) is linked to \(X\) in \(H\) with \(h\) disjoint paths.
  By \Cref{ProperAttachments}, we may assume that all vortices are properly attached, and by \Cref{ensuring_3c}, we may alter the near embedding and the wall so that \(H_0'\) is \(3\)-connected and \(W_0'\) is a good \(h\)-wall, but the facewidth is still at least \(w\), and the small vortices are still properly attached.
  Note that by \Cref{thm:cycle_faces}, the embedding \(\sigma\) is cellular, and thus each face is bounded by a cycle.
  By \Cref{Omega2}, there are at most \(k'-w\) small vortices \((J, \Omega) \in \mathcal{W}\) with \(|\Omega|\le2\) and \((V(J) - \Omega) \cap R \neq \emptyset\).
  By \Cref{BKMMappl}, there are at most \(w\) small vortices \((J, \Omega) \in \mathcal{W}\) with \(|\Omega|=3\) and \((V(J) - \Omega) \cap R \neq \emptyset\), and there is a set of \(w\) facial cycles which cover all vertices in \(V(H_0) \cap R\).

  Let \((J_1, \Omega_1), \ldots, (J_m, \Omega_m)\) denote the small vertices in \(\mathcal{W}\), listed in an order such that for any \(i \in [m]\), if \(i > 2k \ge k+(t+1)\binom{\alpha}2+3(t+1)\alpha + w\), then \((V(J_i) - \Omega_i) \cap R = \emptyset\). We define a star-decomposition \((B_0; B_1, \ldots, B_m)\) by letting \(B_0 = V(H_0) \cup X\), and \(B_i = V(J_i) \cup Z\) for \(i \in [m]\). Now, let \(U = Z \cup (X \setminus V(H_0))\). Then, \(|U| \le |Z| + |X| \le \alpha + k \le k'\), and we have \(H^\#[B_0] - U = H_0'\), so it remains to verify that \(|B_0 \cap B_i| \le k\) for each \(i \in [m]\). We have \(B_0 \cap B_i \subseteq (V(J_i) \cap (V(H_0) \cup X))\cup Z\) and \(|V(J_i) \cap V(H_0)| \le 3\) and \(|Z| \le \alpha\), so it suffices to bound \(|V(J_i) \cap X|\).

  Fix \(i \in [m]\).
  Let \(P_1, \ldots, P_h\) be disjoint paths in \(G\) such that each \(P_j\) is between a vertex \(x_j \in X\) and a vertex \(y_j \in V(W_0')\) such that \(\{y_1, \ldots, y_h\}\) is transversal set in \(W_0'\). Since \(W_0'\) is a good wall, \(J_i - \Omega_i\) contains at most one vertex \(y_i\). The set \(Z \cup \Omega_i\) has at most \(\alpha + 3\) elements, and separates \(V(J_i)\) from \(V(H) \setminus V(J_i)\).
  Since at most \(\alpha+3\) of the paths \(P_1, \ldots, P_h\) intersect \(Z \cup \Omega_i\), we may assume that the paths \(P_{\alpha+4}, \ldots, P_h\) are disjoint from \(Z \cup \Omega_i\), and thus either contained in \(J_i - \Omega_i\), or disjoint from \(V(J_i) \cup \Omega_i\). Furthermore, since \(W_0'\) is a good wall, \(J_i - \Omega_i\) contains at most one vertex \(y_i\), so at most one of the paths \(P_j\) is contained in \(J_i - \Omega_i\).
  Thus, we may assume that the paths \(P_{\alpha+5}, \ldots, P_h\) are disjoint from \(J_i\). In particular, none of the \(\alpha+4\) vertices \(x_{\alpha+5}, \ldots, x_h\) 
  belongs to \(J_i\). Once again, the set \(Z \cup \Omega_i\) has at most \(\alpha+3\) elements and separates \(V(J_i)\) from \(\{x_{\alpha+5}, \ldots, x_h\}\), so by well-linkedness of \(X\), we have
  \(|V(J_i) \cap X| \le \alpha+3\), and thus \(|B_0 \cap B_i| \le 2\alpha+6 \le k\), as required.
\end{proof}

\subsection{Obtaining the docset superprofiles}\label{sec:superprofiles}
The goal of this subsection is to prove \Cref{SpecialDecomposition}, which is our main structural result.
The proof consists of two main steps:
First we modify the tree-decomposition from \Cref{P1-P6} to be $\ell$-special.
Second we explicitly construct a docset superprofile for each bag of the resulting tree-decomposition.

%
For the proof of \Cref{SpecialDecomposition} we need the following results. 

\begin{theorem}[\cite{ringel_1965,bouchet_1978}]
  The Euler genus of a surface that embeds $K_{m,n}$ is at least $\left\lceil\frac{(m-2)(n-2)}{2}\right\rceil$ if $m\ge 2$ and $n\ge 2$. 
\end{theorem}

The above theorem in particular gives a lower bound on the Euler genus of a surface where $K_{3,m}$, $m\in \mathbb{N}$, can be embedded.

\begin{corollary}\label{embeddingK3k}
    Let $m \in \mathbb{N}$ and let $\surf$ be a surface of Euler genus $g:=\eg(\surf)$.
    Then $K_{3,m}$ cannot be embedded in $\surf$ for $m> 2g+2$. 
\end{corollary}

\begin{lemma}\label{interlacing_intervals}
  Let $G$ be a graph embedded in a surface $\surf$ of Euler genus $g$ and let $F$ be a face bounded by a cycle of $G$ in the embedding. 
  Let $a_1,b_1,a_2,b_2,\ldots, a_{m},b_{m}$ be some subset of vertices of $F$ ordered in a cyclic order. 
  Then, if $m>g+1$, there are no two vertex-disjoint connected subgraphs $A,B \subseteq G$ such that $\{a_1,a_2,\ldots,a_{m}\}\subseteq A$ and $\{b_1,b_2,\ldots,b_{m}\}\subseteq B$.
\end{lemma}

\begin{proof}
  Assume to the contrary that we have two vertex-disjoint connected subgraphs  $A,B$ satisfying the above condition.
  We claim that $K_{3,2m}$ can then be embedded in $\surf$, in contradiction to \Cref{embeddingK3k}.
  This can be seen as follows. 
  For any pair of consecutive vertices $a_i$, $b_j$ in the cyclic order given by the face $F$, we add a vertex $c_{i+j}$ ($c_{1}$ if $i=1$ and $j=m$), and edges $(a_i,c_{i+j})$ and $(c_{i+j},b_j)$. 
  We denote this set of $2m$ auxiliary vertices by $C$.
  Place one more auxiliary vertex $v$ inside $F$ and connect it to each of the vertices in $C$. 
  This can be done without introducing edge intersections, since all $a_i$ and $b_i$, $i\in[m]$ are incident to $F$.
  We obtain a model of $K_{3,2m}$ where the vertices in $C$ are the vertices of degree $3$ and $\{v\},V(A),V(B)$ are the branch sets of vertices of degree $2m$ (see \Cref{fig:K3m}), implying that $K_{3,2m}$ can be embedded in $\surf$, as claimed. 
  \begin{figure}
    \centering
    \includegraphics[width=0.45\textwidth]{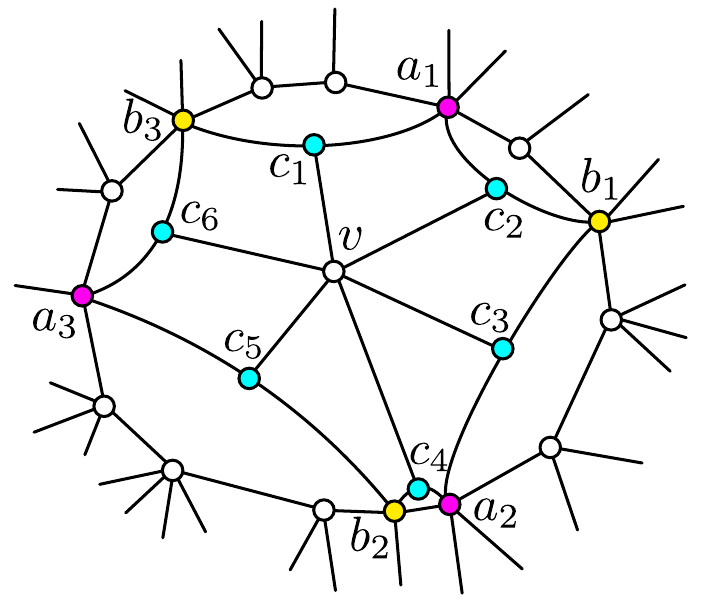}
    \caption{A $K_{3,6}$-minor as in the proof of \Cref{interlacing_intervals}. The auxiliary vertices and edges are drawn within the face $F$ without crossing.}
    \label{fig:K3m}
  \end{figure}
\end{proof}



\begin{proof}[Proof of \Cref{SpecialDecomposition}]
Let $\ell=\ell(t)$ and let $(T',\mathcal{B}')$, where $\mathcal{B}' = \{B'_u : u \in V(T')\}$, be the tree-decomposition we get from \Cref{P1-P6}, that is, a tree-decomposition satisfying the following properties.
\begin{enumerate}
\item the adhesion of $(T',\mathcal{B}')$ is at most \(\ell\), and
\item for every node $u\in V(T')$, all but at most \(\ell\) children of \(u\) are leaves \(v\) with \(B'_v \cap R \subseteq B'_u\), and
\item \label{third_type} for every node $u\in V(T')$, at least one of the following is true
\begin{enumerate}
\item \(|B'_u| \le \ell\), 
\item $u$ is a leaf of $T'$ and $B'_u\cap R \subseteq B'_{u'}$, where $u'$ is the parent of $u$ in $T'$, or 
\item \label{set_Z} there is a set $Z\subseteq B'_u$ such that $|Z|\le \ell$, $G^{\#}[B'_u]-Z$ is $3$-connected, does not contain a rooted $K_{2,t}$-minor with respect to the set of roots $B'_u \cap R - Z$ and has an embedding in a surface of Euler genus at most $\ell$ such that every face is bounded by a cycle of the graph, and all the vertices in $B'_u \cap R - Z$ can be covered using at most $\ell$ facial cycles. 
\end{enumerate}
\end{enumerate}

We can turn such a tree-decomposition into an $\ell$-special tree-decomposition by the following simple procedure. 
We traverse the rooted tree bottom-up and for every node $u\in V(T')$, we add to $B'_u$ all the vertices in the bags $B'_v$ where $v$ is a leaf child of $u$ in $T'$ and $B'_v\cap R\subseteq B'_u$. 
We denote the resulting tree-decomposition by $(T,\mathcal{B})$, where $\mathcal{B} = \{B_u : u \in V(T)\}$. 
Let $u\in V(T)$ be a node of the tree-decomposition $(T,\mathcal{B})$ and let $R_u:=B_u\cap R$ be the set of roots contained in $B_u$. 

If $|R_u|\le\ell$ then we obtain a docset superprofile of $R_u$ by taking all the possible subsets of $R_u$. 
More precisely, $\mathcal{X}_u:=\{R' :  R'\subseteq R_u\}\supseteq \mathcal{P}(G,R_u)$. 
In this case, $|\mathcal{X}_u|\le 2^{\ell}$, which is an upper bound that depends only on $t$ and is independent of the number of vertices in $G$. 


Otherwise, $|R_u| > \ell$, and $B_u$ was obtained from a bag $B'_{u'}$ of type \ref{third_type}.\ref{set_Z} by adding bags of children of $u'$ without additional roots.
Let $Z\subseteq B'_{u'}$ and let $\mathcal{F}$ be the set of faces that cover the elements in $R_u-Z$. 
We obtain a docset superprofile in this case by considering collections of roots in $R_u\cap Z$ together with collections of roots contained in a bounded number of intervals chosen from each $F\in\mathcal{F}$. 
Formally, let $F\in \mathcal{F}$ and let $v_1,v_2,\dots,v_f$ be a cyclic ordering of the vertices on the cycle $C_F$ bounding $F$.
For $i,j\in [f]$, we define an interval $I_{i,j}$ on $C_F$ as the set of consecutive vertices on $C_F$ from $v_i$ to $v_j$ with respect to the cyclic ordering.
That is, $I_{i,j}:=\{v_i,v_{i+1},\dots,v_j\}$ (potentially containing $\dots,v_f,v_1,\dots$), and $I_{i,i-1}:=\emptyset$ for all $i\in[f]$.
Let $X,Y\subseteq V(G)$ be disjoint. 
We call the tuple $\{X,Y\}$ of vertex sets {\em $k$-interlacing} with respect to $F$ if there are at least $k$ pairs $i_1,i_2\in[f]$, where $i_1\neq i_2$ and $v_{i_1}\in X$, $v_{i_2}\in Y$, and $\{v_j:j\in I_{i_1+1,i_2-1}\}\cap(X\cup Y)=\emptyset$.
We first show a statement on the properties of $k$-interlacing sets under partitions.\newline

\begin{claim}
\label{claim:interlacing}
Let $\{X,Y\}$ be $k$-interlacing with respect to some face $F$, and let $p,q\in\N$. 
If we partition $X$ into $X_1,X_2,\dots,X_p$ and $Y$ into $Y_1,Y_2,\dots,Y_q$, then there is some $i\in[p]$ and $j\in[q]$ such that $\{X_i,Y_j\}$ is at least $\left(\frac{k}{pq}\right)$-interlacing with respect to $F$.
\end{claim}

\begin{proof}[Proof of the claim]
Each of the $k$ pairs of vertices from $X,Y$ that certify that $\{X,Y\}$ is $k$-interlacing corresponds to a tuple $(i,j)$ of indices of subsets $X_i$, $Y_j$.
There are at most $pq$ different tuples of this form.
The claim follows by the pigeonhole principle.
\end{proof}



We proceed with the definition of the docset superprofile for the bag $B_u$. 
For each face $F\in\mathcal{F}$, let $R_F:=R_u\cap V(F)$.
Let $\ell':=(\ell+2) \cdot (2^{\ell^3+1} \cdot t \cdot \ell^2)^2 \cdot \ell^4$, and
\begin{equation*}
  \mathcal{X}_u:=\left\{R' :  R'\subseteq R_u,\ \{R'\cap F,(R_u-R')\cap F\}\text{ is not $\ell'$-interlacing w.r.t. }F\ \text{ for any } F\in\mathcal{F}\right\}.
\end{equation*}
Note that we allow all combinations on the set of roots that are in $Z$.
Since the intersection between the roots and each face corresponds to a bounded number of intervals, the size of $\mathcal{X}_u$ is bounded.
We claim that $\mathcal{X}_u\supseteq \mathcal{P}(G,R_u)$.
Arguing by contradiction, suppose this is not the case. 
Then, there is a docset $S$ in $G$ and a face $F\in \mathcal{F}$ such that $\{S\cap R_F,\overline{S}\cap R_F\}$ is $\ell'$-interlacing.


Observe that $(G[B_u], R_u)$ does not contain a rooted $K_{2,t}$-minor, since $(G,R)$ has no rooted $K_{2,t}$-minor. 
We perform three modifications to the graph $G[B_u]$.
Firstly we add each edge $e=xy$ where $x,y\in Z$ and $e$ is not in the graph yet. 
In this way, the graph induced on $Z$ is a clique. 
Secondly, we add each edge $e=xy$ such that $x,y \in B_u\cap B_w$ where $w$ is the parent of $u$ in $T$, and $e$ is not in the graph yet. 
In this way, the graph induced on $B_u\cap B_w$ is also a clique. 
We denote the second set of edges we add as $M$. 
Thirdly, for each child $v$ of $u$ in $T$, we add each edge $e=xy$ such that $x,y \in B_u\cap B_v$, and $e$ is not in the graph yet.
In this way, the graph induced on $B_u\cap B_v$ is a clique for each child $v$.

Let $H$ be the resulting graph. 
Since $G[B_u]$ does not contain a rooted $K_{2,t}$-minor, by \Cref{adding_an_edge} and the facts that $|Z|\le \ell$, $|B_u\cap B_w|\le \ell$, $|B_u\cap B_v|\le \ell$, and there is at most $\ell$ children, $(H, R_u)$ does not contain a rooted $K_{2,t'}$-minor for $t':= 2^{\ell^3}\cdot t$.

Note that the docset $S$ restricted to $H$ is a docset in $H$. 
Let $X$ and $Y$ be the corresponding partition of $V(H)$, such that $H[X]$ and $H[Y]$ are connected.
After the removal of the edges in $M$, the subgraphs $H[X]$ and $H[Y]$ decompose into at most $\ell^2$ connected components.
By \Cref{claim:interlacing} there are sets $X',Y'$ such that $H[X']-M$ and $H[Y']-M$ are connected and $\{X'\cap R_F,Y'\cap R_F\}$ are $((\ell+2) \cdot (2^{\ell^3+1} \cdot t \cdot \ell^2)^2)$-interlacing.
In addition, observe that $H[B'_{u'}]-M-Z$ is a subgraph of $G^{\#}[B'_{u'}]-Z$.
The difference is that we do not add a clique in the adhesion to children of $u'$ that were merged in order to obtain $u$.
Hence we can efficiently find an embedding on the same surface as in \ref{third_type}.\ref{set_Z}.


Consider the collection of connected components of $H[X']-Z$. Let $X_1,X_2,\ldots,X_p$ be the vertex sets those connected components that intersect $R_F$. 
Define $Y_1,Y_2,\ldots,Y_q$ similarly, replacing $X'$ with $Y'$. 

Let $t'':=2^{\ell^3} \cdot t \cdot \ell^2$. 
Assume first that either $p\ge t''$ or $q\ge t''$, say without loss of generality $p\ge t''$. 
In this case we find a model of a rooted $K_{2,t'}$ in $(H, R_u)$ as follows. 
The set $X'\cap Z$ is one of the branch sets of the vertices of degree $t'$. 
The $t'$ branch sets of the vertices of degree $2$ will be modeled by $t'$ well-chosen sets among the $X_i$s, as we explain below. 
Observe first that indeed, $H[X'\cap Z]$ is connected (since $Z$ is a clique in $H$), each $X_i$ contains a root (by definition) and each $X_i$ sends an edge to $X'\cap Z$, since $H[X']$ is connected.
Furthermore, each $X_i$ intersects $C_F$. 
By the connectivity properties of this cycle, each $X_i$ has at least one neighbor in $C_F\cap Y$. Choose one arbitrarily.
Recall that $H[Y]-M$ decomposes into at most $\ell^2$ connected components. 
Thus, by the pigeonhole principle there is one connected component of $H[Y]-M$ that is adjacent to at least $t''/\ell^2=t'$ of the $X_i$s. 
Choosing those $X_i$s for the branch sets of the vertices of degree $2$ and the vertex set of that connected component of $H[Y]-M$ for the last branch set completes the description of our rooted model of $K_{2,t'}$ in $(H, R_u)$, 
This is a contradiction with the fact that $(H, R_u)$ does not contain a rooted $K_{2,t'}$-minor. 

Therefore, $p+q<2t''$. 
Note that $((\ell+2) \cdot (2^{\ell^3+1} \cdot t \cdot \ell^2)^2)/(2t'')^2=\ell+2$.
Hence by \Cref{claim:interlacing}, there are $i\in [p]$ and $j\in [q]$ such that $\{X_i\cap R_F,Y_j\cap R_F\}$ are $\ell+2$-interlacing with respect to $F$.
By \Cref{interlacing_intervals}, this implies a contradiction with the Euler genus of the surface embedding of $G^{\#}[B'_u]-Z$ (if the model uses connectivity through some child, this can be captured by the virtual edges in $G^{\#}[B'_u]-Z$).
\end{proof}

\subsection{Extending to subdivisions}\label{sec:subdivisions}
We extend \Cref{SpecialDecomposition} to subdivisions of rooted $3$-connected graphs $(G,R)$, where roots have degree at least $3$.
Let $t\in\Z_{\ge0}$ be such that $(G,R)$ does not contain a rooted $K_{2,t}$-minor. 
We denote by $(G',R)$ the rooted graph obtained from $(G,R)$ by contracting each subdivided edge to a single edge.
Note that by \Cref{contraction}, $(G',R)$ does not contain a rooted $K_{2,t}$-minor. 
Furthermore, the sets of roots coincide since we never remove a vertex from the set of roots, and never contract an edge between two roots.

\begin{lemma}\label{lem:decomposition}
    Let $\ell\in\Z$. 
    Given an $\ell$-special tree-decomposition $(T',\mathcal{B}')$ of $G'$, where $\mathcal{B}' = \{B'_u : u \in V(T')\}$, we can efficiently compute an $\ell$-special tree-decomposition $(T,\mathcal{B})$ of $G$, where $\mathcal{B} = \{B_u : u \in V(T)\}$, such that $T=T'$ and $R\cap B_u=R\cap B'_u$ for each $u\in V(T)$.
\end{lemma}
\begin{proof}
    Let $F\subseteq E(G')$ denote the edges of $G'$ that correspond to subdivided edges of $G$.
    Each $e\in F$ induces a path in $G$, which we denote by $P(e):=v_0,e_1,v_1,\dots,v_p,e_{p+1},v_{p+1}$, such that each $v_i$ with $i\in[p]$ has degree $2$ in $G$, and $v_0,v_{p+1}$ have degrees at least $3$ in $G$. 

    Let $u_0$ be the root of the decomposition tree $T'$, and let $vv'=e\in F$.
    By property \ref{def:td_edge} of \Cref{def:td}, there is a node $u\in V(T')$, such that $v,v'\in B'_u$.
    By property \ref{def:td_connected} of \Cref{def:td}, the bags of $(T',\mathcal{B}')$ that contain both $v$ and $v'$ correspond to a subtree of $T'$, which we denote by $T'_e$.
    Let $u'$ denote the least common ancestor of $T'_e$ with respect to $u_0$.
    Then we insert the inner nodes $v_i$ for $i\in[p]$ of $P(e)$ into $B'_{u'}$, in order to obtain a collection of subsets of $V(G)$, denoted by $\mathcal{B}$.
    It is easy to see that the resulting pair $(T,\mathcal{B})$ is an $\ell$-special tree-decomposition of $G$ having the desired properties.
\end{proof}
\begin{lemma}\label{lem:profiles}
    Let $R'\subseteq R$ be a subset of the roots.
    Then, $\mathcal{P}(G,R')\subseteq\mathcal{P}(G',R')$.
\end{lemma}
\begin{proof}
    Consider an element $P\in\mathcal{P}(G,R')$, as well as $S\in\mathcal{S}(G)$ a docset of $G$ with $S\cap R'=P$.
    Observe that the restriction of $S$ to $G'$ is a docset of $G'$.
    Thus $P=S\cap R'\in\mathcal{P}(G',R')$.
\end{proof}

\begin{proposition}\label{SpecialDecompositionExtended}
  Let $(G,R)$ be a subdivision of a $3$-connected rooted graph, such that each $v\in R$ has degree at least $3$, and $(G,R)$ has no rooted $K_{2,t}$-minor.
  Then the statement of \Cref{SpecialDecomposition} applies, that is, there is a polynomial-time algorithm that outputs an $f_{\ref{SpecialDecomposition}}(t)$-special tree-decomposition $(T,\mathcal{B})$ of $G$ and a collection $\{\mathcal{X}_u:u\in V(T)\}$ where each $\mathcal{X}_u$ is a docset superprofile of $R\cap B_u$ in $G$ of size polynomial in $|V(G)|$. 
\end{proposition}
\begin{proof}
  By \Cref{SpecialDecomposition}, we obtain an $f_{\ref{SpecialDecomposition}}(t)$-special tree-decomposition $(T, \mathcal{B})$ of $G'$ and a collection $\{\mathcal{X}_u : u \in V(T)\}$ where each $\mathcal{X}_u$ is a docset superprofile of $R_u$ in $G'$ (of size polynomial in $|V(G)|$).
  By \Cref{lem:decomposition}, and \Cref{lem:profiles}, both the tree-decomposition and the superprofiles extend to $G$.
\end{proof}
\section{Discussion} \label{sec:discussion}

In this paper, we initiated the study of integer programs (IPs) on constraint matrices that simultaneously have bounded subdeterminants and are nearly totally unimodular (TU), in the sense that they become TU after the deletion of a constant number of rows and columns. In view of the state of the art \cite{artmann_2017,FJWY25,naegele_2022,naegele_2023}, and also in view of the matroid minors project \cite{geelen_2014,geelen_2015}, this is a natural case of the conjecture that bounding all subdeterminants suffices to guarantee that an IP can be solved in polynomial time. 

We reduced to the case in which the linear constraints are $Ax = b$, $Wx = d$, $\ell \le x \le u$, where $A$ is a TU matrix and $W$ is an integer matrix with a constant number of rows, see \eqref{eqIPequality}. The combinatorial structure of $A$, more specifically the collection of \emph{circuits} of $A$, plays a crucial role in our analysis. We showed that bounding the subdeterminants of the original constraint matrix translates to bounding the weight of every circuit of $A$ with respect to any row of $W$. Moreover, we proved that some optimal solution $\bar{z}$ of the IP can be reached from some integer point $z$, that can be efficiently found through solving the LP relaxation, by adding (or augmenting on) a \emph{bounded} number of circuits of $A$.


We proved that we can reduce in a black-box fashion to the case where $A$ almost $3$-connected. Furthermore, we proved that an optimal augmentation can be found efficiently in the cographic case, that is, when the circuits of $A$ correspond to the minimal cuts of a graph. This leaves open the question of efficiently finding an optimal augmentation in case $A$ is a \emph{general} TU matrix. 

By Seymour's decomposition of regular matroids \cite{seymour_1980}, if $A$ is $4$-connected then there exists a graph $G$ such that the circuits of $A$ correspond either to the circuits of $G$ (graphic case), or to the minimal cuts of $G$ (cographic case). Since we solved the cographic case in the present paper, the graphic case seems to be the second and last main case to consider. We have preliminary observations concerning this. 

First, we expect that, in the graphic case, $G$ admits a tree-decomposition similar to that of \Cref{P1-P6}, replacing rooted $K_{2,t}$ minors with the dual class of obstructions. 
In view of this, the case where $G$ can be embedded in a fixed surface $\surf$ seems once again particularly interesting. If $\surf$ is orientable, then we can reduce the problem to the cographic case, via taking the dual of the directed graph $G$. If $\surf$ is non-orientable, then the dual of $G$ is no longer a directed graph, but rather a bidirected graph. Thus, we are naturally led to consider the case where $A = \begin{bmatrix} B &\identity \end{bmatrix}$ and $B$ is the transpose of a binet matrix. In fact, the ``right'' case to consider next toward the totally $\Delta$-modular IP conjecture after the cographic case might be ``nearly binet or transposed binet'' case. We remark that for the corresponding instances, large vortices are unavoidable in near-embeddings and have to be dealt with in the optimization phase. To achieve this, the ideas of \cite{FJWY25} should be very useful.

Second, we leave open the task of handling $3$-separations of $A$. 
Achieving this in a similar way as we dealt with $2$-separation seems challenging. (The fact that $3$-sums are much more difficult to handle than $1$- and $2$-sums is a recurring theme in recent work using Seymour's decomposition, see for instance~\cite{aprile_2022,artmann_2017}.) 
Despite our efforts, we are not aware of any family of gadgets that could be used to reduce the problem in a black-box fashion to the case where the $3$-separations of $A$ are ``under control''. Instead, we suggest to view $3$-separations of $A$ as small order separations of $\smallmat{A\\ W}$ and develop more systematic ways to handle small order separations in $\smallmat{A\\ W}$, if possible extending also the decomposition technique we developed for graphs forbidding a rooted $K_{2,t}$-minor.


Third, we observe that the graphic case of our problem contains instances of (the optimization version of) the famous \emph{red-blue matching problem} (also known as the \emph{exact matching problem}): given a bipartite graph $G$ whose edges are colored either red or blue, a positive integer $k$, and integer profits on the edges of $G$, find a maximum profit perfect matching of $G$ containing exactly $k$ red edges. The resulting IP reads $\max \{p^\intercal x : A x = \one,\ \zero \le x \le \one,\ \sum_{e\ \mathrm{red}} x(e) = k,\ x \in \Z^{E(G)}\}$, where $A \in \{0,1\}^{V(G) \times E(G)}$ is the incidence matrix of $G$. Bounding the subdeterminants of the constraint matrix of this IP by $\Delta$ (in absolute value) entails asking that every single augmentation changes the number of red edges in a (perfect) matching by at most $\Delta$. (More precisely, the constraint matrix of the IP has all its subdeterminants bounded by $\Delta$ if and only if the weight of every circuit in $G$ augmented with an extra vertex $v_0$ adjacent to all vertices through a zero-weight edge is at most $\Delta$. Rewriting the IP in inequality form allows to get rid of $v_0$, and consider a less restricted set of instances.) There is a randomized polynomial-time algorithm to check whether the IP is feasible, even without assuming that the subdeterminants are bounded, see \cite{mulmuley_1987}. We believe that there is a deterministic polynomial-time algorithm to solve the above IP to optimality, provided that the subdeterminants are bounded.

\section*{Acknowledgements}

We thank Sebastian Wiederrecht for helpful discussions regarding~\cite{TW_FOCS_2022, TW_arxiv_2022}. 
We are grateful to the anonymous reviewers for their careful reading and insightful comments, which have greatly improved the clarity and presentation of this paper. 

Manuel Aprile and Samuel Fiorini acknowledge funding from \emph{Fonds de la Recherche Scientifique} - FNRS through research project BD-DELTA (PDR 20222190, 2021--24). Samuel Fiorini was also funded by \emph{King Baudouin Foundation} through project BD-DELTA-2 (convention 2023-F2150080-233051, 2023--26). Gwenaël Joret and Michał Seweryn acknowledge funding from FNRS (PDR "Product structure of planar graphs").
Stefan Kober and Stefan Weltge were supported by the Deutsche Forschungsgemeinschaft (German Research Foundation) under the projects 277991500/GRK2201 and 451026932, respectively.
Yelena Yuditsky was supported by FNRS as a Postdoctoral Researcher. 

\bibliographystyle{alpha}
\bibliography{../../../references}




\end{document}